\documentclass{article}

\usepackage{amsmath, amssymb, latexsym, amsthm, xspace, color}
\usepackage{ifpdf}
\usepackage{paralist}
\usepackage{graphicx}
\usepackage{subfigure}

\ifpdf
\else
\usepackage{epsfig}
\fi

\numberwithin{equation}{section}

\newcommand{\cx}{{\mathbb C}} 
\newcommand{\half}{{\mathbb H}} 
\newcommand{\HH}{{\mathbb H}} 
\newcommand{\integers}{{\mathbb Z}}
\newcommand{\nats}{{\mathbb N}}
\newcommand{\ratls}{{\mathbb Q}} 
\newcommand{\reals}{{\mathbb R}} 
\newcommand{\proj}{{\mathbb P}}
\newcommand{\zed}{\integers}
\newcommand{\hyp}{{\rm hyp}}

\newcommand{\barratls}{\overline{\ratls}}

\newcommand{\Mobius}{M\"obius\xspace}
\newcommand{\Poincare}{Poincar\'e\xspace}

\newcommand{\Teichmuller}{Teich\-m\"uller\xspace}
\newcommand{\Bezout}{B\'ezout\xspace}

\newcommand{\FF}[1]{\mathbb{F}_{#1}}
\newcommand{\barmoduli}[1][g]{{\overline{\mathcal M}}_{#1}}
\newcommand{\bdry}{\partial}
\newcommand{\bX}{\widehat{X}}

\newcommand{\pderiv}[2]{\frac{\partial #1}{\partial #2}}

\newcommand{\GLtwoR}{{\rm GL}_2(\reals)}
\newcommand{\GLtwoRplus}{{\rm GL}_2^+(\reals)}
\newcommand{\Gm}{\mathbb{G}_m}
\renewcommand{\hat}{\widehat}
\newcommand{\hatH}{\hat{H}}
\renewcommand{\tilde}{\widetilde}
\renewcommand{\Im}{\IM}
\renewcommand{\Re}{\RE}
\newcommand{\RM}[1][\mathcal{O}]{{\cal RM}_{#1}}
\newcommand{\RT}[1][\mathcal{O}]{{\cal RT}_{#1}}
\newcommand{\RA}[1][\mathcal{O}]{{\cal RA}_{#1}}
\newcommand{\barRM}[1][\mathcal{O}]{\overline{\cal RM}_{#1}}
\newcommand{\isom}{\cong}

\newcommand{\moduli}[1][g]{{\mathcal M}_{#1}}
\newcommand{\AVmoduli}[1][g]{{\mathcal A}_{#1}}
\newcommand{\omoduli}[1][g]{{\Omega\mathcal M}_{#1}}

\newcommand{\A}[1][g]{\AVmoduli[#1]}
\newcommand{\EA}[1][\mathcal{O}]{{\cal EA}_{#1}}
\newcommand{\oA}[1][g]{\Omega\AVmoduli[#1]}
\newcommand{\poA}[1][g]{\proj\Omega\AVmoduli[#1]}
\newcommand{\pomoduli}[1][g]{{\proj\Omega\mathcal M}_{#1}}
\newcommand{\pobarmoduli}[1][g]{{\proj\Omega\overline{\mathcal M}}_{#1}}
\newcommand{\potmoduli}[1][g]{\proj\Omega\tilde{\mathcal{M}}_{#1}}

\newcommand{\pom}[1][g]{\proj\Omega{\mathcal M}_{#1}}
\newcommand{\pobarm}[1][g]{\proj\Omega\overline{\mathcal M}_{#1}}

\newcommand{\qtq}[1]{\quad\text{#1}\quad}
\newcommand{\semidirect}{\rtimes}
\newcommand{\SL}{{\mathrm{SL}}}
\newcommand{\SLtwoR}{\SL_2 (\reals)} 

\newcommand{\SOtwoR}{{\rm SO}_2 (\reals)}
\newcommand{\Sp}{\mathrm{Sp}}
\newcommand{\teich}{\mathcal{T}}
\newcommand{\tmoduli}[1][g]{\tilde{\mathcal{M}}_{#1}}
\newcommand{\augteich}{\overline{\teich}}
\newcommand{\tensor}{\otimes}
\newcommand{\homot}{\simeq}
\newcommand{\lmod}{\backslash}
\newcommand{\rmod}{\slash}
\newcommand{\Ord}{\mathcal{O}}
\newcommand{\E}[1][\mathcal{O}]{\mathcal{E}_{#1}}
\newcommand{\barE}[1][\mathcal{O}]{\overline{\mathcal{E}}_{#1}}
\newcommand{\NFQ}{N^F_\ratls}
\newcommand{\odd}{{\rm odd}}

\newcommand{\br}{{\boldsymbol{r}}}
\newcommand{\bs}{{\boldsymbol{s}}}
\newcommand{\bt}{{\boldsymbol{t}}}
\newcommand{\ba}{{\boldsymbol{a}}}
\newcommand{\bb}{{\boldsymbol{b}}}
\newcommand{\bc}{{\boldsymbol{c}}}
\newcommand{\be}{{\boldsymbol{e}}}

\newcommand{\bx}{{\boldsymbol{x}}}
\newcommand{\by}{{\boldsymbol{y}}}
\newcommand{\bz}{{\boldsymbol{z}}}
\newcommand{\bw}{{\boldsymbol{w}}}
\newcommand{\bzero}{\boldsymbol{0}}

\newcommand{\fra}{\mathfrak{a}}

\DeclareMathOperator{\Cl}{Cl}
\DeclareMathOperator{\End}{End}
\DeclareMathOperator{\Ext}{Ext}
\DeclareMathOperator{\Hom}{Hom}
\DeclareMathOperator{\IM}{Im}
\DeclareMathOperator{\RE}{Re}

\DeclareMathOperator{\Aff}{Aff}

\DeclareMathOperator{\Ann}{Ann}

\DeclareMathOperator{\CR}{CR}
\DeclareMathOperator{\Hol}{Hol}

\DeclareMathOperator{\Jac}{Jac}
\DeclareMathOperator{\Ker}{Ker}
\DeclareMathOperator{\Mod}{Mod}

\DeclareMathOperator{\vord}{\rm ord}
\DeclareMathOperator{\Res}{Res}
\DeclareMathOperator{\Spine}{Spine}
\DeclareMathOperator{\Stab}{Stab}
\DeclareMathOperator{\Sym}{Sym}
\DeclareMathOperator{\bS}{{\bf S}}
\DeclareMathOperator{\Tr}{Tr}

\DeclareMathOperator{\Weight}{Weight}
\DeclareMathOperator{\wt}{wt}
\DeclareMathOperator{\Graph}{Graph}
\DeclareMathOperator{\rank}{rank}
\DeclareMathOperator{\length}{length}

\newcommand{\Rec}{{\rm Rec}}
\newcommand{\tRec}{\widetilde{\Rec}}

\newcommand{\SO}{{\rm SO}}

\DeclareMathOperator{\Gal}{Gal}

\DeclareMathOperator{\Span}{Span}
\DeclareMathOperator{\trace}{Tr}
\DeclareMathOperator{\disc}{disc}

\newtheoremstyle{example}{3pt}{3pt}{\upshape}{}{\itshape}{.}{.5em}{}

\newtheorem{theorem}{Theorem}[section] 
\newtheorem{prop}[theorem]{Proposition} 
\newtheorem{cor}[theorem]{Corollary}
\newtheorem{lemma}[theorem]{Lemma}
\newtheorem{algo}[theorem]{Algorithm}
\newtheorem{example}[theorem]{Example}

\newtheorem{conj}[theorem]{Conjecture}

\theoremstyle{example}

\theoremstyle{definition}

\theoremstyle{remark}
\newtheorem*{remark}{Remark}

\makeatletter
\def\blfootnote{\xdef\@thefnmark{}\@footnotetext}

\renewcommand{\l@section}{\@dottedtocline{0}{1.5em}{2.3em}}
\renewcommand{\l@subsection}{\@dottedtocline{1}{3.8em}{3.2em}}
\renewcommand{\l@subsubsection}{\@dottedtocline{2}{7.0em}{4.1em}}

\makeatother

\newcommand{\codim}{{\rm codim}}
\newcommand{\Cone}{{\rm Cone}}

\newcommand{\ord}{\mathcal{O}}

\begin{document}

\bibliographystyle{halpha}

\title{The Deligne-Mumford compactification of the real multiplication locus and \Teichmuller curves
  in genus three}
\author{Matt Bainbridge and Martin M\"oller}
\maketitle

\tableofcontents

\section{Introduction}
\label{sec:introduction}

Each Hilbert modular surface has a beautiful minimal smooth compactification due to Hirzebruch.
Higher-dimensional Hilbert modular varieties instead admit many toroidal compactifications none of
which is clearly the best.  In this paper, we consider canonical compactifications of closely
related varieties, namely the real multiplication locus $\RM$ in the moduli space $\moduli$ of genus
$g$ Riemann surfaces, as well as the locus of eigenforms $\Omega\E$ in the bundle
$\Omega\moduli\to\moduli$ of holomorphic one-forms.

If $g$ is $2$ or $3$, we give a complete
description of the stable curves in the Deligne-Mumford compactification $\barmoduli$ which are in the
boundary of $\RM$, and which stable curves equipped with holomorphic one-forms are in the boundary
of the eigenform locus $\Omega\E$.  If $g>3$, we give strong restrictions on the stable curves in
the boundary of $\RM$.   This allows one to reduce many
difficult questions about Riemann surfaces with real multiplication to concrete problems in
algebraic geometry and number theory by passing to the boundary of $\barmoduli$.  In this paper, we
apply our boundary classification to obtain finiteness results for \Teichmuller curves in
$\moduli[3]$ and noninvariance of the eigenform locus under the  action of $\GLtwoRplus$ on
$\Omega\moduli[3]$. 

\paragraph{Boundary of the eigenform locus.}

We now state a rough version of our calculation of the boundary of the eigenform locus.  See
Theorems~\ref{thm:boundary_nec}, \ref{thm:boundary_suff}, and \ref{thm:explCREQ} for precise
statements.  Consider a totally real cubic field $F$, and let $\mathcal{O}\subset F$ be the ring of
integers (we handle arbitrary orders $\mathcal{O}\subset F$, but stick to the ring of integers here for
simplicity).
The Jacobian of a Riemann surface $X$ has real multiplication by $\mathcal{O}$
roughly if the endomorphism ring of $\Jac(X)$ contains a copy of $\mathcal{O}$ (see
\S\ref{sec:orders-number-fields} for details).  We denote by $\RM\subset\moduli[3]$ the locus of
Riemann surfaces whose Jacobians have real multiplication by $\mathcal{O}$.   Real multiplication
on $\Jac(X)$ determines an  eigenspace decomposition of $\Omega(X)$, the space of holomorphic
one-forms on $X$.  The eigenform locus $\Omega\E\subset\omoduli[3]$ is the locus of pairs $(X,
\omega)$, where $\Jac(X)$ has real multiplication by $\mathcal{O}$, and $\omega\in\Omega(X)$ is an
eigenform.

The bundle $\Omega\moduli\to\moduli$ extends to a bundle $\Omega\barmoduli\to\barmoduli$ whose fiber
over a stable curve $X$ is the space of stable forms on $X$.  A stable form over a stable curve is
a form which is holomorphic, except for possibly simple poles at the nodes, such that the two
residues at a single node are opposite (see \S\ref{sec:RS} for details).  We describe here the closure of
$\Omega\E$ in $\Omega\barmoduli[3]$, which also determines the closure of $\RM$ in $\barmoduli[3]$.

Consider the quadratic map $Q\colon F\to F$, defined by
\begin{equation}
  \label{eq:40}
  Q(x) = N^F_\ratls(x)/x.
\end{equation}
We say that a finite subset $S\subset F$ satisfies the \emph{no-half-space condition} if the
interior of the convex hull of
$Q(S)$ in the $\reals$-span of $Q(S)$ in $F\otimes_\ratls\reals$ contains $0$.

It is well known that every stable curve which is in the closure of the real multiplication locus
$\RM\subset\moduli$ has geometric  genus $0$ or $g$ (we give a proof via complex analysis
in \S\ref{sec:bound-eigenf-locus}). Our description of the closure of the eigenform locus is as
follows.

\begin{theorem}
  \label{thm:charIntro}
  A geometric genus $0$ stable form $(X, \omega)\in\Omega\moduli[3]$ lies in the boundary of the eigenform locus
  $\Omega\E$ if and only  if:
  \begin{itemize}
  \item The set of residues of $\omega$ is a multiple of $\iota(S)$, for some subset $S\subset
    F$, satisfying the no-half-plane condition and spanning an ideal
    $\mathcal{I}\subset\mathcal{O}$, and for some embedding $\iota\colon F\to \reals$.
  \item If $Q(S)$ lies in a $\ratls$-subspace of $F$, then an explicit additional equation, involving
    cross-ratios of the nodes of $X$, is satisfied.
  \end{itemize}
\end{theorem}

\begin{remark}
  The more precise version of this theorem, which we state in \S\ref{sec:bound-eigenf-locus},
  gives a necessary condition which holds more generally in any genus.  In \S\ref{sec:suffg3}, we
  show that this condition is sufficient in genus three.  In fact, it is sufficient also in genus two,
  but we ignore this case as the boundary of the eigenform locus was previously calculated in the
  genus two case in 
  \cite{bainbridge07}.  The higher genus cases are more difficult, as the Torelli map
  $\moduli\to\AVmoduli$ is no longer dominant.
\end{remark}

The boundary of $\E := \proj\Omega\E$ has a stratification into topological types, where two
stable forms are of the same topological type if there is a homeomorphism between them which
preserves residues up to constant multiple.  We may encode a topological type by a directed graph
with the edges weighted by elements of an ideal $\mathcal{I}\subset \mathcal{O}$.  Vertices
represent irreducible components, edges represent nodes, and weights represent residues.  The
corresponding \emph{boundary stratum} of $\E$ is a product of moduli spaces
$\moduli[0,n]$, or a subvariety thereof.  The possible topological types arising in the boundary of
$\RM$ are shown in Figure~\ref{fig:graphs}.  In Appendix~\,\ref{sec:bdcounting}, we give an algorithm
for enumerating all boundary strata of $\E$ associated to a given ideal $\mathcal{I}$.  In
Figure~\ref{cap:countbdry}, we tabulate the number of two-dimensional boundary strata for many
different fields.

An important special case is boundary strata parameterizing irreducible stable curves, otherwise
known as \emph{trinodal curves}.  Consider a basis $\br = (r_1, r_2, r_3)$ of an ideal
$\mathcal{I}\subset\mathcal{O}$.  We say that $\br$ is an \emph{admissible basis} of $\mathcal{I}$
if the $r_i$ satisfy the no-half-space condition.   Let
$\mathcal{S}_\br^\iota\subset\proj\Omega\barmoduli[3]$ be the locus of trinodal forms having residues
$(\pm\iota(r_1), \pm\iota(r_2), \pm\iota(r_3))$.  Since a trinodal curve may be represented by $6$
points in $\proj^1$ identified in pairs, we may identify $\mathcal{S}_\br^\iota$ with the moduli
space $\moduli[0,6]$ of such points.  Suppose $\br$ is admissible.  As three points in $\reals^3$
whose convex hull contains $0$ must be contained in a subspace, we are in the second case of
Theorem~\ref{thm:charIntro}, so $\E\cap\mathcal{S}_\br^\iota$ is cut out by a single
polynomial equation on $\mathcal{S}_\br^\iota\isom\moduli[0,6]$.  We see in
Theorem~\ref{thm:explCREQ} that this equation is
\begin{equation}
  \label{eq:39}
  R_1^{a_1}R_2^{a_2}R_3^{a_3} =1,
\end{equation}
where $R_i\colon\moduli[0,6]\to\cx^*$ are certain cross-ratios of four points and the $a_i$ are
integers determined explicitly by the $r_i$.

\paragraph{Intersecting flats in $\SL_3(\zed)\lmod\SL_3(\reals)\rmod\SO_3(\reals)$.}

In \S\ref{sec:irrstrata}, we show that the notion of an admissible basis of a lattice in a totally
real cubic number field is equivalent to a second condition on bases of totally real number fields,
which we call \emph{rationality} and \emph{positivity}.  Namely, a basis $r_1, \ldots, r_g$ of $F$
is rational and positive if
\begin{equation*}
  \frac{r_i}{s_i}/\frac{r_j}{s_j} \in \ratls^+ \quad\text{for all $i\neq j$,}
\end{equation*}
where $s_1, \ldots, s_g$ is the dual basis of $F$ with respect to the trace pairing.

There is a classical correspondence between ideal classes in totally real degree $g$ number fields
and compact flats in the locally symmetric space $X_g =
\SL_g(\zed)\lmod\SL_g(\reals)\rmod\SO_g(\reals)$, the moduli space of lattices in $\reals^g$.  Given
an lattice $\mathcal{I}$ in a totally real number field $F$, let $U(\mathcal{I})\subset F^*$ be the
group of totally positive units preserving $\mathcal{I}$, embedded in the group
$D\subset\SL_g(\reals)$ of positive diagonal matrices via the $g$ real embeddings of $F$.  There is
an isometric immersion $p_\mathcal{I}$ of the flat torus $T(\mathcal{I}) = U(\mathcal{I})\lmod D$
into $X_g$ arising from the right action of $D$ on $\SL_g(\zed)\lmod\SL_g(\reals)$.   Let
$\Rec\subset X_g$ be the locus of lattices in $\reals^g$ which have an orthogonal basis.  $\Rec$ is
a closed, but not compact, $(g-1)$-dimensional flat.  In \S\ref{sec:irrstrata}, we show that
rational and positive bases of lattices in number fields correspond to intersections of the
corresponding compact flat with  $\Rec$.

\begin{theorem}
  \label{thm:intersecting_flats}
  Given an lattice $\mathcal{I}$ in a totally real number field, there is a natural bijection between the
  set $p_\mathcal{I}^{-1}(\Rec)$ and the set of rational and positive bases of $\mathcal{I}$ up to
  multiplication by units, changing signs, and reordering. 
\end{theorem}

Theorems~\ref{thm:charIntro} and \ref{thm:intersecting_flats} together imply that there is a natural
bijection boundary strata of eigenform loci $\E\subset\proj\Omega\moduli[3]$ and intersection
points of compact flats in $X_3$ with the distinguished flat $\Rec$.  Note that $X_3$ is
$5$-dimensional, while each flat in $X_3$ is at most $2$-dimensional, so one would not expect many
intersections between these flats.  Nevertheless, we show in \S\ref{sec:exist_adm} that the ring of
integers in each 
totally real cubic field has some ideal which has an admissible basis.  In fact, the computations
described in Appendix~\ref{sec:bdcounting} suggest that most lattices in cubic fields have many
admissible bases, although there are also examples of lattices which have none.  It would be an
interesting problem to study the asymptotics of counting these bases.

\paragraph{Algebraically primitive \Teichmuller curves.}

There is an important action of $\GLtwoRplus$ on $\omoduli$, the study of which has many
applications to the dynamics of billiards in polygons and translation flows.  A major open problem
is the classification of $\GLtwoRplus$-orbit-closures.  In genus two, this was solved by McMullen in
\cite{mcmullenabel}, while next to nothing is known for higher genera.

Very rarely, a form $(X, \omega)$ has a $\GLtwoRplus$-stabilizer which is a lattice in $\SLtwoR$.
In that case, the $\GLtwoRplus$-orbit of $(X, \omega)$ projects to an algebraic curve in $\moduli$
which is isometrically immersed with respect to the \Teichmuller metric.  Such a curve in $\moduli$
is called a \emph{\Teichmuller curve}.  A \Teichmuller curve $C$ is uniformized by a Fuchsian group $\Gamma$,
called the Veech group of $C$.   The field $F$ generated by the traces of elements in $\Gamma$ is called the
{\em trace field} of $C$. The trace field is a totally real field of degree at most $g$.  See \S\ref{sec:apTc} for basic definitions around
\Teichmuller curves and the $\GLtwoRplus$-action.

Our main motivation for this work was the problem of classifying \emph{algebraically primitive}
\Teichmuller curves in $\moduli$, that is \Teichmuller curves whose trace field has degree $g$.
Every algebraically primitive \Teichmuller curve lies in $\RM$ for some order $\mathcal{O}$ in its
trace field by \cite{moeller06}, and every \Teichmuller curve has a cusp, so
Theorem~\ref{thm:charIntro} allows one to approach the classification of \Teichmuller curves by
studying the possible stable curves which are limits of their cusps.

In $\Omega\moduli[2]$, each eigenform locus $\Omega\E$ is $\GLtwoRplus$-invariant and contains one
or two \Teichmuller curves (see \cite{mcmullenbild,mcmullenspin}).  These \Teichmuller curves lie in
the stratum $\Omega\moduli[2](2)$ (where we write $\Omega\moduli[g](n_1, \ldots,
n_k)\subset\Omega\moduli$ for the stratum of forms having zeros of order $n_1, \ldots, n_k$).  These
\Teichmuller curves were discovered independently by Calta in \cite{calta}.

A major obstacle to the existence of algebraically primitive \Teichmuller curves in higher genus is
that the eigenform loci are no longer $\GLtwoRplus$-invariant.  McMullen showed in
\cite{mcmullenbild} that $\Omega\E$ is not $\GLtwoRplus$-invariant for $\mathcal{O}$ the ring of
integers in $\ratls(\cos(2 \pi/7))$.  We prove in \S\ref{sec:non-inv} the following stronger
non-invariance statement

\begin{theorem} \label{thm:noninvIntro}
  The eigenform locus $\Omega\E$ is not invariant for $\mathcal{O}$ the ring of integers in any
  totally real cubic field.
\end{theorem}

In contrast to the situation in $\moduli[2]$, we give in this paper strong evidence for the
following conjecture.

\begin{conj}
  \label{conj:intro}
  There are only finitely many algebraically primitive \Teichmuller curves in $\moduli[3]$.  
\end{conj}

In \S\ref{sec:finiteness}, we prove the following instance of this conjecture.

\begin{theorem}
  There are only finitely many algebraically primitive \Teichmuller
  curves generated by a form in the stratum $\omoduli[3](3,1)$.
\end{theorem}

The proof uses the cross-ratio equation \eqref{eq:39} together with a torsion condition from
\cite{moeller} which gives strong restrictions on \Teichmuller curves generated by forms with more
than one zero.  This torsion condition was used previously in \cite{mcmullentor} to show that there is a
unique primitive \Teichmuller curve in $\Omega\moduli[2](1,1)$ and in \cite{Mo08} to show finiteness
of algebraically primitive \Teichmuller curves in the hyperelliptic components
$\Omega\moduli[g](g-1,g-1)^{\rm hyp}$ of $\Omega\moduli[g](g-1,g-1)$.  Similar ideas should
establish finiteness in the strata of $\Omega\moduli[3]$ with more than two zeros.   More ideas are
needed in the strata
$\Omega\moduli[3](4)$ and the component $\Omega\moduli[3](2,2)^{\rm odd}$ of
$\Omega\moduli[3](2,2)$, as the torsion condition gives no information (in
$\Omega\moduli[3](2,2)^{\rm \odd}$ due to the presence of hyperelliptic curves). 

While we cannot rule out infinitely many algebraically primitive \Teichmuller curves in the stratum
$\Omega\moduli[3](4)$, Theorem~\ref{thm:charIntro} gives an efficient algorithm for searching any
given eigenform locus $\Omega\E$ for \Teichmuller curves in this stratum.  Given an order
$\mathcal{O}$, first one lists all admissible bases of ideals in $\mathcal{O}$ as described in
Appendix~\ref{sec:bdcounting}.  For each admissible basis, there are a finite number of irreducible
stable forms having these residues and a fourfold zero.  One then lists these possible stable forms
and then checks each to see if the cross-ratio equation~\eqref{eq:39} holds.  If it never holds,
then there are no possible cusps of \Teichmuller curves in $\Omega\moduli[3](4)\cap\Omega\E$, so
there are no \Teichmuller curves.

Due to numerical
difficulties with the odd component, we have only applied this algorithm to the hyperelliptic
component $\Omega\moduli[3](4)^{\rm hyp}$.  The algorithm recovers the one known example in this
stratum, Veech's $7$-gon curve, contained in $\Omega\E$ for $\mathcal{O}$ the ring of integers in the
unique cubic field of discriminant $49$; it has ruled out algebraically primitive \Teichmuller
curves in $\Omega\moduli[3](4)^{\rm hyp}$ for every other eigenform locus it has considered.

\begin{theorem} \label{thm:numevidence}
  Except for Veech's $7$-gon curve
  there are no algebraically primitive \Teichmuller curves generated by a form in
  $\Omega\E\cap\Omega\moduli[3](4)^{\rm hyp}$ for $\mathcal{O}$ the ring of integers in any of the
  $1778$ totally real cubic fields of discriminant less than $40000$.
\end{theorem}
\par
We discuss the algorithm on which this theorem is based in \S\ref{sec:finconj}.
We also give in this section some further evidence for Conjecture~\ref{conj:intro} in
$\Omega\moduli[3](4)^{\rm hyp}$, that an infinite sequence of algebraically primitive \Teichmuller
curves in this stratum would have to satisfy some unlikely arithmetic restrictions on the widths of
cylinders in periodic directions.

\par
For completeness we mention that there is no hope of proving a finiteness theorem
for algebraically primitive \Teichmuller curves in $\barmoduli$ without bounding $g$. 
Already Veech's fundamental paper \cite{veech89} and also \cite{ward} and \cite{bm} contain infinitely
many algebraically primitive \Teichmuller curves for growing genus $g$.

\paragraph{The eigenform locus is generic.}

A rough dimension count leads one to expect Conjecture\,\ref{conj:intro} to hold for the stratum
$\Omega\moduli[3](4)$, as the expected dimension of $\E\cap\proj\Omega\moduli[3](4)$ is $0$, which
is too small to contain a \Teichmuller curve.  
On the other hand, if the eigenform locus $\Omega\E\subset\Omega\moduli[3]$ is contained in some
stratum besides the 
generic one $\Omega\moduli[3](1,1,1,1)$, one would expect this intersection to be positive
dimensional.  This would be a source of possible \Teichmuller curves.  In
\S\ref{sec:intersecting}, we prove that the eigenform locus is indeed generic.

\begin{theorem}
  For any order $\mathcal{O}$ in a totally real cubic field, each component of the eigenform locus
  $\Omega\E$ lies generically in $\Omega\moduli[3](1,1,1,1)$.
\end{theorem}

The proof uses Theorem~\ref{thm:charIntro} to construct a stable curve in the boundary of $\Omega\E$
where each irreducible component is a thrice-punctured sphere.  A limiting eigenform on this curve
must have a simple zero in each component.

\paragraph{Primitive  but not algebraically primitive \Teichmuller curves.}

From a \Teichmuller curve in $\moduli[g]$, one can construct many \Teichmuller curves in higher genus
moduli spaces by a branched covering construction.  A \Teichmuller curve is \emph{primitive} if it
does not arise from one in lower genus via this construction.  Every algebraically primitive
\Teichmuller curve is primitive, but the converse does not hold.
In $\moduli[3]$, McMullen exhibited in \cite{mcmullenprym} infinitely many primitive \Teichmuller
curves with quadratic trace field.  These curves lie in the intersection of $\Omega\moduli[3](4)$
with the locus of \emph{Prym eigenforms}, that is,
forms $(X, \omega)$ with an involution $i\colon X\to X$ such that the $-1$ part of $\Jac(X)$ is an
Abelian surface with real multiplication having $\omega$ as an eigenform.  It is not known whether
all primitive \Teichmuller curves in $\moduli[3] $ with quadratic trace fields arise from this Prym
construction. 

Our approach to  classifying algebraically primitive \Teichmuller curves could also be applied to
the classification of (say) primitive \Teichmuller curves in $\moduli[3]$ with quadratic trace field.
Given a positive integer $d$ and an order $\mathcal{O}$ in a real quadratic field $F$, there is the
locus $\E(d)\subset\pomoduli[3]$ of forms $(X, \omega)$ such that there exists a degree $d$ map of
$X$ onto an elliptic curve $E$ with the kernel of the induced map $\Jac(X)\to E$ having real
multiplication by $\mathcal{O}$ with $\omega$ as an eigenform.  The locus $\E(d)$ is
three-dimensional, and $\E(2)$ coincides with McMullen's Prym eigenform locus.  \Teichmuller curves
in $\moduli[3]$ having quadratic trace field must be generated by a form in some $\E(d)$.  There is a
classification of the geometric genus zero forms in the boundary of $\E(d)$, similar to that of
Theorem~\ref{thm:charIntro}, with the map $Q$ replaced by a quadratic map
\begin{equation*}
  Q\colon F\oplus \ratls\to F\oplus\ratls.
\end{equation*}Each boundary stratum of $\E(d)$ parameterizing trinodal curves is again a subvariety
of $\moduli[0,6]$ cut out by an equation in cross-ratios similar to \eqref{eq:39}.

Since the cross-ratio equation \eqref{eq:39} was responsible for ruling out algebraically primitive
\Teichmuller curves in $\Omega\moduli[3](4)$, one might wonder why its analogue  does not also rule out
McMullen's \Teichmuller curves in $\E(2)$.  The difference is that the cross-ratio equation cutting
out the trinodal boundary strata of $\E(2)$ no longer depends on the associated residues $r_i\in F$
as in \eqref{eq:39}.  Moreover, each such boundary stratum contains  canonical forms having a
fourfold zero, as opposed to the algebraically primitive case where these forms almost never exist.
We hope to provide the details of this discussion in a future paper.

\paragraph{Towards the proof of Theorem~\ref{thm:charIntro}.}

We conclude by summarizing the proof of Theorem~\ref{thm:charIntro}.  For simplicity, we continue to
assume that $\mathcal{O}$ is a maximal order.  The reader may also wish to ignore the case of
nonmaximal orders on a first reading.

The real multiplication locus $\RM\subset\moduli$ (or more precisely, its lift to the \Teichmuller
space) is cut out by certain linear combinations of
period matrices.  To better understand the equations which cut out the real multiplication locus, in
\S\ref{sec:period-matrices} we give a coordinate-free description of period matrices.  Given an
Abelian group $L$, we define a cover $\moduli(L)\to\moduli$, the space of Riemann surfaces $X$
equipped with a \emph{Lagrangian marking}, that is, an isomorphism of $L$ onto a Lagrangian subspace
of $H_1(X; \zed)$.  We define a homomorphism
\begin{equation*}
  \Psi\colon \bS_\zed(\Hom_\zed(L, \zed))\to\Hol^*\moduli(L),
\end{equation*}
where $\bS_\zed(\cdot)$ denotes the symmetric square, and $\Hol^*\moduli(L)$ is the group of nowhere
vanishing holomorphic functions on $\moduli(L)$.  Each function $\Psi(a)$ is a product of exponentials of
entries of period matrices.  There is a Deligne-Mumford compactification
$\barmoduli(L)$ of $\moduli(L)$ with a boundary divisor $D_\gamma$ for each $\gamma\in L$,
consisting of stable curves where a curve homologous to $\gamma$ has been pinched.  In
Theorem~\ref{thm:meromorphic} we show that each $\Psi(a)$ is meromorphic on $\barmoduli(L)$ with
order of vanishing
$$\vord_{D_\gamma}\Psi(a) = \langle a, \gamma\otimes\gamma\rangle$$
along $D_\gamma$.

Cusps of the real multiplication locus correspond to ideal classes in $\mathcal{O}$ (or extensions
of ideal classes if $\mathcal{O}$ is nonmaximal).  Given an ideal $\mathcal{I}\subset\mathcal{O}$,
we define in \S\ref{sec:bound-eigenf-locus} a real multiplication locus
$\RM(\mathcal{I})\subset\moduli[3](\mathcal{I})$, covering $\RM\subset\moduli[3]$, of surfaces which
have real multiplication in a way which is compatible with the Lagrangian marking by $\mathcal{I}$.
The closure of $\RM(\mathcal{I})$ in $\barmoduli[3](\mathcal{I})$ covers the closure of the cusp of
$\RM$ corresponding to $\mathcal{I}$, so it suffices to compute the closure in
$\barmoduli[3](\mathcal{I})$.  In \S\ref{sec:bound-eigenf-locus}, we construct a rank $3$ subgroup
$\Gamma$ of $\bS_\zed(\Hom(\mathcal{I}, \zed)) \isom \bS_\zed(\mathcal{I}^\vee)$ (where
$\mathcal{I}^\vee\subset F$ is the inverse different of $\mathcal{I}$) such that $\RM(\mathcal{I})$
is cut out by the equations
\begin{equation}
  \label{eq:42}
  \Psi(a) =1
\end{equation}
for all $a\in \Gamma$.  The proof
of Theorem~\ref{thm:admiss-nonhalf} yields an identification of $\Gamma$ with a lattice in $F$ with
the property that for each $a\in \Gamma$ and $t\in \mathcal{I}$, the order of vanishing of $\Psi(a)$
along the divisor $D_t\subset\barmoduli(\mathcal{I})$ is
\begin{equation}
  \label{eq:41}
  \vord_{D_t}\Psi(a) = \langle a, Q(t)\rangle  
\end{equation}
with the pairing the trace pairing on $F$ and $Q(t)$ as in \eqref{eq:40}.

Now suppose that $\mathcal{S}\subset\barmoduli(\mathcal{I})$ is a boundary stratum which is the
intersection of the divisors $D_{t_i}$ for $t_1, \ldots, t_n\in \mathcal{I}$, and suppose that the
$t_i$ do not satisfy the no-half-space condition.  This means that we can find a vector $a\in F$
such that $\langle a, Q(t_i)\rangle \geq 0$ for each $t_i$ with strict inequality for at least
one.  Multiplying $a$ by a sufficiently large integer, we may assume $a\in\Gamma$.  From
\eqref{eq:42} we see that $\Psi(a)\equiv 1$ on $\RM(\mathcal{I})$, and from \eqref{eq:41} we see
that $\Psi(a)\equiv 0$ on $\mathcal{S}$.  It follows that $\barRM(\mathcal{I})\cap\mathcal{S} =
\emptyset$, from which we conclude the first part of Theorem~\ref{thm:charIntro}.

If the $Q(t_i)$ lie in a subspace of $F$, then we may choose $a\in \Gamma$ to be orthogonal to each
$Q(t_i)$.  By \eqref{eq:41}, the function $\Psi(a)$ is nonzero and holomorphic on $\mathcal{S}$.
The equation $\Psi(a)=1$ restricted to $\mathcal{S}$ cuts out a codimension-one subvariety of
$\mathcal{S}$, which yields the second part of Theorem~\ref{thm:charIntro}.  In the case where
$\mathcal{S}$ parameterizes trinodal curves, the equation $\Psi(a)=1$ is exactly the cross-ratio
equation~\eqref{eq:39}.  This concludes the necessity of the conditions of
Theorem~\ref{thm:charIntro}.

To obtain sufficiency of these conditions, in \S\ref{sec:suffg3} we show that one can often define,
using the functions $\Psi(a)$,
local coordinates from  a neighborhood of a boundary stratum $\mathcal{S}$ in $\barmoduli(L)$ into
$(\cx^*)^m\times\cx^n$.  In these coordinates, $\mathcal{S}$ is $(\cx^*)^m\times\{\bzero\}$, and the real
multiplication locus $\RM(\mathcal{I})$ is a subtorus of $(\cx^*)^{m+n}$.  The computation of the
boundary of the real multiplication locus is thus reduced to the computation of the closure of an
algebraic torus in $(\cx^*)^{m+n}$, which is done in Theorem~\ref{thm:torus_closure}.

\paragraph{Hilbert modular varities and the locus of real multiplication.}

We conclude with a discussion of the relation between Hilbert modular varieties and the real
multiplication locus. In several textbooks (e.g.\ \cite{Freitag}) Hilbert modular varieties are
defined as the quotients $\HH^g/\Gamma$, where $\Gamma = \SL(\Ord \oplus
\Ord^\vee)\isom\SL_2(\mathcal{O})$ for some order $\Ord \subset F$, or even more restrictively for
$\Ord$ the ring of integers \cite{Goren}.  There is a natural map from $\HH^g/\Gamma$ to the moduli
space of Abelian varieties whose image is a component of the locus of Abelian varieties with real
multiplication by $\Ord$.  In Appendix~\ref{sec:HMvsRM}, we provide an example showing that the real
multiplication locus need not be connected, so it is in general not the image of $\HH^g/\Gamma$.
This phenomenon is surely known to experts but is often swept under the rug.  If one restricts to
quadratic fields (as in \cite{vandergeer88}) or to maximal orders (as in \cite{Goren}) this
phenomenon disappears.

In this paper, we regard a Hilbert modular variety more generally as a quotient $\HH^g/\Gamma'$ for
any $\Gamma'$ commensurable with $\SL_2(\mathcal{O})$.  With this more general definition, the locus
$\RA\subset\AVmoduli$ of Abelian varieties with real multiplication by $\mathcal{O}$ is parametrized
by a union $X_\mathcal{O}$ of Hilbert modular varieties.

The eigenform loci $\E\subset\pomoduli$ which we compactify are closely related to the Hilbert
modular varieties $X_\mathcal{O}$.  In genus two, $\E$ is isomorphic to $X_\mathcal{O}$, while in genus three,
$\E$ is a (degree-one) branched cover of $X_\mathcal{O}$.  The real multiplication locus
$\RM\subset\moduli$ is a quotient of $\E$ by an action of the Galois group.   See
\S\ref{sec:orders-number-fields} for details on Hilbert modular varieties and the various real
multiplication loci.

\paragraph{Acknowledgments.} The authors thank Gerd Faltings, Pascal Hubert,  Curt McMullen and 
Don Zagier for providing useful ideas and arguments. The authors 
thank the MPIM Bonn for supporting the research of the second named author and providing
both authors an excellent working atmosphere.

\paragraph{Notation.}

Throughout the paper, $F$ will denote a totally real number field, $\mathcal{O}$ and order in $F$,
and $\mathcal{I}\subset F$ a lattice whose coefficient ring contains $\mathcal{O}$.

Given an $R$-module $M$, we write $\Sym_R(M)$ for the submodule of $M \otimes_R M$ fixed by the
involution $\theta(x\otimes y) = y\otimes x$.  We write $\bS_R(M)$ for the quotient of $M \otimes_R
M$ by the submodule generated by the relations $\theta(x) - x$.

Given a bilinear pairing $\langle \, , \, \rangle\colon M\times N \to R$, we write
$\Hom_R^+(M,N)$ and $\Hom_R^-(M,N)$ for the self-adjoint and anti-self-adjoint maps from $M$ to $N$.

We write $\Delta_r$ for the disk of radius $r$ about the origin in $\cx$; we write $\Delta$ for the
unit disk, and $\Delta^*$ for the unit disk with the origin removed.

\section{Orders, real multiplication, and Hilbert modular varieties}
\label{sec:orders-number-fields}

In this section, we discuss necessary background material on orders in number fields, Abelian
varieties with real multiplication, and their various moduli spaces.

\paragraph{Orders.}

Consider a number field $F$ of degree $d$.  A \emph{lattice in $F$} (also called \emph{full module})
is a subgroup of the additive group of $F$ isomorphic to a rank $d$ free Abelian group.  An
\emph{order} in $F$ is a lattice which is also a subring of $F$ containing the identity
element.  The ring of integers in $F$ is the unique maximal order.

Given a lattice $\mathcal{I}$ in $F$, the \emph{coefficient ring} of $\mathcal{I}$ is the order
\begin{equation*}
  \mathcal{O}_\mathcal{I}=\{a \in F : ax\in M \text{ for all } x\in M\}.
\end{equation*}
We will sometimes write $\mathcal{O}_\mathcal{I}$ for the coefficient ring of $\mathcal{I}$.

Lattices in finite dimensional vector spaces over $F$ and their coefficient rings are defined similarly.

\paragraph{Ideal classes.}

Two lattices  $\mathcal{I}$ and $\mathcal{I}'$ in $F$ are \emph{similar} if $\mathcal{I} =
\alpha\mathcal{I}'$ for some $\alpha\in F$.  An \emph{ideal class} is an equivalence class of this
relation.  Given an order $\mathcal{O}$ the set $\Cl(\mathcal{O})$ of ideal classes of lattices with
coefficient ring $\mathcal{O}$ is a finite set (see \cite{BoShaf}).  If $\mathcal{O}$ is a maximal
order, $\Cl(\mathcal{O})$ is the ideal class group of $\mathcal{O}$.

\paragraph{Modules over orders.}

Let $\mathcal{O}$ be an order in a number field $F$ and $M$ a module over $\mathcal{O}$.  The
\emph{rank} of $M$ is the dimension of $M\tensor \ratls$ as a vector space over $F$.  We say $M$ is
\emph{proper} if the $\mathcal{O}$-module structure on $M$ does not extend to a larger order in $F$.

Every torsion-free, rank-one $\mathcal{O}$-module $M$ is isomorphic to a \emph{fractional ideal}
of $\mathcal{O}$, that is, a lattice in $F$ whose coefficient ring contains $\mathcal{O}$.

A \emph{symplectic $\mathcal{O}$-module} is a torsion-free $\mathcal{O}$-module $M$
together with a unimodular symplectic form $\langle\;, \;\rangle\colon M\times M\to\zed$ which is
compatible with the $\mathcal{O}$-module structure in the sense that
$$\langle \lambda x, y\rangle = \langle x, \lambda y\rangle$$
for all $\lambda\in\mathcal{O}$ and $x,y\in M$.

We equip $F^2$ with the symplectic pairing
\begin{equation}
  \label{eq:sympl}
  \langle (\alpha_1, \beta_1), (\alpha_2, \beta_2)\rangle = \Tr(\alpha_1\beta_2 - \alpha_2\beta_1).
\end{equation}
Every rank-two symplectic $\mathcal{O}$-module is isomorphic to a lattice $L$ in $F^2$ with coefficient
ring contains $\mathcal{O}$ such that the symplectic form on $F$ induces a unimodular symplectic paring
$L\times L\to\zed$.

\paragraph{Inverse different.}

Given a lattice $\mathcal{I}\subset F$ with coefficient ring $\mathcal{O}$, the \emph{inverse
  different} of $\mathcal{I}$ is the lattice
\begin{equation*}
  \mathcal{I}^\vee=\{x \in F : \Tr(xy)\in\zed \text{ for all } y\in M\}.
\end{equation*}
 $\mathcal{I}^\vee$ and $\mathcal{I}$ have the same coefficient rings.  The trace pairing induces an
 $\mathcal{O}$-module isomorphism $\mathcal{I}^\vee\to \Hom(\mathcal{I}, \zed)$.

The sum $\mathcal{I}\oplus\mathcal{I}^\vee$ is a  symplectic $\mathcal{O}$-module with the canonical
symplectic form \eqref{eq:sympl}.

\paragraph{Symplectic Extensions.}

We now discuss the classification of certain extensions of lattices in number fields.  This will be
important in the discussion of cusps of Hilbert modular varieties below.

Let $\mathcal{I}$ be a lattice in a number field $F$ with coefficient ring $\mathcal{O}_\mathcal{I}$.
An extension of $\mathcal{I}^\vee$ by $\mathcal{I}$ over an order $\mathcal{O}\subset\mathcal{O}_\mathcal{I}$
is an exact sequence of $\mathcal{O}$-modules,
\begin{equation*}
  0 \to \mathcal{I} \to M \to \mathcal{I}^\vee\to 0,
\end{equation*}
with $M$ a proper $\mathcal{O}$-module.
Given such an extension, a $\zed$-module splitting $s\colon \mathcal{I}^\vee \to M$ determines a
$\zed$-module isomorphism $\mathcal{I}\oplus\mathcal{I}^\vee\to M$.  The module $M$ inherits the
symplectic form \eqref{eq:sympl}, which does not depend on the choice of the splitting $s$.  We say that
this is a \emph{symplectic extension} if the symplectic form is compatible with the
$\mathcal{O}$-module structure of $M$.

Let $E(\mathcal{I})$ be the set of all symplectic extensions of $\mathcal{I}^\vee$ by $\mathcal{I}$
over any order $\mathcal{O}\subset \mathcal{O}_\mathcal{I}$ up to isomorphisms of exact sequences which are the
identity on $\mathcal{I}$ and $\mathcal{I}^\vee$.  We give $E(\mathcal{I})$ the usual Abelian group
structure: given two symplectic extensions,
\begin{equation*}
  0 \to \mathcal{I} \xrightarrow{\iota_i} M_i \xrightarrow{\pi_i}\mathcal{I}^\vee\to 0,
\end{equation*}
define $\pi\colon M_1 \oplus M_2 \to \mathcal{I}^\vee$ by $\pi(\alpha, \beta) = \pi_1(\alpha)-
\pi_2(\beta)$ and $\iota\colon\mathcal{I}\to M_1\oplus M_2$ by $\iota = \iota_1\oplus(-\iota_2)$.
The sum of the two extensions is
\begin{equation*}
  0\to\mathcal{I}\to \Ker(\pi)/\Im(\iota)\to\mathcal{I}^\vee\to 0.
\end{equation*}
and the identity element is the trivial extension $\mathcal{I}\oplus\mathcal{I}^\vee$.

Let $\Hom_{\ratls}^{+}(F,F)$ be the space of endomorphisms of $F$ that are self-adjoint with respect
to the trace pairing. Note that $\Hom_F(F, F)\subset\Hom_\ratls^+(F, F)$.  For $x\in F$, let
$M_x\in\Hom_F(F, F)$ denote the multiplication-by-$x$ endomorphism.

Given $T\in\Hom_\ratls^+(F,F)$, let $\mathcal{O}(T)$ be the order
\begin{equation*}
  \{x\in F : [M_x,T](\mathcal{I}^\vee)\subset\mathcal{I}\},
\end{equation*}
where $[X , Y]=XY-YX$ is the commutator.  That $\mathcal{O}(T)$ is a subring of $F$ follows from
the formula
\begin{equation*}
  M_\lambda[M_\mu,T]+[M_\lambda, T]M_\mu = [M_{\lambda\mu},T].
\end{equation*}
Define a symplectic extension
$(\mathcal{I}\oplus\mathcal{I}^\vee)_T$ of $\mathcal{I}^\vee$ by $\mathcal{I}$ over $\mathcal{O}(T)$
by giving $\mathcal{I}\oplus\mathcal{I}^\vee$ the $\mathcal{O}(T)$-module structure
\begin{equation*}
  \lambda\cdot(\alpha, \beta) = (\lambda\alpha + [M_\lambda, T](\beta), \lambda\beta).
\end{equation*}

\begin{theorem}
  \label{thm:symp-exts}
  The map $T\mapsto (\mathcal{I}\oplus\mathcal{I}^\vee)_T$ induces an isomorphism
  \begin{equation*}
    \Hom_\ratls^+(F,F)/(\Hom_F(F,F)+\Hom_\zed^+(\mathcal{I}^\vee, \mathcal{I})) \to E(\mathcal{I}).
  \end{equation*}
\end{theorem}

\begin{proof}
  To see that our map is a well-defined homomorphism is just a matter of working through the
  definitions, which we leave to the reader.

  To show our map is a monomorphism, suppose $(\mathcal{I}\oplus\mathcal{I}^\vee)_T$ is isomorphic
  to the trivial extension via $\phi\colon(\mathcal{I}\oplus\mathcal{I}^\vee)_T
  \to\mathcal{I}\oplus\mathcal{I}^\vee$.  This isomorphism must be of the form $\phi(\alpha,\beta) =
  (\alpha + R(\beta), \beta)$ for some self-adjoint $R\colon\mathcal{I}^\vee\to\mathcal{I}$.  The
  condition that this is an $\mathcal{O}(T)$-module isomorphism implies $[M_x, T-R]=0$ for all
  $x\in\mathcal{O}(T)$.  Since $\Hom_F(F,F)$ is its own centralizer in $\Hom_\ratls(F,F)$, we must
  have $T-R\in\Hom_F(F,F)$, so $T\in\Hom_F(F,F)+\Hom_\zed^+(\mathcal{I}^\vee, \mathcal{I})$.

  Now consider the space $\mathcal{D}=\Hom_\ratls(F,\Hom_\ratls^-(F,F))$.  We write elements of
  $\mathcal{D}$ as $Q_x$ with $Q_x\in \Hom_\ratls^-(F,F)$ for each $x\in F$.  Let
  $\mathcal{C}\subset\mathcal{D}$ be those elements $Q_{\_}$ satisfying
  \begin{equation}
    \label{eq:2}
    M_xQ_y + Q_xM_y = Q_{xy}
  \end{equation}
  for all $x,y\in F$.  We claim that every element of $\mathcal{C}$ is of the form $Q^T_x = [M_x,
  T]$.  To see this, let $\theta$ be a generator of $F$ over $\ratls$.  The map $\mathcal{C}\to
  \Hom_\ratls^-(F,F)$ sending $Q_{\_}$ to $Q_\theta$ is injective by \eqref{eq:2}, so $\dim
  \mathcal{C}\leq d(d-1)/2$, where $d=[F:\ratls]$.  The map
  $\Hom_\ratls^+(F,F)/\Hom_F(F,F)\to\mathcal{C}$ sending $T$ to $Q_\_^T$ is injective so is an
  isomorphism because the domain also has dimension $d(d-1)/2$.  Thus every element of $\mathcal{C}$
  has the desired form.

  Now, every symplectic extension of $\mathcal{I}^\vee$ by $\mathcal{I}$ over an order
  $\mathcal{O}$ is isomorphic as a symplectic $\zed$-module to $\mathcal{I}\oplus\mathcal{I}^\vee$
  with the $\mathcal{O}$-module structure,
  \begin{equation*}
    \lambda\cdot(\alpha,\beta) = (\lambda\alpha + Q_\lambda(\beta), \lambda\beta),
  \end{equation*}
  with $Q_\_\in\mathcal{C}$.  Since $Q_\_ = Q_\_^T$ for some $T$, our map is surjective.
\end{proof}

Given an order $\mathcal{O}\subset\mathcal{O}_\mathcal{I}$, let $E^\mathcal{O}(\mathcal{I})\subset
E(\mathcal{I})$ be the subgroup of extensions over some order $\mathcal{O}'$ such that
$\mathcal{O}\subset\mathcal{O}'\subset\mathcal{O}_\mathcal{I}$, and let
$E_\mathcal{O}(\mathcal{I})\subset E^\mathcal{O}(\mathcal{I})$ be the set of extensions over
$\mathcal{O}$.  From the above description of $E(\mathcal{I})$, we obtain:

\begin{cor}
  $E(\mathcal{I})$ is a torsion group with $E^\mathcal{O}(\mathcal{I})$  a finite subgroup.
\end{cor}

If two lattices  $\mathcal{I}$ and $\mathcal{I}'$ are in the same ideal class, then the groups
$E(\mathcal{I})$ are canonically isomorphic.

\paragraph{Real multiplication.}

We now suppose $F$ is a totally real number field of degree $g$.

Consider a principally polarized $g$-dimensional Abelian variety $A$.  We let $\End(A)$ be the ring of
endomorphisms of $A$ and $\End^0(A)$ the subring of endomorphisms such that the induced endomorphism
of $H_1(A; \ratls)$ is self-adjoint with respect to the symplectic structure defined by the
polarization.

\emph{Real multiplication} by $F$ on $A$ is a monomorphism $\rho\colon F\to \End^0(A) \tensor_\zed
\ratls$.  The subring $\mathcal{O}=\rho^{-1}(\End(A))$ is an order in $F$, and we say that
$A$ has real multiplication by $\mathcal{O}$.

There can be many ways for a given Abelian variety to have real multiplication by $\mathcal{O}$.  We
write $\Gal(\mathcal{O}/\zed)$ for the subgroup of the Galois group $\Gal(F/\ratls)$ which preserves
$\mathcal{O}$.  If $\rho\colon\mathcal{O}\to\End^0(A)$ is real multiplication of $\mathcal{O}$ on
$A$, then so is $\rho\circ\sigma$ for any $\sigma\in \Gal(\mathcal{O}/\zed)$.

Let $\AVmoduli= \half_g/\Sp_{2g}(\zed)$ be the moduli space of $g$-dimensional principally polarized
Abelian varieties (where $\half_g$ is the $g(g+1)/2$-dimensional Siegel upper half space).  We
denote by $\RA\subset\AVmoduli$ the locus of Abelian varieties with real multiplication by $\mathcal{O}$. 

\paragraph{Eigenforms.}

Real multiplication $\rho\colon\mathcal{O}\to\End^0(A)$ induces a monomorphism
$\rho\colon\mathcal{O}\to\End\Omega(A)$, where $\Omega(A)$ is the vector space of holomorphic
one-forms on $A$.  If $\iota\colon F\to\reals$ is
an embedding of $F$, we say that $\omega\in\Omega(X)$ is an $\iota$-eigenform if
\begin{equation*}
  \lambda\cdot\omega = \iota(\lambda)\omega
\end{equation*}
for all $\lambda\in\mathcal{O}$.  Equivalently, $\omega$ is an $\iota$-eigenform if
\begin{equation*}
  \int_{\lambda\cdot\gamma}\omega = \iota(\lambda)\int_\gamma\omega
\end{equation*}
for all $\lambda\in\mathcal{O}$ and $\gamma\in H_1(A; \zed)$.  If we do not wish to specify an
embedding $\iota$, we just call $\omega$ an eigenform.

Given an embedding $\iota$ and $\iota$-eigenform $(A, \omega)$, there is a unique choice of real
multiplication $\rho\colon\mathcal{O}\to\End^0(A)$ which realizes $(A, \omega)$ as an $\iota$-eigenform.
Thus considering $\iota$-eigenforms allows one to eliminate the ambiguity of the choice of real
multiplication.  

We denote by $\Omega^\iota(X)$ the one-dimensional space of $\iota$-eigenforms.  We obtain the
eigenform decomposition,
\begin{equation}
  \label{eq:eigenspace_decomposition}
  \Omega(X) =
  \bigoplus_{\iota\colon F\to\reals}\Omega^\iota(X),
\end{equation}
where the sum is over all field embeddings $\iota$.

We denote by $\oA\to\A$ the moduli space of pairs $(A, \omega)$ where $A$ is a principally polarized
Abelian variety and $\omega$ is a nonzero holomorphic one-form on $A$.  We write $\EA\subset\poA$
for the locus of eigenforms for real multiplication by $\mathcal{O}$ and $\EA^\iota$ for the locus
of $\iota$-eigenforms.  Note that for $\Gal(\mathcal{O}/\zed)$-conjugate embeddings $\iota$ and
$\iota'$, the eigenform loci $\EA^\iota$ and $\EA^{\iota'}$ coincide (as an $\iota$-eigenform is
simultaneously an $\iota'$-eigenform for a Galois conjugate real multiplication); however, each $(A,
\omega)\in\EA^\iota$ comes with  a canonical choice or real multiplication which depends on $\iota$.

\paragraph{Hilbert modular varieties.}

Choose an ordering $\iota_1, \ldots, \iota_g$ of the $g$ real embeddings of $F$.  We use the
notation $x^{(i)}=\iota_i(x)$.  The group $\SL_2(F)$ then acts on $\half^g$ by $A\cdot(z_i)_{i=1}^g
= (A^{(i)}\cdot z_i)_{i=1}^g$, where $\SL_2(\reals)$ acts on the upper-half plane $\half$ by \Mobius
transformations in the usual way.

Given a lattice $M\subset F^2$, we define $\SL(M)$ to be the subgroup of $\SL_2(F)$ which preserves
$M$.  The Hilbert modular variety associated to $M$ is
\begin{equation*}
  X(M) = \half^g/\SL(M).
\end{equation*}
Given an order $\mathcal{O}\subset F$, we define
\begin{equation*}
  X_\mathcal{O} = \coprod_M X(M),
\end{equation*}
where the union is over a set of representatives of all isomorphism classes of proper rank two
symplectic $\mathcal{O}$-modules. If $\mathcal{O}$ is a maximal order, then every rank two
symplectic $\mathcal{O}$-module is isomorphic to $\mathcal{O}\oplus\mathcal{O}^\vee$ (this also
holds if $g=2$; see \cite{mcmullenabel}), so in this case $X_\mathcal{O}$ is connected.  In general,
$X_\mathcal{O}$ is not connected, as there are nonisomorphic proper symplectic
$\mathcal{O}$-modules; see Appendix\,\ref{sec:HMvsRM}.

There are canonical maps $j_\iota\colon X_\mathcal{O}\to\EA^\iota$ and $j\colon X_\mathcal{O}\to\RA$
defined as follows.  Given a lattice $M\subset F^2$ and $\tau = (\tau_i)_{i=1}^g\in \half^g$, we
define $\phi_\tau\colon M \to \cx^g$ by
\begin{equation*}
  \phi_\tau(x, y) = (x^{(i)}+ y^{(i)}\tau_i)_{i=1}^g.
\end{equation*}
The Abelian variety $A_\tau = \cx^g/\phi_\tau(M)$ has real multiplication by $\mathcal{O}$ defined
by
$\lambda\cdot(z_i)_{i=1}^g = (\lambda^{(i)} z_i)_{i=1}^g$.  The form $dz_i$ is an
$\iota_i$-eigenform.

The map $j_\iota\colon X_\mathcal{O}\to\EA^\iota$ is an isomorphism, so we may regard $X_\mathcal{O}$ as
the moduli space of principally polarized Abelian varieties $A$ with a choice of real multiplication
$\rho\colon\mathcal{O}\to\End^0(A)$.

The Galois group $\Gal(\mathcal{O}/\zed)$ acts on $X_\mathcal{O}$, and the map $j$ factors through
to a generically one-to-one map $j'\colon X_\mathcal{O}/\Gal(\mathcal{O}/\zed)\to \RA$.

\paragraph{Cusps of Hilbert modular varieties.}

The Baily-Borel-Satake compactification $\bX(M)$ of $X(M)$ is a projective variety obtained by
adding finitely many points to $X(M)$ which we call the \emph{cusps} of $X(M)$.  More precisely, we
embed $\proj^1(F)$ in $(\half\cup\{i\infty\})^g$ by $(x : y)\mapsto(x^{(i)}/y^{(i)})_{i=1}^g$.  We
define $\half^g_F = \half^g\cup\proj^1(F)$ with a certain topology whose precise definition is not
needed for this discussion; see \cite{borelji}.  The compactification of $X(M)$ is $\bX(M) =
\half^g_F/\SL(M)$.  We define $\bX_\mathcal{O}$ to be the union of the compactifications of its
components.

\begin{prop}
  \label{prop:cusp_classification}
  There is a natural bijection between the set of cusps of $X_\mathcal{O}$ and the set of
  isomorphism classes of symplectic extensions
  \begin{equation}
    \label{eq:37}
    0 \to \mathcal{I} \to N \to \mathcal{I}^\vee \to 0
  \end{equation}
  with $N$ a primitive rank-two symplectic $\mathcal{O}$-module and $\mathcal{I}$ a torsion-free
  rank one $\mathcal{O}$-module.  The cusps of $X(M)$ correspond to the isomorphism classes of such
  extensions where $M\isom N$ as symplectic $\mathcal{O}$-modules.
\end{prop}

\begin{proof}[Sketch of proof]
  Fix a lattice $M\subset F^2$.  We must provide a $\SL(M)$-equivariant bijection between lines
  $L\subset F^2$ and extensions $0\to \mathcal{I}\to M\to \mathcal{I}^\vee\to 0$ (up to isomorphism
  which is the identity on $M$).  We assign to a line $L$, the extension $0\to L\cap M \to M \to
  M/(L\cap M)\to 0$.  The line $L$ is recovered from an extension $0\to\mathcal{I}\to
  M\to\mathcal{I}^\vee\to 0$ by defining $L = \mathcal{I}\otimes\ratls$.

  The bijection for cusps of $X_\mathcal{O}$ follows immediately.
\end{proof}

Consider the set of all pairs $(\mathcal{I}, T)$, where $\mathcal{I}$ is a lattice in $F$ whose
coefficient ring contains $\mathcal{O}$, and $T\in E_\mathcal{O}(\mathcal{I})$.  The multiplicative
group of $F$ acts on such pairs by $a\cdot(\mathcal{I}, T) = (a\mathcal{I}, T^a)$, where $T^a(x) =
aT(ax)$ (using the identification of Theorem~\ref{thm:symp-exts}).  
We define a \emph{cusp packet} for real multiplication by $\mathcal{O}$ to be an equivalence class
of a pair $(\mathcal{I},T)$ under this relation.
  
We denote by $\mathcal{C}(\mathcal{O})$ the finite set of cusp packets for real multiplication by
$\mathcal{O}$.  We have seen that there are canonical bijections
between $\mathcal{C}(\mathcal{O})$, the set of isomorphism classes of symplectic extensions of the
form \eqref{eq:37}, the set of cusps of $X_\mathcal{O}$, and the set of cusps of $\EA^\iota$.
Moreover, there is a canonical bijection between the set of cusps of $\RA$ and $\mathcal{C}(\mathcal{O})/\Gal(\mathcal{O}/\zed)$.

\section{Stable Riemann surfaces and their moduli}
\label{sec:RS}


In this section, we discuss some background material on Riemann surfaces with nodal singularities,
holomorphic one-forms, and their various moduli spaces.

\paragraph{Stable Riemann surfaces.}  

A \emph{stable Riemann surface} (or stable curve) is a connected,
compact, one-dimensional, complex analytic variety with possibly finitely many
nodal singularities~-- that is, singularities of the form $zw=0$~-- such that each component of the
complement of the singularities has negative Euler characteristic.  In other terms, a stable Riemann
surface can be regarded a disjoint union of finite volume hyperbolic Riemann surfaces with cusps, together
with an identification of the cusps into pairs, each pair forming a node.  We will refer to a pair
of cusps facing a node as \emph{opposite cusps}.

The \emph{arithmetic genus} of a stable Riemann surface is the genus of the nonsingular surface obtained by
thickening each node to an annulus;  the \emph{geometric genus} is the
sum of the genera of its irreducible components.

\paragraph{Homology.}

Given a stable Riemann surface $X$, let $X_0$ be the complement of the nodes.  For each cusp $c$ of
$X_0$, let $\alpha_c\in H_1(X_0;\zed)$ be the class of a positively oriented simple closed curve
winding once around $c$, and let $I\subset H_1(X_0; \zed)$ be the subgroup generated by the
expressions $\alpha_c + \alpha_d$, where $c$ and $d$ are cusps joined to a node on $X$.

We define $\hatH_1(X; \zed) = H_1(X_0; \zed)/I$.  Defining $C(X)\subset \hatH_1(X; \zed)$ to be the
free Abelian subgroup (of rank equal to the number of nodes) generated by the $\alpha_c$, we have
the canonical exact sequence
\begin{equation*}
  0 \to C(X) \to \hatH_1(X; \zed)\to H_1(\tilde{X}; \zed)\to 0,
\end{equation*}
where $\tilde{X}\to X$ is the normalization of $X$.

\paragraph{Markings.}

Fix a genus $g$ surface $\Sigma_g$, and let $X$ be a genus $g$ stable Riemann surface.  A
\emph{collapse} is a map $f\colon\Sigma_g\to X$ such that the inverse image of each node is a simple
closed curve and $f$ is a homeomorphism on the complement of these curves.

A \emph{marked stable Riemann surface} is a stable Riemann surface $X$ together with a collapse
$f\colon\Sigma_g\to X$.  Two marked stable Riemann surfaces $f\colon \Sigma_g\to X$ and $g\colon
\Sigma_g\to Y$ are equivalent if there is homeomorphism $\phi\colon \Sigma_g\to\Sigma_g$ which is
homotopic to the identity and a conformal isomorphism $\psi\colon X\to Y$ such that $g\circ\phi =
\psi\circ f$.

\paragraph{Augmented \Teichmuller space.}

The \emph{\Teichmuller space} $\teich(\Sigma_g)$ is the space of nonsingular marked Riemann
surfaces of genus $g$.  It is contained in the \emph{augmented \Teichmuller space}
$\augteich(\Sigma_g)$, the space of marked stable Riemann surfaces of genus $g$.
We give  $\augteich(\Sigma_g)$
the smallest topology such that the hyperbolic length of any simple
closed curve is continuous as a function  $\augteich(\Sigma_g)\to\reals_{\geq 0}
\cup \{\infty\}$.  Abikoff \cite{abikoff77} showed that this topology agrees with other natural
topologies on $\augteich$ defined via quasiconformal mappings or quasi-isometries.

\paragraph{Deligne-Mumford compactification.}

The mapping class group $\Mod(\Sigma_g)$ of orientation preserving homeomorphisms of $\Sigma_g$
defined up to isotopy acts on $\teich(\Sigma_g)$ and $\augteich(\Sigma_g)$ by precomposition of
markings.  The moduli space of genus $g$ Riemann surfaces is the quotient
$\moduli=\teich(\Sigma_g)/\Mod(\Sigma_g)$.  The Deligne-Mumford compactification of $\moduli$ is
$\barmoduli = \augteich(\Sigma_g)/\Mod(\Sigma_g)$, the moduli space of genus $g$ stable curves.

Over $\barmoduli$ is the universal curve $p\colon\overline{\mathcal{C}}\to\barmoduli$, a compact algebraic
variety whose fiber over a point representing a stable curve $X$ is a curve isomorphic to $X$ (provided
$X$ has no automorphisms).

\paragraph{Stable Abelian differentials.}

Over $\moduli$ is the vector bundle $\Omega\moduli\to\moduli$ whose fiber over $X$ is the space
$\Omega(X)$ of holomorphic one-forms on $X$.  We extend this to the vector bundle
$\Omega\barmoduli\to \barmoduli$ whose fiber $\Omega(X)$ over $X$ is the space of stable Abelian
differentials on $X$, defined as follows.

Given a genus $g$ stable Riemann surface $X$, a \emph{stable Abelian differential} is a holomorphic
one-form on $X_0$, the complement in $X$ of its nodes, such that:
\begin{itemize}
\item $\omega$ has at worst simple poles at the cusps of $X_0$.
\item If $p$ and $q$ are opposite cusps of $X_0$, then
  \begin{equation*}
    \Res_p\omega = -\Res_q\omega.
  \end{equation*}
\end{itemize}
The dualizing sheaf $\omega_X$ is the sheaf on $X$ of
one-forms locally satisfying the two above conditions (see \cite[p.~82]{harrismorrison}), so
a stable Abelian differential is simply a global section of the dualizing
sheaf $\omega_X$.  We write $\Omega(X)$ for the space of stable Abelian differentials on $X$, a
$g$-dimensional vector space by Serre duality.

In the universal curve $p\colon\overline{\mathcal{C}}\to\barmoduli$, let $\overline{\mathcal{C}}_0$ be the
complement of the nodes of the fibers.  The relative cotangent sheaf of
$\overline{\mathcal{C}}_0\to\barmoduli$ (the sheaf of cotangent vectors to the fibers) is an
invertible sheaf which extends in a unique way to an invertible sheaf
$\omega_{\overline{\mathcal{C}}/\barmoduli}$ on $\overline{\mathcal{C}}$,  the relative dualizing
sheaf of this family of curves.

The restriction of $\omega_{\overline{\mathcal{C}}/\barmoduli}$ to a fiber $X$ of this family is
simply $\omega_X$.   The push-forward $p_*\omega_{\overline{\mathcal{C}}/\barmoduli}$ is the
sheaf of sections of the rank $g$ vector bundle $\Omega\barmoduli\to\barmoduli$.

\paragraph{Plumbing coordinates.}

Following Wolpert \cite{wolpert87} we give explicit holomorphic coordinates at the boundary of
$\barmoduli$ and a model of the universal curve in these coordinates.  See also \cite{bers73,bers81}
and \cite{masur76}.

Let $X$ be a stable curve with nodes $n_1, \ldots, n_k$, and let $X_0$ be $X$ with the nodes
removed, a disjoint union of punctured Riemann surfaces.  At each node $n_i$, let $U_i$ and $V_i$ be
small neighborhoods of $n_i$ in each of the two branches of $X$ through $n_i$, and choose conformal
maps $F_i\colon U_i\to\cx$ and $G_i\colon V_i\to\cx$ whose images contain the unit disk around the
origin $\Delta_1$.  We write $z_i$ and $w_i$ for the coordinates on $U_i$ and $V_i$ induced by these
maps.  We define
\begin{gather*}
  X^* = X \setminus \bigcup_i \left(\{|z_i|<1\} \cup \{|w_i|<1\}\right) \quad\text{and}\\
  M = X^* \times\Delta_1^k.
\end{gather*}

We take a model of a degeneration of a family of curves.
\begin{equation*}
  \mathbb{V}_i = \{(x_i,y_i,\bt)\in \Delta_1\times\Delta_1\times\Delta_1^k : x_iy_i=t_i\},
\end{equation*}
where $\bt=(t_i,\ldots, t_k)$.  The fiber $\mathbb{V}^{\bt}$ of the projection $(x_i,y_i,\bt) \mapsto \bt$ is a
nonsingular annulus except when $t_i=0$, in  which case it is two disks meeting at a node.  

Let $\mathcal{X}\to\Delta_1^k$ be the family of stable curves obtained by gluing each $\mathbb{V}_i$
to $M$ by the maps
\begin{equation*}
  \hat{F}_i(p, \bt) = (F_i(p), t_i/F_i(p), \bt) \qtq{and} \hat{G}_i(p, \bt) = (t_i/G_i(p),G_i(p), \bt),
\end{equation*}
defined on subsets of $M$.  The fiber $X_\bt$ over $\bt$ is simply the stable Riemann surface
obtained by removing the disks $\{|z_i| < |t_i|^{1/2}\}$ and $\{|w_i|<|t_i|^{1/2}\}$ and gluing the
boundary circles by the relation $w_i = t_i/z_i$.  If $t_i=0$, the node $n_i$ is unchanged.

Let $Q$ be the space of holomorphic quadratic differentials on $X_0$ with at worst simple poles at
the nodes.  Choose $3g-3-k$ Beltrami differentials $\mu_i$ on $X_0 \setminus \bigcup (U_i \cup V_i)$
so that no nontrivial linear combination of the $\mu_i$ pairs trivially with a quadratic
differential in $Q$.  Given ${\bs}\in \Delta_\epsilon^{3g-3-k}$ for sufficiently small $\epsilon$,
the Beltrami differential $\mu_{\bs} = \sum s_i \mu_i$ satisfies $\|\mu_{\bs}\|_\infty < 1$.

We define a family of stable curves $\mathcal{Y}\to \Delta_\epsilon^{3g-3-k}\times\Delta_1^k$ by
endowing $\mathcal{Y} = \mathcal{X}\times \Delta_\epsilon^{3g-3-k}$ with the complex structure on
$\mathcal{Y}$ defined by placing on each fiber $X_\bt^\bs$ over $(\bs,\bt)$ the Beltrami
differential $\mu_\bs$.

We obtain a holomorphic (orbifold) coordinate chart $\Delta_s^{3g-3-k}\times\Delta_1^k\to\barmoduli$
sending $(\bs,\bt)$ to the point representing the stable curve $X_\bt^\bs$.  The family
$\mathcal{Y}$ is the pullback of the universal curve by this coordinate chart.

\paragraph{Lagrangian markings.}

Given a genus $g$ stable curve $X$, a \emph{Lagrangian subgroup} of $\hatH_1(X; \zed)$ is a free
Abelian subgroup $L$ of rank $g$ such that $\hatH_1(X; \zed)/L$ is torsion-free and the restriction
of the intersection form on $H_1(\tilde{X}; \zed)$ to the image of $L$ under the canonical
projection $\hatH_1(X;\zed)\to H_1(\tilde{X}; \zed)$ is trivial.

Fix a free Abelian group $L$ of rank $g$.  A \emph{Lagrangian marking} of a genus $g$ stable Riemann
surface $X$ by $L$ is a monomorphism $\rho\colon L\to \hatH_1(X;\zed)$ whose image is a Lagrangian
subgroup.  The image $\rho(L)$ necessarily contains the subgroup $C(X)$ of $\hatH_1(X; \zed)$
generated by the nodes.  Thus we may assign to each node of $X$ its ``homology class'' in $L$, an
element of $L$ well-defined up to sign.

Let $\barmoduli(L)$ be the space of genus $g$ stable Riemann surfaces with a Lagrangian marking by $L$ and
$\moduli(L)\subset\barmoduli(L)$ the subspace of nonsingular surfaces.  If we
identify $L$ with a Lagrangian subgroup of $H_1(\Sigma_g; \zed)$, we have
$$\moduli(L) = \teich(\Sigma_g)/ \Mod(\Sigma_g, L),$$
where $\Mod(\Sigma_g, L)$ is the subgroup of $\Mod(\Sigma_g)$ fixing $L$ pointwise.
Moreover
$$\barmoduli(L) = \augteich(\Sigma_g, L) / \Mod(\Sigma_g, L),$$
where $\augteich(\Sigma_g, L)\subset\augteich(\Sigma_g)$ is the locus of stable Riemann surfaces
which can be obtained by collapsing only curves on $\Sigma_g$ whose homology class belongs to $L$ (including
homologically trivial curves).

Given a nonzero $\gamma\in L$, there is the divisor $D_\gamma\subset\barmoduli(L)$
consisting of stable curves where a curve homologous to $\gamma$ has been pinched.  $D_\gamma$
and $D_{-\gamma}$ are the same divisor.

The above plumbing coordinates provide in the same way coordinates at the boundary of $\barmoduli(L)$.

\paragraph{Weighted stable curves.}

Given a free Abelian group $L$, we define an \emph{$L$-weighted stable curve} to be a geometric genus $0$
stable curve with an element of $L$ associated to each cusp of $X$, called the weight of that cusp,
subject to the following restrictions:
\begin{itemize}
\item  Opposite cusps of $X_0$ have opposite weights.
\item The sum of the weights of the cusps of an irreducible component of $X$ is zero.
\item The weights of $X$ span $L$.
\end{itemize}
In other words, the first two conditions mean that the weights are subject to the same restrictions
as the residues of a stable form.

We say that two $L$-weighted stable curves $X$ and $Y$ are \emph{isomorphic} (resp.
\emph{topologically equivalent}) if there is a weight-preserving conformal isomorphism (resp.\
homeomorphism) $X\to Y$.

The notion of an $L$-weighting of a geometric genus $0$ stable curve $X$ is in fact equivalent to a
Lagrangian marking $\rho\colon L\to \hat{H}_1(X; \zed)$ (necessarily an isomorphism because $X$ is
genus $0$).  If $\alpha_c\in \hat{H}_1(X; \zed)$ is the
class of a positively oriented curve around a cusp $c$ with weight $w$, the marking $\rho$ maps $w$
to $\alpha_c$.

\paragraph{Weighted boundary strata.}

An \emph{$L$-weighted boundary stratum} is a topological equivalence class in the set of
all $L$-weighted stable curves.  If $X$ is an $L$-weighted stable curve having
$m$ components $C_i$, each homeomorphic to $\proj^1$ with $n_i$ points removed and with each
component having distinct weights, then the
corresponding $L$-weighted boundary stratum is an algebraic variety isomorphic to 
$$\prod_{i=1}^m\moduli[{0,n_i}],$$
where $\moduli[{0,n}]$ is the moduli space of $n$ labeled points on $\proj^1$, with each point being
labeled by its weight.  

The notion of a $L$-weighted boundary stratum is in fact equivalent to that of a boundary stratum in
$\barmoduli(L)$.  We consider two marked stable curves $(X, \rho)$ and $(Y, \sigma)$ in
$\barmoduli(L)$ to be equivalent if there is a homeomorphism $f\colon X\to Y$ which commutes with
the markings, and we define a \emph{Lagrangian boundary stratum} in $\bdry\barmoduli(L)$ to be an
equivalence class of this relation.  A Lagrangian boundary stratum is simply a maximal connected
subset of $\bdry\barmoduli(L)$ parameterizing homeomorphic stable curves.

In view of the above correspondence between $L$-weightings and Lagrangian markings by $L$, every
$L$-weighted boundary stratum $\mathcal{S}$ can be regarded canonically as a geometric genus zero Lagrangian
boundary stratum $\mathcal{S}\subset\barmoduli(L)$, and vice-versa. 

Given an $L$-weighted boundary stratum $\mathcal{S}$, we define
$\Weight(\mathcal{S})\subset L$ to be the set of weights of any surface in $\mathcal{S}$.

\paragraph{Embeddings of strata.}

Suppose now that $\mathcal{I}$ is a lattice in a degree $g$ number field $F$.
Given an $\mathcal{I}$-weighted boundary stratum $\mathcal{S}$ and a real embedding $\iota$ of $F$,
we define $p_\iota\colon\mathcal{S}\to\pobarmoduli$ by associating to a weighted stable curve $X$
the unique stable form on $X$ which has residue $\iota(w)$ at a cusp with weight $w$.  The $i^{\rm
  th}$ embedding $\mathcal{S}^\iota$ of $\mathcal{S}$ is its image under $p_\iota$.

\paragraph{Similar strata.}

Suppose $\mathcal{I}$ and $\mathcal{J}$ are lattices in a number field $F$.  We say that
$\mathcal{I}$ and $\mathcal{J}$-weighted stable curves $X$ and $Y$ are \emph{similar} if there is a
conformal isomorphism $X\to Y$ which sends each weight $x$ to $\lambda x$ for some fixed $\lambda\in
F$.

We say that two weighted boundary strata are \emph{similar} if they parameterize
similar weighted stable curves.  Note that if the unit group of $F$ is infinite, then 
$\mathcal{I}$-weighted boundary stratum is similar to infinitely many distinct
$\mathcal{I}$-weighted boundary strata.

\paragraph{Extremal length and the Hodge norm.}

Given any Riemann surface $X$, the \emph{Hodge norm} on $H_1(X; \reals)$ is defined by
\begin{equation*}
  \|\gamma\|_X = \sup_{\omega\in\Omega_1(X)} \left|\int_\gamma\omega\right|,
\end{equation*}
where $\Omega_1(X)$ denotes the space of forms with unit norm, for the norm $$\|\omega\| =
\left(\int_X|\omega|^2\right)^{1/2}.$$

Given a curve $\gamma$ on a Riemann surface $X$, we write $\Ext(\gamma)$ for the extremal length of
the family of curves which are homotopic to $\gamma$, that is
\begin{equation*}
  \Ext(\gamma) = \sup_\rho \frac{L(\rho)^2}{A(\rho)},
\end{equation*}
where the supremum is over all conformal metrics $\rho(z)dz$ with $\rho$ nonnegative and measurable,
\begin{equation*}
  L(\rho) = \inf_{\delta \homot \gamma} \int_\delta \rho(z) |dz|,
\end{equation*}
and
\begin{equation*}
  A(\rho) = \int_X \rho(z)^2 |dz|^2.
\end{equation*}

The relation between curves with small extremal length and homology classes with small Hodge norm is
summarized by the following two Propositions.

\begin{prop}
  \label{prop:ext_hodge_1}
  For any curve $\gamma$ on a Riemann surface $X$, we have
  \begin{equation*}
    \|\gamma\|_X^2 \leq \Ext(\gamma).
  \end{equation*}
\end{prop}

\begin{proof}
  Choose a form $\omega$ such that $\|\omega\|=1$ and $|\int_\gamma\omega| = \|\gamma\|_X$.  
  Regarding $|\omega|$ as a conformal metric on $X$, we obtain
  \begin{equation*}
    \|\gamma\|_X = \left|\int_\gamma\omega\right| \leq \int_\gamma|\omega|,
  \end{equation*}
  thus
  \begin{equation*}
    \|\gamma\|_X^2 \leq L(|\omega|)^2 \leq \Ext(\gamma).\qedhere
  \end{equation*}
\end{proof}

\begin{prop}
  \label{prop:ext_hodge_2}
  Given any Riemann surface $X$, there is a constant $C$~-- depending only on the genus of $X$~-- such
  that any cycle $\gamma\in H_1(X; \zed)$ is homologous to a sum of simple closed curves $\gamma_1,
  \ldots, \gamma_n$ such that for each $i$,
  \begin{equation}
    \label{eq:4}
    \Ext(\gamma_i) \leq C \|\gamma\|_X^2
  \end{equation}
\end{prop}

\begin{proof}
  Let $\omega$ be a holomorphic one-form on $X$ such that $\Im\omega$ is \Poincare dual to
  $\gamma$.  Since $\Im\omega$ has integral periods, the map $f\colon X\to \reals / \zed$ defined by
  $f(q) = \int_p^q\Im\omega$ (with $p$ a chosen basepoint) is well-defined.  The horizontal
  foliation of $\omega$ (that is, the kernel foliation of $\Im\omega$) is periodic, and each fiber
  $\gamma_r = f^{-1}(r)$ is a union of closed, horizontal leaves of $\omega$.  Giving the leaves of
  $\gamma_r$ the orientation defined by $\Re\omega$, we can regard $\gamma_r$ as a cycle in $H_1(X;
  \zed)$ which is homologous to $\gamma$.  By \Poincare duality,
  \begin{equation*}
    \length(\gamma_r) = \int_{\gamma_r}\Re\omega = \int_X \Re\omega\wedge\Im\omega = \frac{1}{2}\|\omega\|^2,
  \end{equation*}
  so each component of $\gamma_r$ has length at most $\|\omega\|^2/2$.

  Since $\omega$ has at most $2g-2$ distinct zeros, there is an open interval $I\subset\reals/\zed$
  of length at least $1/(2g-2)$ which is disjoint from the images of the zeros of $\omega$.  Choose
  some $r\in I$.  The inverse image $f^{-1}(I)$ consists of flat cylinders $C_1, \ldots, C_n$, each
  of height at least $1/(2g-2)$, and with each $C_i$ containing a component $\gamma^i_r$ of
  $\gamma_r$.  We obtain the bound,
  \begin{equation}
    \label{eq:11}
    \Mod(C_i) \geq \frac{2}{(2g-2)\|\omega\|^2},
  \end{equation}
  for the modulus of $C_i$. From monotonicity of extremal length, (see
  \cite[Theorem~I.2]{ahlfors_quasiconformal}) we have $\Ext(\gamma^i_r) \leq 1/\Mod(C_i)$, which
  with \eqref{eq:11} implies \eqref{eq:4} (setting $\gamma_i = \gamma^i_r$).
\end{proof}

\begin{remark}
  A similar argument is used by Accola in \cite{accola}, where he shows that $\|\gamma\|_X$ is equal
  to the extremal length of the homology class $\gamma$.
\end{remark}

\section{Period Matrices}
\label{sec:period-matrices}

In this section, we study period matrices as functions on $\barmoduli$.  We develop a
coordinate-free version of the classical period matrices.  We see that exponentials of entries of
period matrices are canonical meromorphic functions on $\barmoduli(L)$, and we calculate the orders
of vanishing of these functions along boundary divisors of $\barmoduli(L)$.
 
Fix a genus $g$ surface $\Sigma_g$ and a splitting of $H_1(\Sigma_g; \zed)$ into a sum of Lagrangian
subgroups,
\begin{equation*}
  H_1(\Sigma_g; \zed) = L \oplus M.
\end{equation*}
Given a surface $X\in\teich(\Sigma_g)$, integration of forms yields isomorphisms
\begin{equation*}
  P_L^X\colon \Omega(L)\to\Hom_\zed(L, \cx) \qtq{and} P_M^X\colon\Omega(X)\to\Hom_\zed(M, \cx).
\end{equation*}
We obtain a holomorphic map
\begin{equation}
  \label{eq:PM3}
  \teich(\Sigma_g)\to \Hom_\cx(\Hom_\zed(L, \cx), \Hom_\zed(M,
  \cx))\xrightarrow{\isom}L\otimes_\zed L\otimes_\zed\cx,
\end{equation}
where the second map uses the isomorphism $L\to M^*$ provided by the intersection form.  The Riemann
bilinear relations imply that the image of the map \eqref{eq:PM3} lies in $\Sym_\zed(L)$, so we
obtain a holomorphic map,
\begin{equation*}
  \Phi\colon\teich(\Sigma_g)\to \Sym_\zed(L)\otimes\cx,
\end{equation*}
and the dual homomorphism,
\begin{equation*}
  \Phi^*\colon \bS_\zed(\Hom(L, \zed))\to\Hol\teich(\Sigma_g),
\end{equation*}
where $\Hol\teich(\Sigma_g)$ denotes the additive group of holomorphic functions on $\teich(\Sigma_g)$.

The map $\Phi^*$ is just a coordinate-free version of the classical period matrix.  If we choose a
basis $(\alpha_i)$ of $L$ and dual bases $(\beta_i)$ of $M$  and $(\omega_i)$ of $\Omega(X)$, the
period matrix is $(\tau_{ij})$ where $\tau_{ij} = \omega_i(\beta_j)$.  The map $\Phi^*$ is simply
\begin{equation*}
  \Phi^*(\alpha_i^*\otimes \alpha_j^*) = \tau_{ij},
\end{equation*}
where $(\alpha_i^*)$ is the dual basis of $\Hom(L, \zed)$.

The map $\Phi^*$ depends on the choice of the complementary Lagrangian subgroup $M$. Every
complementary Lagrangian is of the form
\begin{equation*}
  M_T = \{m + T(m):m\in M\},
\end{equation*}
for some self-adjoint $T\colon M\to L$. Suppose we choose a different complementary Lagrangian $M_T$,
and $\Phi^*_T$ is the corresponding homomorphism.  The new homomorphism $\Phi^*_T$ is related to the old
one by
\begin{equation*}
  \Phi^*_T(x) = \Phi^*(x) + \langle x, T\rangle,
\end{equation*}
where we are regarding $T$ as an element of $\Sym_\zed(L)$.  It follows that the functions
$\Psi(x) =  e^{2\pi i\Phi^*(x)}$ do not depend on the choice of $M$ and so descend to nonzero
holomorphic functions on $\moduli(L)$.  We obtain a canonical homomorphism
\begin{equation*}
  \Psi\colon \bS_\zed(\Hom(L, \zed))\to\Hol^*\moduli(L).
\end{equation*}

\begin{theorem}
  \label{thm:meromorphic}
  For each $a\in\bS_\zed(\Hom(L, \zed))$, the function $\Psi(a)$ is meromorphic on $\barmoduli(L)$.
  For each nonzero $\gamma\in L$, the order of vanishing of $\Psi(a)$ along $D_\gamma$ is
  \begin{equation*}
    \vord_{D_\gamma}\Psi(a) = \langle \gamma\otimes\gamma, a\rangle.
  \end{equation*}

  $\Psi(a)$ is holomorphic and nowhere vanishing along any Lagrangian boundary stratum obtained by
  pinching a curve homologous to zero.

  If $\mathcal{S}\subset\bdry\barmoduli(L)$ is a Lagrangian boundary stratum with
  \begin{equation}
    \label{eq:7}
    \langle\gamma\otimes\gamma, a\rangle \geq 0
  \end{equation}
  for all $\gamma\in\Weight(\mathcal{S})$, then $\Psi(a)$ is holomorphic on $\mathcal{S}$.  If the
  pairing \eqref{eq:7} is zero for all $\gamma\in\Weight(\mathcal{S})$, then $\Psi(a)$ is nowhere
  vanishing on $\mathcal{S}$.  Otherwise $\Psi(a)$ vanishes identically on $\mathcal{S}$.
\end{theorem}

\begin{proof}
  We use in this proof the plumbing coordinates and related notation introduced in \S\ref{sec:RS}.
  Let $X$ be a stable curve with nodes $n_1, \ldots. n_k$ obtained by pinching curves $\gamma_1,
  \ldots, \gamma_k$ with homology classes $[\gamma_1], \ldots, [\gamma_k]\in L$. Let
  $$\mathcal{Y}\to B:=\Delta_\epsilon^{3g-3-k}\times\Delta_1^k$$
  be the family of stable curves
  constructed above with $X$ the fiber over $(\bzero,\bzero)$.  The nodes of this family are contained in the
  open sets
  \begin{equation*}
    \mathbb{W}_i := \mathbb{V}_i \times \Delta_\epsilon^{3g-3-k} = \{(x_i,y_i, \bs,\bt)\in
    \Delta_1\times\Delta_1\times\Delta_\epsilon^{3g-3-k}\times\Delta_1^k : x_iy_i=t_i\},
  \end{equation*}
  for $i=1, \ldots, k$.
  Define sections $p_i,q_i\colon B\to\mathcal{Y}$ with image in $\bdry\mathbb{W}_i$ by
  \begin{equation*}
    p_i(\bs,\bt) = (1,t_i, \bs, \bt) \qtq{and} q_i(\bs,\bt) = (t_i, 1, \bs, \bt).
  \end{equation*}

  Choose $\alpha_1\otimes\alpha_2\in \bS_\zed(\Hom(L, \zed))$ and let $\eta$ be the holomorphic
  section of the relative dualizing sheaf $\omega_\mathcal{Y/B}$ such that each period homomorphism
  $L\to\cx$ defined by each restriction $\eta_\bt^\bs$ to the fiber $X_\bt^\bs$ agrees with $\alpha_1\colon
  L\to\zed$.

  On $\mathbb{W}_i$ we may express $\eta$ as
  \begin{equation}
    \label{eq:5}
    \eta = \frac{\alpha_1([\gamma_i])}{2\pi i}\frac{dx_i}{x_i} + f_i\, dx_i + g_i\, dy_i
  \end{equation}
  with $f_i$ and $g_i$ holomorphic functions of $x_i, y_i, \bs$, and $\bt$.

  Let $\delta_{\bt,i}^\bs\colon[-1,1]\to \mathbb{W}_i$ be a path in the fiber of $\mathbb{W}_i$ over
  $(\bs,\bt)$ joining $p_i(\bs,\bt)$ to $q_i(\bs, \bt)$.  We may explicitly parameterize
  this path as
  \begin{equation*}
    \delta_{\bt,i}^\bs(r) =
    \begin{cases}
      (\sqrt{t_i}-r(1-\sqrt{t_i}),t_i/(\sqrt{t_i} - r(1-\sqrt{t_i})), \bs, \bt) & \text{if $r\leq 0$} \\
      (t_i/(r(1-\sqrt{t_i}) + \sqrt{t_i}),r(1-\sqrt{t_i}) + \sqrt{t_i},  \bs, \bt) &\text{if $r\geq 0$}.
    \end{cases}
  \end{equation*}
  We may choose a continuous family of 1-chains $\delta_{\bt,0}^\bs$ in $X_\bt^\bs$ with endpoints
  in $\{p_i(\bs,\bt), q_i(\bs,\bt)\}_{i=1}^k$ such that 
  \begin{equation*}
    \delta_{\bt}^\bs = \delta_{\bt,0}^\bs + \sum_{i=1}^k \alpha_2([\gamma_i])\delta_{\bt,i}^\bs
  \end{equation*}
  is a 1-cycle whose intersection with classes in $L$ agrees with the homomorphism $\alpha_2\colon
  L\to\zed$.

  We have
  \begin{equation}
    \label{eq:PM4}
    \Psi(\alpha_1\otimes\alpha_2)(\bs, \bt) = E\left(\int_{\delta_\bt^\bs}\eta_\bt^\bs\right),
  \end{equation}
  where we use the notation $E(z) = e^{2\pi i z}$.  The integral $\int_{\delta_{\bt,0}^\bs}\eta_\bt^\bs$
  is an integral of a holomorphically varying form over a 1-cycle with holomorphically varying
  endpoints, and so its contribution to \eqref{eq:PM4} is holomorphic and nonzero.  Thus it does not
  contribute to the order of vanishing of $\Psi(\alpha_1\otimes\alpha_2)$.

  The integral
  \begin{equation*}
    \int_{\delta_{\bt, i}^\bs} f_i \,dx_i + g_i\, dx_i
  \end{equation*}
  is a finite holomorphic function of $\bs$ and $\bt$ and so does not contribute to the order of
  vanishing of $\Psi(\alpha_1\otimes\alpha_2)$.  The factor of $\Psi(\alpha_1\otimes\alpha_2)$
  coming from  the first term of \eqref{eq:5} is
  \begin{equation*}
    E\left(\alpha_1([\gamma_i])\alpha_2([\gamma_i]) \int_{\delta_{\bt,i}^\bs}\frac{dx_i}{x_i}\right)
    = t_i^{\alpha_1([\gamma_i])\alpha_2([\gamma_i])}.
  \end{equation*}

  In our $(\bs, \bt)$-coordinates for $\barmoduli(L)$, the divisor $D_{\gamma_i}$ is the locus
  $\{t_i=0\}$.  We have seen that in these coordinates,
  \begin{equation}
    \label{eq:8}
    \Psi(\alpha_1\otimes\alpha_2)(\bs, \bt) = k(\bs, \bt)\prod_it_i^{\alpha_1([\gamma_i])\alpha_2([\gamma_i])},
  \end{equation}
  with $k$ a nonzero holomorphic function.  Thus $\Psi(\alpha_1\otimes\alpha_2)$ is meromorphic with
  the desired orders of vanishing.

  Now suppose $\mathcal{S}$ is a Lagrangian boundary stratum and $a\in \Hom(L, \zed)$ with
  $\langle\gamma\otimes\gamma, a\rangle \geq 0$ for each weight $\gamma$ , we see from \eqref{eq:8}
  that $\Psi(a)$ is holomorphic on $\mathcal{S}$, since each $t_i$ has nonnegative exponent.  If
  $\langle\gamma\otimes\gamma, a\rangle > 0$ for some weight $\gamma$, then some $t_i$ has positive
  exponent, so $\Psi(a)$ vanishes on $\mathcal{S}$.
\end{proof}

We will also need the following strengthening of this theorem.
\begin{cor}
  \label{cor:period_matrix_boundary}
  Let $\mathcal{S}\subset\bdry\barmoduli(L)$ be a Lagrangian boundary stratum obtained by pinching $m$
  curves on $\Sigma_g$ whose homology classes are $\gamma_1, \ldots, \gamma_n\in L$.  Take local
  coordinates $t_1, \ldots, t_n$ around some $x\in \mathcal{S}$ in which the divisor $D_{\gamma_i}$
  of curves obtained by pinching $\gamma_i$ is cut out by the equation $t_i=0$.  Then for any
  $a\in\bS_\zed(\Hom(L, \zed))$, the function
  \begin{equation*}
    \prod_{i=1}^m t_i^{-\langle \gamma\otimes\gamma, a\rangle} \Psi(a)
  \end{equation*}
  is holomorphic and nonzero on a neighborhood of $x$.  
\end{cor}

\begin{proof}
  This follows immediately from \eqref{eq:8}.
\end{proof}

\section{Boundary of the eigenform locus: Necessity}
\label{sec:bound-eigenf-locus}

In this section we begin the study of the boundary of the locus of Riemann surfaces whose Jacobians
have real multiplication.  We give an explicit necessary condition for a stable curve to lie in the
boundary of the real multiplication locus.  In \S\ref{sec:suffg3}, we will see that this
condition is also sufficient in genus three.

In all that follows, $F$ will denote totally real number field of degree $g$, $\mathcal{O}$ will
denote an order in $F$, and $\mathcal{I}$ will denote a lattice in $F$ whose coefficient ring
contains $\mathcal{O}$.

\paragraph{The real multiplication locus.}

The Jacobian of a stable curve $X$ is
$$\Jac(X) = \Omega(X)^*/ \hat{H}_1(X; \zed)=\Omega(X)^*/H_1(X_0; \zed),$$
where $X_0\subset X$ is the complement of the nodes.  The Jacobian is a compact Abelian variety if
each node of $X$ is separating, or equivalently if the geometric genus of $X$ is $g$.  Otherwise it
is a noncompact semi-Abelian variety.  We denote by $\tmoduli\subset\barmoduli$ the locus of stable
curves with compact Jacobians.  The Torelli map $t\colon\tmoduli\to\A$ maps each Riemann surface to
its Jacobian.

Let $\RM\subset\tmoduli$ be the locus of Riemann surfaces whose Jacobians have real multiplication
by $\mathcal{O}$.  In other words, $\RM = t^{-1}(\RA)$.  If $g$ is $2$ or $3$, then $t$ is a
bijection, so $\RM$ is a $g$-dimensional subvariety of $\tmoduli$.  In general, it is not known what
is the dimension of $\RM$, or even whether $\RM$ is nonempty.

We define $\E\subset\potmoduli$ to be the locus of eigenforms for real multiplication by
$\mathcal{O}$  and $\E^\iota$ to be the locus of $\iota$-eigenforms.  The Torelli map exhibits
$\E^\iota$ as a one-to-one branched cover of $\EA^\iota\isom X_\mathcal{O}$.

\paragraph{Admissible strata.}

The tensor product $F\otimes_\ratls F$ has the structure
of an $F$-bimodule.  We define
\begin{equation*}
  \Lambda^1 = \{x\in F\otimes_\ratls F : \lambda\cdot x = x\cdot\lambda\text{ for all $\lambda\in F$}\}.
\end{equation*}

\begin{prop}
  $\Lambda^1\subset\Sym_\ratls(F)$.
\end{prop}

\begin{proof}
  Identify $F$ with $\Hom_\ratls(F, \ratls)$ via the trace pairing.  This induces a canonical
  isomorphism $F\otimes_\ratls F\to\Hom_\ratls(F, F)$.  Under this isomorphism, $\Sym_\ratls(F)$
  corresponds to the self-adjoint endomorphisms $\Hom_\ratls^+(F, F)$, and $\Lambda^1$ corresponds
  to $\Hom_F(F, F)$.  Since left multiplication by $x\in F$ is self-adjoint,  $\Hom_F(F,
  F)\subset\Hom_\ratls^+(F, F)$.
\end{proof}

Identifying $F$ with its dual as above, the dual of
$\Sym_\ratls(F)$ is $\bS_\ratls(F)$.  We let $\Ann(\Lambda^1)\subset\bS_\ratls(F)$ denote the
annihilator of $\Lambda^1$.

Given an $\mathcal{I}$-weighted boundary stratum $\mathcal{S}$, we define the following cone and subspace of
$\bS_\ratls(F)$:
\begin{align*}
  C(\mathcal{S}) &= \{x\in \bS_\ratls(F): \langle x, \alpha\otimes\alpha\rangle\geq 0 \text{ for all $\alpha\in \Weight(\mathcal{S})$}\} \\
  N(\mathcal{S}) &= \{x\in \bS_\ratls(F): \langle x, \alpha\otimes\alpha\rangle =  0 \text{ for all $\alpha\in \Weight(\mathcal{S})$}\}.
\end{align*}

We say that an $\mathcal{I}$-weighted boundary stratum $\mathcal{S}$ is \emph{admissible} if 
\begin{equation}
  \label{eq:cone_condition}
  C(\mathcal{S})\cap \Ann(\Lambda^1)\subset N(\mathcal{S}).
\end{equation}
We will see in Corollary~\ref{cor:noweightedfinite} 
that if $\mathcal{I}$ is a lattice in a cubic field, then there are only
finitely many admissible $\mathcal{I}$-weighted boundary strata up to similarity.

\paragraph{Algebraic tori.}

Fix an $\mathcal{I}$-weighted boundary stratum $\mathcal{S}$.  There is a  surjective map of algebraic tori:
\begin{equation} \label{eq:defofp}
  p\colon \Hom(N(\mathcal{S})\cap\bS_\zed(\mathcal{I}^\vee), \Gm)\to \Hom(N(\mathcal{S})\cap
  \Ann(\Lambda^1)\cap \bS_\zed(\mathcal{I}^\vee), \Gm).
\end{equation}
The reader unfamiliar with algebraic groups should think of $\Gm$ as the multiplicative group
$\cx^*$ of nonzero complex numbers.

By the discussion at the end of \S\ref{sec:RS}, we may regard $\mathcal{S}$ as a boundary stratum
of $\barmoduli(\mathcal{I})$.  By Corollary~\ref{cor:period_matrix_boundary}, for each nonzero $a\in
N(\mathcal{S})\cap\bS_\zed(\mathcal{I}^\vee)$ the restriction of $\Psi(a)$ to $\mathcal{S}$ is a
nonzero holomorphic function on $\mathcal{S}$.  We obtain a canonical morphism,
\begin{equation} \label{eq:defCR}
  \CR\colon\mathcal{S}\to \Hom(N(\mathcal{S})\cap\bS_\zed(\mathcal{I}^\vee), \Gm). 
\end{equation}

Recall that $E(\mathcal{I})$ is the torsion Abelian group of symplectic extensions of
$\mathcal{I}^\vee$ by $\mathcal{I}$.  Identifying $\Hom_\ratls^+(F, F)$ with $\Sym_\ratls(F)$ via
the trace pairing,
the isomorphism of Theorem~\ref{thm:symp-exts} becomes an isomorphism,
\begin{equation*}
  \Sym_\ratls(F)/(\Lambda^1 + \Sym_\zed(\mathcal{I}))\to E(\mathcal{I}).
\end{equation*}
Given $T\in\Sym_\ratls(F)$ and $a\in N(\mathcal{S})\cap \Ann(\Lambda^1)\cap
\bS_\zed(\mathcal{I}^\vee)$, we define
\begin{equation}
  \label{eq:qTa}
  q(T)(a) = e^{-2\pi i \langle T, a\rangle}.
\end{equation}
Since $q(T)(a)=1$ if $T$ lies in $\Lambda^1$ or $\Sym_\zed(\mathcal{I})$, \eqref{eq:qTa} defines a
homomorphism,
\begin{equation*}
  q\colon E(\mathcal{I})\to\Hom(N(\mathcal{S})\cap\Ann(\Lambda^1)\cap \bS_\zed(\mathcal{I}^\vee), \Gm).
\end{equation*}

Given a symplectic extension $T\in E(\mathcal{I})$, we define
$$G(T) = p^{-1}(q(T)),$$
a translate of a subtorus of $\Hom(N(\mathcal{S})\cap\bS_\zed(\mathcal{I}^\vee))$.
We then obtain for each extension $T$ a subvariety of $\mathcal{S}$:
$$\mathcal{S}(T) = \CR^{-1}(G(T)).$$
We define $\mathcal{S}^\iota(T)\subset\pobarmoduli$ to be the image of $\mathcal{S}(T)$ under $p_\iota$.

If $\mathcal{S}$ is an $\mathcal{I}$-weighted stratum and $\mathcal{S}'$ is a similar
$a\mathcal{I}$-weighted stratum, then the subvarieties $\mathcal{S}(T)$ and $\mathcal{S}'(T^a)$ are
identified under the canonical isomorphism $\mathcal{S}\to\mathcal{S}'$.  Thus the variety
$\mathcal{S}(T)$ can be regarded as depending only on the similarity class of $\mathcal{S}$ and the
cusp packet $(\mathcal{I}, T)$.
 
\paragraph{Boundary of $\RM$.}

We can now state our necessary condition for a stable curve to be in the boundary of $\RM$.

\begin{theorem}
  \label{thm:boundary_nec}
  Consider an order $\mathcal{O}$ in a degree $g$ totally real number field $F$, a real embedding
  $\iota$ of $F$, and a cusp packet $(\mathcal{I}, T)\in\mathcal{C}(\mathcal{O})$.  The closure in
  $\pobarmoduli$ of
  the cusp of $\E^\iota$ associated to $(\mathcal{I}, T)$ is contained in the union over all
  admissible $\mathcal{I}$-weighted boundary strata $\mathcal{S}$ of the varieties
  $\mathcal{S}^\iota(T)$.

  The closure of the corresponding cusp of $\RM$ in $\barmoduli$ is contained in the union over all
  $\mathcal{I}$-weighted boundary strata $\mathcal{S}$ of the images of the $\mathcal{S}(T)$ under
  the forgetful map to $\barmoduli$.
\end{theorem}
\par
The proof of Theorem~\ref{thm:boundary_nec} comes at the end of this section.

\paragraph{Auxiliary real multiplication loci.}

Given a cusp packet $(\mathcal{I}, T)\in\mathcal{C}(\mathcal{O})$, let
\begin{equation*}
  \RM(\mathcal{I}, T)\subset\moduli(\mathcal{I})
\end{equation*}
be the locus of Riemann surfaces with Lagrangian marking $(X, \rho)$ such that $\Jac(X)$ has real
multiplication by $\mathcal{O}$, the marking $\rho\colon\mathcal{I}\to H_1(X; \zed)$ is an
$\mathcal{O}$-module homomorphism, and the extension of $\mathcal{O}$-modules
\begin{equation*}
  0\to\rho(\mathcal{I})\to H_1(X; \zed) \to H_1(X; \zed)/\rho(\mathcal{I})\to 0
\end{equation*}
is isomorphic to the extension determined by $(\mathcal{I}, T)$.

We also have bundles of eigenforms over $\RM(\mathcal{I}, T)$.  On $\barmoduli(\mathcal{I})$, there
is the trivial bundle $\Omega^\iota\barmoduli(\mathcal{I})$ of forms $\omega$ such that for some
constant $c$ and for each $\lambda\in \mathcal{I}$, we have $\int_{\rho(\lambda)}\omega = c
\iota(\lambda)$, where $\rho$ is the Lagrangian marking.  The restriction
$\Omega^\iota\barRM(\mathcal{I}, T)$ of $\Omega^\iota\barmoduli(\mathcal{I})$ to $\barRM(\mathcal{I}, T)$
is the trivial line bundle of $\iota$-eigenforms.  We denote its projectivization by
$\barE^\iota(\mathcal{I}, T)\subset\pobarmoduli(\mathcal{I})$.  

Given a cusp packet $(\mathcal{I}, T)$ and a symplectic isomorphism $\rho\colon
\mathcal{I}\oplus\mathcal{I}^\vee \to H_1(\Sigma_g; \zed)$, we define
\begin{equation*}
  \RT(\mathcal{I}, T,\rho)\subset\teich(\Sigma_g)
\end{equation*}
to be the locus of marked Riemann surfaces $(X, f)$ such that $\Jac(X)$ has real multiplication by
$\mathcal{O}$ and the symplectic $\zed$-module isomorphism
$$f_*\circ\rho\colon(\mathcal{I}\oplus\mathcal{I}^\vee)_T\to H_1(X; \zed)$$
is also an isomorphism of symplectic $\mathcal{O}$-modules.

The homomorphism $\rho$ determines a Lagrangian splitting of $H_1(\Sigma_g; \zed)$, and we obtain as
in \S\ref{sec:period-matrices} a holomorphic map $\Phi\colon \teich(\Sigma_g)\to \Sym_\zed(\mathcal{I})\otimes\cx$.

\begin{prop}
  \label{prop:TRM_formula}
  We have
  \begin{equation*}
    \RT(\mathcal{I}, T, \rho)= \Phi^{-1}(\Lambda^1 \otimes_\ratls \cx - T)
  \end{equation*}
\end{prop}

\begin{proof}
  In this proof, we will identify $\Sym_\zed(\mathcal{I})$ with $\Hom^+(\mathcal{I}^\vee,
  \mathcal{I})$.  Under this identification, we have
  {
    \allowdisplaybreaks
    \begin{gather*}
      \Sym_\zed(\mathcal{I})\otimes\cx = \Hom_\cx^+(\mathcal{I}^\vee\otimes\cx, \mathcal{I}\otimes\cx), \\
      \Lambda^1\otimes\cx = \Hom_F^+(\mathcal{I}^\vee\otimes\cx, \mathcal{I}\otimes\cx) \\
      \phi := \Phi(X, f) \in \Hom^+_\cx(\mathcal{I}^\vee\otimes\cx, \mathcal{I}\otimes\cx), \quad\text{and} \\
      T\in\Hom_\ratls^+(\mathcal{I}^\vee\otimes\ratls, \mathcal{I}\otimes\ratls).
    \end{gather*}
  }

  We have two splittings of $H_1(X; \cx)$: the one induced by $\rho$,
  \begin{equation*}
    H_1(X; \cx) = (\mathcal{I}\otimes\cx)\oplus(\mathcal{I}^\vee\otimes\cx),
  \end{equation*}
  and the Hodge decomposition,
  \begin{equation*}
    H_1(X; \cx) = \Hom(\Omega(X), \cx)\oplus\Hom(\overline{\Omega(X)}, \cx).
  \end{equation*}
  The Hodge decomposition is determined by the map
  $\phi\colon\mathcal{I}^\vee\otimes\cx\to\mathcal{I}\otimes\cx$:
  \begin{equation}
    \label{eq:10}
    \Hom(\Omega(X), \cx) = \Graph(\phi).
  \end{equation}

  The $\mathcal{O}$-module structure of $H_1(X; \cx)$ inherited from that of
  $(\mathcal{I}\oplus\mathcal{I}^\vee)_T$  induces real multiplication on $\Jac(X)$ if
  and only if it preserves the Hodge decomposition.   By \eqref{eq:10}, the Hodge decomposition is
  preserved if and only if
  \begin{equation*}
    \phi(\lambda\cdot\alpha) = \lambda\cdot\phi(\alpha) + [M_\lambda, T](\alpha)
  \end{equation*}
  for all $\alpha\in\mathcal{I}^\vee$ and $\lambda\in\mathcal{O}$, which holds if and
  only if
  \begin{equation*}
    (\phi+T)(\lambda\cdot\alpha) = \lambda\cdot(\phi+T)(\alpha),
  \end{equation*}
  that is, if and only if $\phi+T\in\Lambda^1$.
\end{proof}

\begin{cor}
  \label{cor:Psi_on_RM_formula}
  Given any $a\in\Ann(\Lambda^1)\subset\bS_\zed(\mathcal{I}^\vee)$, we have
  \begin{equation*}
    \Psi(a) \equiv q(T)(a)
  \end{equation*}
  on $\RM(\mathcal{I}, T)$.
\end{cor}

\begin{proof}
  This follows directly from Proposition~\ref{prop:TRM_formula} and the definition of $q$.
\end{proof}

\paragraph{Invariant vanishing cycles.}

Consider a family $\mathcal{X}\to\Delta$ of stable curves which is smooth over $\Delta^*$ in the
sense that the fiber $X_p$ over nonzero $p$ is smooth.  Any such family defines a holomorphic map
$\Delta\to\barmoduli$ sending $p$ to $X_p$, and conversely any holomorphic disk
$\Delta\to\barmoduli$ sending $\Delta^*$ to $\moduli$, after possibly taking a base extension (a
cover of $\Delta$ ramified only over $0$), arises from such a family.

In any smooth fiber $X_p$, there is a collection of isotopy class of simple closed curves, which we
call the \emph{vanishing curves} which are pinched as $p\to 0$.  The vanishing curves are consistent
in the sense that given any path in $\Delta^*$ joining $p$ to $q$, the lifted homeomorphism
$f\colon X_p\to X_q$ (defined up to isotopy) preserves the vanishing curves.  The \emph{vanishing
  cycles} in $H_1(X_p; \zed)$ are those cycles generated by the vanishing curves.  Trivializing the
family over a path starting and ending at $p$ yields a homeomorphism of $X_p$ which is simply a
product of Dehn twists around the vanishing curves. Thus the  monodromy
action of $\pi_1(\Delta^*, p)$ on $H_1(X_p; \zed)$ is unipotent and fixes pointwise the subgroup
$V_p\subset H_1(X_p; \zed)$ of vanishing cycles.

Real multiplication by $\mathcal{O}$ on the family $\mathcal{X}\to\Delta$ is a monomorphism
$\rho\colon\mathcal{O}\hookrightarrow\End^0\Jac_{\mathcal{X}/\Delta}$, where
$\Jac_{\mathcal{X}/\Delta}\to\Delta$ is the relative Jacobian of the family $\mathcal{X}\to\Delta$.
This is equivalent to a choice of real multiplication $\rho\colon\mathcal{O}\to\Jac(X_p)$ for each
smooth fiber $X_p$ with the requirement that each isomorphism $H_1(X_p;\zed)\to H_1(X_q;\zed)$
arising from the Gauss-Manin connection commutes with the action of $\mathcal{O}$.

\begin{prop}
  \label{prop:invariant_vanishing_cycles}
  Consider a family of genus $g$ stable curves $\mathcal{X}\to\Delta$, smooth over $\Delta^*$, with
  real multiplication by $\mathcal{O}$.  For each nonzero $p$, the subgroup $V_p\subset H_1(X_p;
  \zed)$ of vanishing cycles is preserved by the action of $\mathcal{O}$ on $H_1(X_p; \zed)$.
\end{prop}

\begin{proof}
  Since the action of $\mathcal{O}$ on first homology commutes with the Gauss-Manin connection, it
  is enough to show that $V_p$ is invariant for a single $p$.

  Let $\lambda\in\mathcal{O}$ be a primitive element for $F$.  For any $\gamma\in H_1(X_p; \zed)$,
  we have the bound,
  \begin{equation*}
    \|\lambda\cdot\gamma\|_{X_p} \leq \|\lambda\|_\infty \|\gamma\|_{X_p},
  \end{equation*}
  where $\|\lambda\|_\infty = \sup_\iota |\iota(\lambda)|$, with the supremum over all field
  embeddings $\iota\colon F\to\reals$, and $\|\cdot\|_{X_p}$ is the Hodge norm introduced in
  \S\ref{sec:RS}.
  
  There is a constant $D$ such that $\Ext(\gamma)\geq D$ for any curve $\gamma$ on $X_p$ which is not a
  vanishing curve.  For any $\epsilon>0$, we may choose $p$ sufficiently small that
  $\Ext(\gamma_i)<\epsilon$ for any vanishing curve $\gamma_i$.  By
  Proposition~\ref{prop:ext_hodge_1}, we have
  \begin{equation*}
    \|\lambda\cdot\gamma_i\|_{X_p} \leq \|\lambda\|_{\infty}\|\gamma_i\| < \|\lambda\|_\infty \epsilon^{1/2}.
  \end{equation*}
  By Proposition~\ref{prop:ext_hodge_2}, $\lambda\cdot\gamma_i$ is homologous to a
  sum of simple closed curves $\delta_j$ with
  \begin{equation*}
    \Ext(\delta_j) < C\|\lambda\|_\infty^2 \epsilon.
  \end{equation*}
  Thus $\Ext(\delta_j)<D$ if $\epsilon$ is chosen sufficiently small.  The $\delta_j$ must
  then be vanishing curves.  Thus the action of $\lambda$ preserves $V_p$, and since $\lambda$ is a
  primitive element, $V_p$ is preserved by $\mathcal{O}$.
\end{proof}

\begin{cor} \label{cor:geomgenus}
  Each stable curve in $\barRM\subset\barmoduli$ has geometric genus either $0$ or
  $g$.
\end{cor}

\begin{proof}
  Suppose $X$ is a stable curve in $\barRM$.  Choose a family of stable curves
  $\mathcal{X}\to\Delta$, smooth over $\Delta^*$, with real multiplication by $\mathcal{O}$, with
  $X$ the fiber over $0$.  The geometric genus of $X$ is $g -
  \rank V_p$ for any nonzero $p$.  By Proposition~\ref{prop:invariant_vanishing_cycles},
  $V_p\otimes\ratls$ is a vector space over $F$, so $\dim_\ratls V_p\otimes\ratls$ must be a
  multiple of $[F:\ratls] = g$.
\end{proof}

\paragraph{Proof of Theorem~\ref{thm:boundary_nec}.}

Consider $(X_0, \omega_0)$ in the closure of the cusp of $\E^\iota$ determined by the cusp packet
$(\mathcal{I}, T)$.  We first claim that $(X_0, \omega_0)$ must lie in the image of
$\barE^\iota(\mathcal{I}, T)\subset\pobarmoduli(\mathcal{I}, T)$.  Since $\barE^\iota$ is a
variety, we may choose a holomorphic disk $f\colon \Delta\to \barE^\iota$ sending $0$ to $(X_0,
\omega_0)$ and $\Delta^*$ to the cusp of $\E^\iota$ determined by $(\mathcal{I}, T)$.  Possibly
taking a base extension, we may assume $f$ arises from a family of stable curves
$\mathcal{X}\to\Delta$ with real multiplication by $\mathcal{O}$.  For each $p\in\Delta^*$, the
vanishing cycles $V_p$ for the fiber $X_p$ over $p$ are $\mathcal{O}$-invariant by
Proposition~\ref{prop:invariant_vanishing_cycles}, so we obtain an extension of
$\mathcal{O}$-modules
\begin{equation*}
  0\to V_p \to H_1(X; \zed) \to H_1(X; \zed)/V_p\to 0,
\end{equation*}
which must be isomorphic to the extension determined by $(\mathcal{I}, T)$.  Since the monodromy
action of $\pi_1(\Delta^*, p)$ on $V_p$ is trivial, we may identify each $V_q$ with $\mathcal{I}$
and obtain a consistent Lagrangian marking of
$\hatH_1(X_q;\zed)$ by $\mathcal{I}$ for each $q$, which defines a lift $g\colon\Delta\to\barE^\iota(\mathcal{I},
T)\subset\pobarmoduli(\mathcal{I})$.  It follows  that $(X_0, \omega_0)$ lies in the image of some
$(Y, \eta)\in\barE^\iota(\mathcal{I}, T)$ as claimed.

The form $(Y, \eta)$ must lie in some boundary stratum
$\mathcal{S}^\iota\subset\pobarmoduli(\mathcal{I})$ lying over
a boundary stratum $\mathcal{S}\subset\barmoduli(\mathcal{I})$.  We must then show that if the intersection
$\mathcal{S}\cap\barRM(\mathcal{I},T)$ is nontrivial, then $\mathcal{S}$ is admissible, and moreover
that $\mathcal{S} \cap\barRM(\mathcal{I}, T)\subset\mathcal{S}(T)$.

Suppose that the stratum $\mathcal{S}$ is not admissible, so the cone condition
\eqref{eq:cone_condition} does not hold.  Then there is some $a$ in
$C(\mathcal{S})\cap\Ann(\Lambda^1)\cap\bS_\zed(\mathcal{I}^\vee)$ but not in $N(\mathcal{S})$.  By
Theorem~\ref{thm:meromorphic}, the function $\Psi(a)$ is holomorphic and identically zero on
$\mathcal{S}$.  By Corollary~\ref{cor:Psi_on_RM_formula}, $\Psi(a)(x)\equiv q(T)(a)$, a nonzero
constant on ${\barRM(\mathcal{I}, T)}$.  In particular, $\Psi(a)$ is nonzero along
$\mathcal{S}\cap{\barRM(\mathcal{I}, T)}\neq\emptyset$, a contradiction.  Thus $\mathcal{S}$
is admissible.

Since $\Psi(a)(x)\equiv q(T)(a)$ on ${\barRM(\mathcal{I}, T)}$ for all $a\in
N(\mathcal{S})\cap\Ann(\Lambda^1)\cap\bS_\zed(\mathcal{I}^\vee)$, it follows immediately that ${\barRM(\mathcal{I}, T)}\cap\mathcal{S}\subset\mathcal{S}(T)$.
\qed
  
\section{A geometric reformulation of admissibility}
\label{sec:reform_admiss}

The aim of this section is to give a more explicit reformulations of when
an $\mathcal{I}$-weighted boundary stratum is admissible. 

\paragraph{The no-half-space condition.}

Consider a finite dimensional vector space $V$ over $\ratls$.  We say that a set $S=\{v_1, \ldots,
v_n\}\subset V$ satisfies the \emph{no-half-space condition} if it is not contained in a closed
half-space of its $\ratls$-span.  Equivalently, $S$ satisfies the no-half space condition if and
only if zero is in the interior of the convex hull of $S$.  

\paragraph{The reformulation.}

Consider a totally real number field $F$ with Galois closure $K$.  Let $G = \Gal(K/\ratls)$ and $H =
\Gal(K/F)$.  We define $I = H\times H \semidirect \zed/2\zed$, with $\zed/2\zed$ acting on $H\times
H$ by exchanging the factors. The group $I$ acts on $G$ by
\begin{equation*}
  (h_1,h_2,\epsilon)\cdot\gamma = h_2\gamma^\epsilon h_1^{-1},
\end{equation*}
where $\epsilon = \pm 1 \in \zed/2\zed$.  We let $\Stab(\sigma)\subset I$ denote the stabilizer of
$\sigma\in G$, and we define a homomorphism $\phi_\sigma\colon\Stab(\sigma)\to G$ by
\begin{equation*}
  \phi_\sigma(h_1,h_2,\epsilon) =
  \begin{cases}
    h_1 & \text{if $\epsilon=1$;} \\
    h_1\sigma & \text{if $\epsilon = -1$.}
  \end{cases}
\end{equation*}
Let $G_\sigma = \phi_\sigma(\Stab(\sigma))$ and $K_\sigma = K^{G_\sigma}$.  We define for each
$\sigma\in G$ a quadratic map $Q_\sigma\colon F \to K_\sigma$ by
$$Q_\sigma(t) = t \sigma^{-1}(t).$$

\begin{theorem}
  \label{thm:admiss-nonhalf}
  A weighted boundary stratum with weights $\{t_1, \ldots, t_n\}\subset F$ is admissible if and only
  if for each $\sigma\in G\setminus H$, the set
  $\{Q_\sigma(t_1),\ldots,Q_\sigma(t_n)\}\subset K_\sigma$ satisfies the no-half-space condition.
  In fact, it is enough to check this for each $\sigma$ in a set of orbit representatives of $G/I$.
\end{theorem}

The tensor product $K\otimes K$ has the structure of a $K$-bimodule.  Given $\sigma\in G$, we define
\begin{equation*}
  \Lambda_K^\sigma = \{ \lambda\in K\otimes K : x\cdot\lambda = \lambda\cdot\sigma(x) \text{ for all
  } x\in K\},.
\end{equation*}
generalizing the definition of $\Lambda^1\in\Sym_\ratls(F)$ in \S\ref{sec:bound-eigenf-locus}.

The trace pairing $\langle x, y\rangle_K = \Tr^K_\ratls(xy)$ on $K$ induces a pairing on $K\otimes
K$:
\begin{equation*}
  \langle x_1 \otimes x_2, y_1 \otimes y_2\rangle = \langle x_1, y_1\rangle_K\langle x_2, y_2\rangle_K.
\end{equation*}

\begin{lemma}
  \label{lem:epsilon_sigma}
  Let $r_1,\ldots, r_g$ be a basis of $K$ over $\ratls$ and $s_1, \ldots, s_g$ the dual basis with respect
  to the trace pairing.  The element
  \begin{equation*}
    \epsilon_\sigma = \sum_{i=1}^g r_i \otimes \sigma(s_i)\in K\otimes K.
  \end{equation*}
  lies in $\Lambda^\sigma$ and does not depend on the choice of basis $(r_i)$.  Moreover, for any
  $x\in K_\sigma$ and $t\in F$, we have
  \begin{equation*}
    \langle x \epsilon_\sigma, t \otimes t\rangle = [K:K_\sigma] \langle x, Q_\sigma(t)\rangle_{K_\sigma}.
  \end{equation*}
\end{lemma}

\begin{proof}
  Identifying $K\otimes K$ with $\Hom_\ratls(K, K)$ via the trace pairing, $\Lambda^\sigma$
  corresponds to
  \begin{equation*}
    \{\phi\colon K\to K : \phi(x\lambda) = \sigma(x)\phi(\lambda) \text{ for all
      $x,\lambda\in K$}\}.
  \end{equation*}
  Under this correspondence, $\epsilon_\sigma$ is the canonical map $\phi_\sigma(x) = \sigma(x)$.  Thus,
  $\epsilon_\sigma\in \Lambda^\sigma$ and does not depend on the choice of the $r_i$.

  Now, write $t\in F$ as $t = \sum t_i \sigma(r_i)$ for $t_i\in \ratls$.   We calculate
  \begin{align*}
    \allowdisplaybreaks
    \left\langle x\epsilon_\sigma , t \otimes t \right\rangle &= \left\langle\sum_k x r_k \otimes
      \sigma(s_k), \sum_{\ell,m}t_\ell t_m \sigma(r_\ell)\otimes\sigma(r_m)\right\rangle\\
    &= \sum_{k,\ell,m}t_\ell t_m \langle x r_k, \sigma(r_\ell)\rangle_K\langle \sigma(s_k),
    \sigma(r_m)\rangle_K \\
    &=\sum_{k,\ell} t_k t_\ell \langle x r_k, \sigma(r_\ell)\rangle_K \\
    &= \Tr^K_\ratls(x t \sigma^{-1}(t)) \\
    &= [K:K_\sigma]\Tr^{K_\sigma}_\ratls(x Q_\sigma(t)).
    \qedhere
  \end{align*}
\end{proof}

\paragraph{Proof of Theorem~\ref{thm:admiss-nonhalf}.}

We first wish to identify $\Sym_\ratls(F)$ and the orthogonal complement $(\Lambda^1_F)^\perp$ as
subspaces of $K\otimes K$.  We have the orthogonal decomposition,
\begin{equation*}
  K\otimes K = \bigoplus_{\sigma\in G} \Lambda^\sigma_K.
\end{equation*}
$\Sym_\ratls(F)$ is the subspace of $K\otimes K$ fixed by the action of $I$, so
\begin{equation*}
  \Sym_\ratls(F) = \bigoplus_{\tau\in G/I}\Gamma^\tau,
\end{equation*}
where for each orbit $\tau\in G/I$, we define $\Gamma^\tau$ to be the subspace of $\bigoplus_{\sigma\in
  \tau} \Lambda^\sigma_K$ fixed pointwise by the action of $I$.  Given any $\sigma$ in an orbit
$\tau\in G/I$, we define the isomorphism $v_\sigma\colon K_\sigma\to \Gamma^\tau$ by
\begin{equation*}
  v_\sigma(x) = \sum_{\gamma\in I/\Stab(\sigma)} \gamma(x\epsilon_\sigma) = \sum_{\gamma\in
    I/\Stab(\sigma)} x\epsilon_{\gamma\cdot\sigma}.
\end{equation*}
Choose a set $\sigma_1 = 1, \sigma_2, \ldots, \sigma_n\in G$ of representatives of the orbits $G/I$.  We
obtain an isomorphism,
\begin{equation*}
  v\colon \bigoplus_{i=2}^n K_{\sigma_i} \to (\Lambda^1_F)^\perp\subset \Sym_\ratls(F),
\end{equation*}
defined by $v(x_i)_{i=2}^n = (v_{\sigma_i}(x_i))_{i=2}^n$.
By Lemma~\ref{lem:epsilon_sigma}, we have for any $x_i\in K_{\sigma_i}$ and $t\in F$,
\begin{equation}
  \label{eq:38}
  \langle v(x_i)_{i=2}^n, t\otimes t\rangle = \sum_{i=2}^n q_i \langle x_i,
  Q_\sigma(t)\rangle_{K_{\sigma_i}},
\end{equation}
for some positive rationals $q_i$.

Now, identifying $\Ann(\Lambda^1_F)\subset \bS_\ratls(F)$ with $(\Lambda^1_F)^\perp
\subset\Sym_\ratls(F)$ via the trace pairing, the admissibility condition is that for any $x\in
(\Lambda^1_F)^\perp$, if $\langle x, t_i\otimes t_i\rangle \geq 0$ for all $i$, then $\langle x,
t_i\otimes t_i\rangle = 0$ for all $i$.  By \eqref{eq:38}, this is equivalent to the $Q_\sigma(t_i)$
satisfying the no-half-space condition for each $i$.
\qed

\paragraph{Cubic fields.}

We now suppose $F$ is a cubic field.  Define a quadratic map $Q\colon F\to F$ by
\begin{equation*}
  Q(x) = N^F_\ratls(x)/x.
\end{equation*}
In this case, Theorem~\ref{thm:admiss-nonhalf} becomes
\begin{cor}
  \label{cor:admiss-nonhalf}
  Given a totally real cubic field $F$, a weighted boundary stratum with weights $\{t_1, \ldots,
  t_n\}\subset F$ is admissible if and only if $\{Q(t_1), \ldots, Q(t_n)\}\subset F$ satisfies the
  no-half-space condition.
\end{cor}

\begin{proof}
  If $F$ is Galois, this follows directly from Theorem~\ref{thm:admiss-nonhalf}, so suppose $F$ is
  non-Galois with Galois closure $K$.  We may identify $G=\Gal(K/\ratls)$ with the symmetric group
  $S_3$ with $F = K^{(12)}$.  The action of $I$ on $G$ has two orbits, so we need only to check the
  condition of Theorem~\ref{thm:admiss-nonhalf} for a single $\sigma\in G\setminus H$.  Take $\sigma
  = (13)$.  We have $(132)\cdot Q_{(12)}(x) = Q(x)$ for all $x\in F$, thus the two conditions
  coincide.   
\end{proof}

\section{Rationality  and positivity} \label{sec:irrstrata}

In this section, we study in more detail the irreducible strata -- that is, those that
parameterize irreducible stable curves -- in the boundary of the real multiplication locus.
Given a basis $\br = (r_1,\ldots, r_g)$ of a lattice $\mathcal{I}\subset F$, we write
$\mathcal{S}_\br$ for the associated $\mathcal{I}$-weighted boundary stratum, parameterizing
irreducible stable curves having $2g$ nodes with weights $\pm r_1, \ldots, \pm r_g$.  We say that
$\br$ is an \emph{admissible basis} of $\mathcal{I}$ if $\mathcal{S}_\br$ is an admissible stratum
in the sense of \S\ref{sec:bound-eigenf-locus}.

We introduce in this section two additional properties of bases of number fields which we call
\emph{rationality} and \emph{positivity}.  We show that for totally real cubic fields, rationality
and positivity together are equivalent to admissibility.  For higher degree fields, the relation
between these conditions is not clear.  We then show that the rationality and positivity conditions
are necessary for an irreducible stratum to intersect the boundary of the real multiplication locus.
Finally, we give a geometric interpretation of the rationality and positivity conditions in terms of
the geometry of locally symmetric spaces, from which we conclude that there any lattice has only
finitely many rational and positive bases, up to similarity.

\paragraph{Rationality and positivity.}

\par
Consider a basis
$\br = (r_1, \ldots, r_g)$
of a lattice in a totally real number field $F$.  We denote by $(s_i)_{i=1}^g$ the dual basis.  We
say that $\br$ is \emph{rational} if
$$\frac{r_i}{s_i}/\frac{r_j}{s_j} \in \ratls \quad\text{for all $i\neq j$}.$$
We say that $\br$ is \emph{positive} if
$$\frac{r_i}{s_i} \gg 0 \quad\text{for all $i$},$$
where $x\gg 0$ means that $x$ is positive under each embedding $F\to\reals$.

\par
As an intermediate technical notion we say that $\br$ is \emph{weakly positive} 
if
\begin{equation*}
  \frac{r_i}{s_i}/\frac{r_j}{s_j} \gg 0 \quad\text{for all $i\neq j$}.
\end{equation*}
\par
\begin{lemma}
  \label{lem:weak_pos_implies_pos}
  Every weakly positive and rational basis of $F$ is positive.
\end{lemma}
\par
\begin{proof}
  Suppose $(r_i)$ is a basis of $F$ which is weakly positive and rational but not positive.
  We define for each $j$
  \begin{equation*}
    a^{(j)} = \left|\frac{s_1^{(j)}}{r_1^{(j)}}\right|^{1/2}
  \end{equation*}
  and for each $i$, the vectors
  \begin{equation*}
    \tilde{r}_i = (a^{(j)}r_i^{(j)})_{j=1}^g \qtq{and} \tilde{s}_i = (s_i^{(j)}/a^{(j)})_{j=1}^g.
  \end{equation*}
  Note that the bases $(\tilde{r}_i)$ and $(\tilde{s}_i)$ are dual with respect to the standard
  inner product on $\reals^n$.
  For each $i$, define 
  \begin{equation*}
     q_i=\frac{r_i}{s_i}/\frac{r_1}{s_1}.
  \end{equation*}
  By weak positivity and rationality, each $q_i$ is a positive rational.  We then have for each $i$ and
  $j$
  \begin{equation}
    \label{eq:9}
    \tilde{r}_i^{(j)} = \epsilon^{(j)}q_i \tilde{s}_i^{(j)}
  \end{equation}
  with each $\epsilon^{(j)} = \pm 1$.  Since the basis $(r_i)$ is not positive, we must have
  $\epsilon^{(j)}=-1$ for some $j$.  Consider the matrices $R = (R_{ij}) = (r_i^{(j)})$ and $S =
  (S_{ij}) = (s_i^{(j)})$.  Let $D_\epsilon$ be the diagonal matrix with $\epsilon^{(j)}$ the
  $j^{\rm th}$ diagonal entry, and define $D_q$ similarly.  We then have by \eqref{eq:9},
  \begin{equation*}
    S = D_q R D_\epsilon,
  \end{equation*}
  so since $R^t S = I$,
  \begin{equation}
    \label{eq:12}
    R^t D_q R = D_\epsilon^{-1}.
  \end{equation}
  Thus $R$ can be interpreted as an isomorphism between the indefinite quadratic form given by the
  matrix $D_\epsilon^{-1}$ and the definite quadratic form given by $D_q$, which is impossible.
\end{proof}
\par

\begin{prop} \label{prop:nhsequivratpos}
  A basis $(r_1,r_2,r_3)$ of a cubic field $F$ is admissible if and only if
  it is both rational and positive.
\end{prop}

\begin{proof}
  Suppose that the no-half-space condition holds. If the three elements $Q(r_1)$, $Q(r_2)$, $Q(r_3)$
  are $\ratls$-linearly independent, their convex hull cannot contain zero. Since $r_1, r_2, r_3$
  are a basis of $F$, the $Q(r_i)$ cannot all be $\ratls$-multiples. Hence their $\ratls$-span is
  plane. Let $v_i = Q(r_i) \times Q(r_{i+1})$. One calculates that
  $$ v_i =  r_i r_{i+1} s_{i+2} \Delta(r_1,r_2,r_3),$$
  where $\Delta(w_1,w_2,w_3) = \det (w_i^{(j)})$.
  \par
  The no-half-space-condition implies that the $v_i$ are all proportional as elements of $\reals^3$,
  i.e., $v_i/v_j \in \ratls$ when considered as elements of $F$. This is the rationality condition.
  \par
  Moreover, the no-half-space condition implies that the angle between $Q(r_i)$ and $Q(r_{i+1})$ (in
  $\Span(Q(r_i), i=1,2,3)$) is strictly contained in $(0,\pi)$. Thus the $v_i$ are all pointing in
  the same direction. Consequently, the rational number $\frac{r_is_j}{r_js_i}$ is positive. This is
  weak positivity and the preceding lemma concludes one implication.
  \par
  Conversely, suppose that rationality and positivity hold for $\{r_i, r_2, r_3\}$.  The first part
  read backwards implies that the $Q(r_i)$ lie in a plane. If the no-half-space condition fails, we
  have that $v_i/v_j \in \ratls^+$ but $v_i/v_k \in \ratls^-$ for a suitable numbering with
  $\{i,j,k\} = \{1,2,3\}$.  This contradicts weak positivity, and hence positivity.
\end{proof}

\paragraph{Necessity of rationality and positivity.}

Given an irreducible $\mathcal{I}$-weighted boundary stratum $\mathcal{S}_\br$ and a real embedding
$\iota$ of $F$, recall that $\mathcal{S}_\br^\iota\subset\pobarmoduli$ is the stratum of
irreducible stable forms having $2g$ poles of residues $\pm \iota(r_1), \ldots, \pm \iota(r_g)$.

\begin{theorem}
  \label{thm:ratandposdirect}
  Any irreducible stable form in the boundary of $\E^\iota$ is contained in $\mathcal{S}_\br^\iota$
  for some rational and positive basis $\br$ of a lattice $\mathcal{I}\subset F$ whose coefficient
  ring contains $\mathcal{O}$.
\end{theorem}

\begin{proof}
  Consider a family of stable curves $\mathcal{X}\to\Delta$, smooth over $\Delta^*$, the fiber $X_0$
  over $0$ irreducible, of geometric genus $0$, and with real multiplication by $\mathcal{O}$.  We
  label the vanishing cycles of the fiber $X_p$ over $p$ as $\alpha_1, \ldots, \alpha_g$, and we
  choose a family of cycles $\beta_1, \ldots, \beta_g$ (with $\beta_i$ defined only up to Dehn twist
  around $\alpha_i$) such that $(\alpha_i, \beta_i)_{i=1}^g$ is a symplectic basis of $H_1(X_p;
  \zed)$.  As in \S\ref{sec:bound-eigenf-locus}, we may identify as an $\mathcal{O}$-module the
  subspace $V_p\subset H_1(X_p; \zed)$ spanned by the vanishing cycles with some lattice
  $\mathcal{I}$ whose coefficient ring contains $\mathcal{O}$.  Under this identification, the
  $\alpha_i$ correspond to some $r_i\in \mathcal{I}$.  Choose an ordering $\iota_1=\iota, \ldots,
  \iota_g$ of the real embeddings of $F$.  We let $\omega^{(j)}\in\Omega(X_p)$ be the
  $\iota_j$-eigenform determined by
  \begin{equation*}
    \omega^{(j)}(\alpha_i) = r_i^{(j)}.
  \end{equation*}
  We must show that the $r_i$ are a rational and positive basis of $\mathcal{I}$.
  
  The plumbing coordinates from \S\ref{sec:RS} provide holomorphic functions
  $t_i\colon\Delta\to\cx$ which parameterize the opening-up of the $i^{\rm th}$ node of $X_0$.
  Since $X_p$ is nonsingular for $p\neq 0$, each function $t_i$ vanishing only at $0$.  We claim
  that for some positive integers $n_i$,
  \begin{equation}
    \label{eq:13}
    \Im \frac{\omega^{(j)}(\beta_i)}{\omega^{(j)}(\alpha_i)}\sim\frac{n_i}{2\pi}\log\frac{1}{|t_i|},
  \end{equation}
  meaning that the ratio of both sides tends to $1$ as $p\to 0$.

  Denote by $\eta_i\in\Omega(X_p)$ the form with $\eta_i(\alpha_j) = \delta_{ij}$.  We then have
  \begin{equation*}
    \omega^{(j)} = \sum_i r_i^{(j)}\eta_i,
  \end{equation*}
  so after exponentiation, we obtain
  \begin{equation}
    \label{eq:14}
    E\left(\frac{\omega^{(j)}(\beta_i)}{\omega^{(j)}(\alpha_i)}\right) = E(\eta_i(\beta_i))
    \prod_{k\neq i}E\left(\frac{r_k^{(j)}}{r_i^{(j)}}\eta_k(\beta_i)\right).
  \end{equation}
  By Corollary~\ref{cor:period_matrix_boundary}, we have
  \begin{equation}
    \label{eq:24}
    E(\eta_i(\beta_i)) = t_i^{n_i} \phi \qtq{and} E(\eta_i(\beta_j)) = \psi_j
  \end{equation}
  for $\phi$ and $\psi_j$ nonzero holomorphic functions on $\Delta$ and $n_i$ a positive integer
  (equal to the intersection number of $\Delta$ with the boundary stratum where $\alpha_i$ has been pinched).
  Substituting \eqref{eq:24} into \eqref{eq:14} and taking logarithms yields
  \begin{equation*}
    \Im\frac{\omega^{(j)}(\beta_i)}{\omega^{(j)}(\alpha_i)} = \frac{n_i}{2\pi}\log\frac{1}{|t_i|}+O(1),
  \end{equation*}
  from which \eqref{eq:13} follows.

  Since we have identified $V_p$ with $\mathcal{I}$ as $\mathcal{O}$-modules, we also have the
  $\mathcal{O}$-module isomorphism,
  \begin{equation*}
    H_1(X_p; \zed)/V_p\isom\Hom(V_p, \zed)\isom \Hom(\mathcal{I}, \zed)\isom\mathcal{I}^\vee,
  \end{equation*}
  where the first isomorphism arises from the intersection pairing and the last from the trace
  pairing.  Under this isomorphism, the basis $(\beta_i, \ldots, \beta_g)$ of $H_1(X; \zed)/V_p$
  corresponds to the basis $(s_1, \ldots, s_g)$ of $\mathcal{I}^\vee$ which is dual to $(r_1,
  \ldots, r_g)$.  Thus under the action of real multiplication, we have
  \begin{equation*}
    \frac{r_i}{r_k}\cdot\alpha_k = \alpha_i \qtq{and} \frac{s_i}{s_k}\cdot\beta_k = \beta_i \pmod {V_p}.
  \end{equation*}
  From this and \eqref{eq:13}, we then obtain
  \begin{equation}
    \label{eq:15}
    \frac{s_i^{(j)}}{r_i^{(j)}}/\frac{s_k^{(j)}}{r_k^{(j)}} =
    \Im\frac{\omega^{(j)}(\beta_i)}{\omega^{(j)}(\alpha_i)}/ \Im\frac{\omega^{(j)}(\beta_k)}{\omega^{(j)}(\alpha_k)} \sim \log\frac{n_i}{|t_i|}/\log\frac{n_k}{|t_k|}.
  \end{equation}
  Since the right side of \eqref{eq:15} is independent of $j$, so is the left side.  Thus
  $(s_i/r_i)/(s_k/r_k)$ is rational.  The right side of \eqref{eq:15} is also positive for $p\sim
  0$ because $t_\ell(0) = 0$ for all $\ell$, so $(s_i/r_i)/(s_k/r_k)$ is positive as well.
  Therefore this basis is both rational and weakly positive.  By
  Lemma~\ref{lem:weak_pos_implies_pos} the basis is then positive.
\end{proof}

\par

\paragraph{Finiteness of rational and positive bases.}

We now give a geometric interpretation of bases of lattices satisfying the rationality and positivity
conditions as points of intersection of flats in the locally symmetric space $\SL_g(\zed)\lmod
\SL_g(\reals)\rmod\SO_g(\reals)$.  This yields a quick proof that there are only finitely many such
bases up to the action of the unit group.

We recall the classical correspondence between similarity classes of lattices in degree $g$ totally
real number fields and compact flats in $\SL_g(\zed)\lmod \SL_g(\reals)\rmod\SO_g(\reals)$.  Consider a
degree $g$ totally real number field $F$ with an ordering $\iota_1, \ldots, \iota_g$ of the
embeddings of $F$ into $\reals$.  Let $\mathcal{I}$ be a lattice in $F$, which we regard as point in
the space of lattices $\SL_g(\zed)\lmod\SL_g(\reals)$.  Let
$U(\mathcal{I})\subset\mathcal{O}_\mathcal{I}$ be the group of totally positive units $\epsilon$ such that
$\epsilon\mathcal{I}=\mathcal{I}$.  We embed $U(\mathcal{I})$ in the group $D\subset\SL_g(\reals)$ of
positive diagonal matrices by the embeddings $\iota_i$.  By Dirichlet's units theorem, $U(\mathcal{I})$ is a
lattice in $D$.  Let $T(\mathcal{I}) = U(\mathcal{I})\lmod D$, a compact torus.  The stabilizer of
$\mathcal{I}$ under the right action of $D$ on $\SL_g(\zed) \lmod \SL_g(\reals)$ is $U(\mathcal{I})$, so
we obtain an immersion $p_\mathcal{I}\colon T(\mathcal{I})\to \SL_g(\zed)\lmod
\SL_g(\reals)\rmod\SO_g(\reals)$ of $T(\mathcal{I})$ as a compact flat in $\SL_g(\zed)\lmod
\SL_g(\reals)\rmod\SO_g(\reals)$.  Since similar lattices lie on the same $D$-orbit, this associates a
compact flat to each similarity class of lattices.

Let $\Rec\subset\SL_g(\zed)\lmod \SL_g(\reals)\rmod\SO_g(\reals)$ be the locus of lattices in $\reals^g$
which have a basis of orthogonal vectors, a closed (but not compact) flat isometric to
$\reals^g/C_g$, where $C_g\subset\SO_g(\reals)$ is the group of symmetries of the cube.

\begin{theorem}
  \label{thm:SLgFin}
  For each lattice $\mathcal{I}$ in a totally real number degree $g$ number field $F$, the flat
  $p_\mathcal{I}(T(\mathcal{I}))$ intersects $\Rec$ transversely.  There is a natural bijection
  between the set $p_\mathcal{I}^{-1}(\Rec)$ and the set of rational and positive bases of
  $\mathcal{I}$, up to the action of $U(\mathcal{I})$, changing signs, and reordering.
\end{theorem}

\begin{proof}
  Let $\tRec\subset\SL_g(\reals) \rmod \SO_g(\reals)$ be the image of the diagonal orbit of the standard
  basis of $\reals^g$, a lift of $\Rec$ to $\SL_g(\reals) \rmod \SO_g(\reals)$.

  Lifts of $T(\mathcal{I})$ to $\SL_g(\reals) \rmod \SO_g(\reals)$ correspond to oriented bases of
  $\mathcal{I}$ up to the action of the unit group by associating the flat $(r_i^{(j)})\cdot
  D\cdot\SO_g(\reals)$ to the basis $r_1, \ldots, r_g$.  Points of $p_\mathcal{I}^{-1}(\Rec)$
  correspond bijectively (up to the action of the group $C_g\subset\SL_g(\zed)$ of symmetries of the
  cube) to intersection points of $p_\mathcal{I}(T(\mathcal{I}))$ with $\Rec$.  Note that if a lift
  $F$ intersects $\tRec$, then so does the lift $\gamma\cdot F$ for any $\gamma\in C_g$.  These
  intersection points correspond to the same point in $p_\mathcal{I}^{-1}(\Rec)$, and on the level
  of bases, replacing $F$ with $g\cdot F$ corresponds to reordering and changing signs in the basis
  $(r_i)$.

  We must show that $(r_i^{(j)})\cdot D\cdot\SO_g(\reals)$ intersects $\tRec$ if and only if $(r_i)$
  is rational and positive.  Note that the rationality and positivity conditions make sense for
  bases of $\reals^n$, with the $j^{\rm th}$ embedding $r_i^{(j)}$ interpreted as the $j^{\rm th}$
  coordinate of the vector $r_i$.  A vector is regarded as rational if all of its coordinates are
  equal, totally positive if all of its coordinates are positive, and so on.  With this
  interpretation, a orthogonal basis $(r_1, \ldots, r_n)$ of $\reals^n$ is rational and positive,
  since the basis is orthogonal if and only if each dual vector $s_i$ is a positive multiple of the
  corresponding $r_i$.  The rationality and positivity conditions are invariant under the action of
  $D$, thus any basis $(r_i)$ whose $D$-orbit meets $\tRec$ is rational and positive.
  
  Now suppose the basis $(r_i)$ of $\mathcal{I}$ is rational and positive.  Let $(s_i)$ be the dual
  basis.  For each $j$, let $a^{(j)} = \sqrt{s_1^{(j)}/r_1^{(j)}}.$ Let $A$ be the diagonal matrix
  $(a^{(1)}, \ldots, a^{(g)})$, and let
  \begin{equation*}
    (\tilde{r}_i^{(j)}) = (a^{(j)}r_i^{(j)}) \qtq{and} (\tilde{s}_i^{(j)}) = (s_i^{(j)}/a^{(j)}).
  \end{equation*}
  Note that $(\tilde{s}_i^{(j)})$ is the dual basis to $(\tilde{r}_i^{(j)})$, and $A$ is the unique
  diagonal matrix for which the vectors $(\tilde{r}_1^{(j)})$ and $(\tilde{s}_1^{(j)})$ are
  positively proportional.  If the positivity and rationality conditions are satisfied, we have
  \begin{equation*}
    \frac{\tilde{s}_i^{(j)}}{\tilde{r}_i^{(j)}} = \frac{1}{(a^{(j)})^2}\cdot\frac{s_i^{(j)}}{r_i^{(j)}}
    = \frac{q_i}{(a^{(j)})^2}\cdot\frac{s_1^{(j)}}{r_1^{(j)}} = q_i
  \end{equation*}
  for some positive $q_i\in\ratls$.  Since each $\tilde{s}_i$ is proportional to $\tilde{r}_i$, the
  basis $(\tilde{r}_i)$ of $\reals^g$ is rectangular, so it is the unique intersection point of
  $(r_i^{(j)})\cdot D \cdot\SO_g(\reals)$ and $\tRec$.  Otherwise for some $i$ the vectors
  $(\tilde{r}_i^{(j)})$ and $(\tilde{s}_i^{(j)})$ are not proportional, so the flats are disjoint.
  Since we saw that there was at most one intersection point between each lift of the two flats,
  these intersection points are transverse.
\end{proof}

\begin{cor}
  \label{cor:rat_pos_finite}
  The set of bases of $\mathcal{I}$ satisfying the rationality and positivity conditions is finite,
  up to the action of $U(\mathcal{I})$
\end{cor}

\begin{proof}
  Since $T(\mathcal{I})$ is compact, there are at most finitely many intersection points with $\Rec$
  by transversality.
\end{proof}

\section{Boundary of the eigenform locus: Sufficiency for genus three}
\label{sec:suffg3}

In this section we specialize to genus three. We prove that the boundaries of $\RM$ and
$\E^\iota$ are indeed the unions of the varieties described in Theorem~\ref{thm:boundary_nec}.
Moreover, we show how to derive these subvarieties explicitly from the weights of a boundary
stratum.
\par

\paragraph{Boundary strata in genus three.}

The topological type of a geometric genus zero stable curve (or a weighted boundary stratum) can be
encoded by a graph where each vertex represents an irreducible component and an edge joining two
vertices (or possibly joining a vertex to itself) represents a node at the intersection of those two
components.  There are fifteen topological types of arithmetic genus three, geometric genus zero
stable curves, shown in Figure~\ref{fig:graphs}.  We will refer to a stable curve represented by the
$j^{\rm th}$ graph in the $i^{\rm th}$ row of Figure~\ref{fig:graphs} as a \emph{type $(i,j)$ stable curve}.

An $\mathcal{I}$-weighted stable curve can be represented by a graph together with a direction and a weight
$r\in \mathcal{I}$ attached to each edge $e$.  The cusp on the component represented by the vertex at the front of $e$ has weight $r$,
and the other cusp has weight $-r$.

\begin{figure}[htbp]
  \centering
  \includegraphics{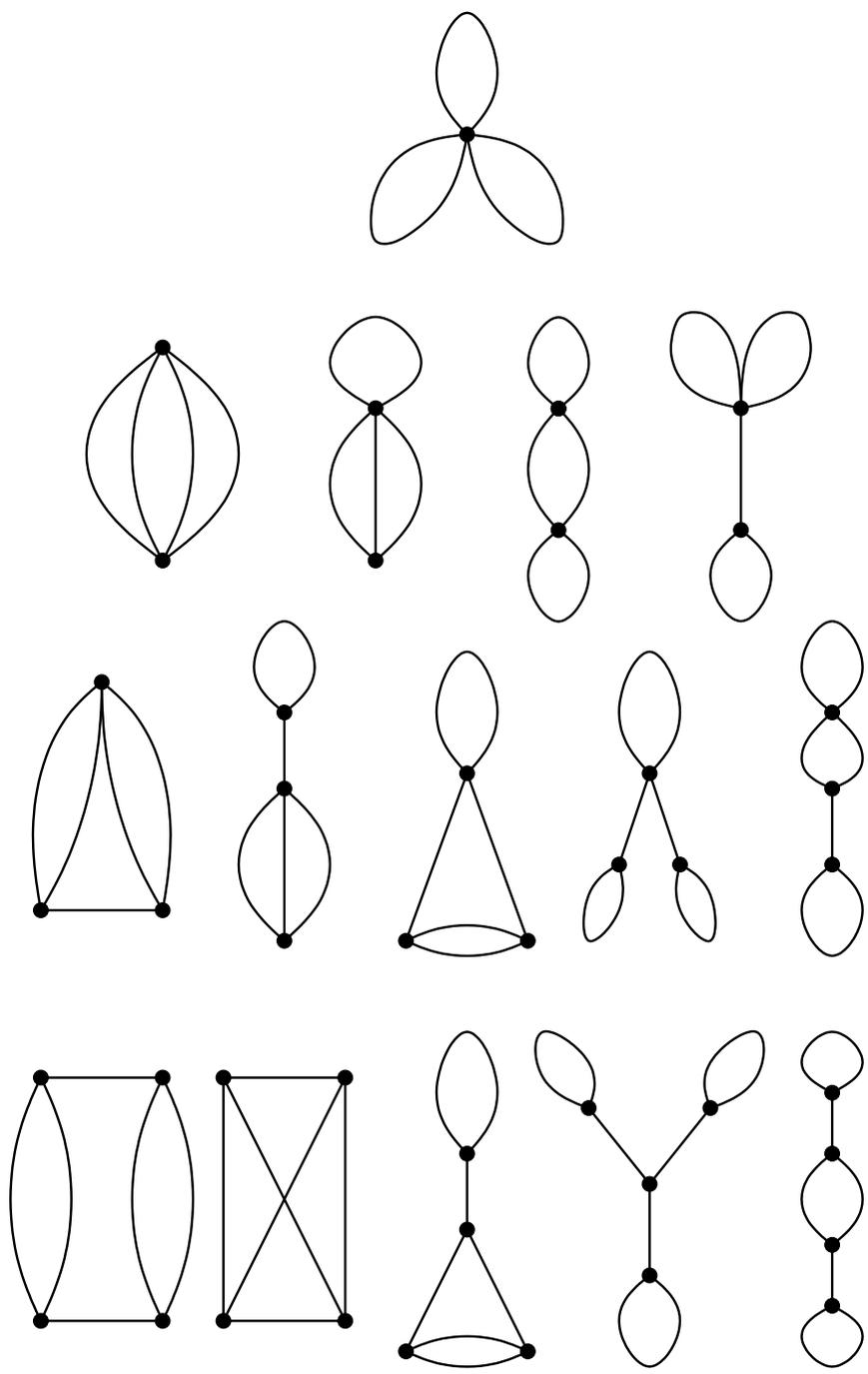}
  \caption{Genus three, geometric genus zero stable curves}
  \label{fig:graphs}
\end{figure}

It will be convenient have a compact notation for boundary strata without
separating curves, the only ones which will be important in the sequel.
For all but one of these strata the components of
the corresponding stable curves can be arranged in chain or one loop.  We code those boundary strata
in the following way: we write $[m_i]$ for a genus zero component of the stable curve with $m_i$ marked
points. We write $\times^{a_i}$ for the number of intersection points with the subsequent curve.
The possible patterns for curve systems without separating curves include $[6]$, $[m_1] \times^a
[m_2]$, $[m_1] \times^{a_1} [m_2] \times^{a_2} [m_3]$ or $[m_1] \times^{a_1} [m_2] \times^{a_2}
[m_3] \times^{a_3} $.  In the last pattern, $a_3$ is the number of nodes joining the last and the
first component.  For example, a $[5]\times^3[3]$ boundary stratum is represented by graph $(2, 2)$
in Figure~\ref{fig:graphs} and a $[4]\times^2[3]\times^1[3]\times^2$ boundary stratum is represented
by graph $(3,1)$.

Boundary strata of type $[6]$ parameterize irreducible stable curves with three
nonseparating nodes, often called ``trinodal curves.''

\par

\begin{theorem}
  \label{thm:boundary_suff}
  Consider an order $\mathcal{O}$ in a totally real cubic number field $F$, a real embedding
  $\iota$ of $F$, and a cusp packet $(\mathcal{I}, T)\in\mathcal{C}(\mathcal{O})$.  The closure in
  $\pobarmoduli$ of
  the cusp of $\E^\iota$ associated to $(\mathcal{I}, T)$ is equal to the union over all
  admissible $\mathcal{I}$-weighted boundary strata $\mathcal{S}$ of the varieties
  $\mathcal{S}^\iota(T)$.

  The closure of the corresponding cusp of $\RM$ in $\barmoduli$ is equal to the union over all
  $\mathcal{I}$-weighted boundary strata $\mathcal{S}$ of the images of the $\mathcal{S}(T)$ under
  the forgetful map to $\barmoduli$.
\end{theorem}
\par
After some preliminary discussions, we prove Theorem~\ref{thm:boundary_suff} at the end of this section.

Since the intersection of two algebraic subvarieties of $\barmoduli[3]$ has a finite
number of components, we obtain the following generalization for genus three 
of Theorem~\ref{thm:SLgFin}.
\par
\begin{cor} \label{cor:noweightedfinite}
  Given a lattice $\mathcal{I}$ in a cubic number field $F$, the number of
  $\mathcal{I}$-weighted admissible boundary strata up to similarity is finite.
\end{cor}
\par 
We will discuss in Appendix~\ref{sec:bdcounting} various aspects concerning enumerating and
counting this set of admissible weighted boundary strata.
\par
In order to make Theorem~\ref{thm:boundary_suff} completely explicit, we will now give coordinates
on some weighted boundary strata in terms of cross-ratios and give explicit equations cutting out
the subvarieties $\mathcal{S}(T)$.
\par
We say that a weighted boundary stratum
$\mathcal{S}_1$ is a {\em degeneration} of  $\mathcal{S}_2$, if $\mathcal{S}_1$
is obtained by pinching a collection of curves on a surface represented by $\mathcal{S}_2$.
We also say that $\mathcal{S}_2$ is an {\em undegeneration} of  $\mathcal{S}_1$
in this situation.

\paragraph{Irreducible strata.}

Consider an irreducible stratum (that is, type $[6]$ if we are in genus three) $\mathcal{S}_\br$.  A
weighted stable curve parameterized by $\mathcal{S}_{\br}$ is determined $2g$ distinct points $p_1,
\ldots, p_g$ and $p_{-1}, \ldots, p_{-g}$ on $\proj^1$ with weights $r_i$ at $p_i$ and $-r_i$ at
$p_{-i}$, so $\mathcal{S}_\br\isom\moduli[0,2g]$.  For $j\neq k$ we define the cross-ratio morphisms
$R_{[jk]}\colon\mathcal{S}_\br\to\cx\setminus\{0,1\}$ by
\begin{equation} \label{CR6}
R_{[jk]}=[p_j,p_{-j}, p_{-k}, p_k].
\end{equation}
where for $z_1, \ldots z_4\in\cx$,  $$[z_1,z_2, z_3, z_4]=\frac{(z_1-z_3)(z_2-z_4)}{(z_1-z_4)(z_2-z_3)}.$$
\par
Take $(s_1,\ldots,s_g)$ to be the dual basis of $\mathcal{I}^\vee$ (with respect to the
trace pairing) to $(r_1,\ldots,r_g)$.   We can now make the cross-ratio map $\CR$ 
defined in~\eqref{eq:defCR} more explicit.
\par
\begin{prop}
  \label{prop:CRpsisjsk}
  The elements $s_j \otimes s_k$ for $j\neq k$ form a basis of $N(\mathcal{S}_{\br})$.  Moreover we
  have $\Psi(s_j\otimes s_k) = R_{[jk]}$ as functions on $\mathcal{S}_\br$.
\end{prop}
\par
\begin{proof}
  That $s_j \otimes s_k$ belongs to $N(\mathcal{S}_{\br})$ follows from the definition
  of the dual basis with respect to the trace pairing.  They are obviously linearly independent and
  thus a basis by a dimension count
  \par
  We normalize a point $P = (p_{-g}, \ldots, p_g)$ of $\mathcal{S}_\br$ by a \Mobius transformation so
  that  $p_j=0$, $p_{-j} = \infty$ and $p_k = 1$. By definition of $\Psi(s_j\otimes s_k)$
  we must choose the stable one-form $\omega$ on $\proj^1$ with residue $\pm \trace(s_j r_m)/2\pi i$ at
  the point $p_{\pm m}$, i.e. \ we have to choose $\omega = dz/2\pi iz$. We then integrate this function
  over the path whose intersection with the loop around the node at $p_{\pm m}$ is $\trace(s_k
  r_m)$.  On $\proj^1$, this is a path $\gamma$ joining $p_k=1$ to $p_{-k}$.
  We then have
  \begin{equation*}
    \Psi(s_j\otimes s_k)(P) = e^{2\pi i\int_\gamma\omega} = p_{-k} =R_{[jk]}(P).\qedhere
  \end{equation*}
\end{proof}
\par
\begin{cor}
  \label{cor:CRarecoord6}
  For $g=3$, after identifying $\Hom(N(\mathcal{S})\cap\bS_\zed(\mathcal{I}^\vee), \cx^*)$ with
  $(\cx^*)^3$ via the basis $(s_1\otimes s_2, s_2\otimes s_3, s_3\otimes s_1)$ of
  $N(\mathcal{S})\cap\bS_\zed(\mathcal{I}^\vee)$, the map $\CR$ becomes
  \begin{equation*}
    \CR = (R_{[12]}, R_{[23]}, R_{[31]}) \colon\mathcal{S}_\br\to (\cx\setminus\{0,1\})^3.
  \end{equation*}
  The map $\CR$ is a two-to-one branched cover which identifies orbits of the involution
  $i\colon\mathcal{S}_\br\to\mathcal{S}_\br$ which exchanges each pair $p_i$ and $p_{-i}$.
\end{cor}
\par
\begin{proof}
  That $\CR$ is of this form follows immediately from the definition of $\CR$ and
  Proposition~\ref{prop:CRpsisjsk}.
  \par
  That the map $\CR = (R_{[12]}, R_{[13]}, R_{[23]})$ is two-to-one onto its image can be
  checked three of the $p_i$ and solving for the rest.  Interchanging each $p_i$ and $p_{-i}$
  leaves each cross-ratio $R_{[jk]}$ invariant, so $\CR$ is the quotient map by this involution. 
\end{proof}

\paragraph{Type $[4]\times^4[4]$ strata.}
\label{page:4times4}

Consider an $\mathcal{I}$-weighted stable curve $X$ of type $[4]\times^4[4]$ having weights $r_1,
\ldots, r_4\in \mathcal{I}$ with $\sum r_i = 0$, and let $\mathcal{S}$ be the corresponding
$\mathcal{I}$-weighted boundary stratum.  We name $u_1, \ldots, u_4$ the four points on one
irreducible component with weight $r_1, \ldots, r_4$ and name $v_1, \ldots, v_4$ the opposite points
on the other irreducible component.  We define the cross-ratios,
\begin{equation*}
R_u = [u_1,u_2,u_3,u_4] \quad \text{and} \quad R_v=[v_1,v_2,v_3,v_4].
\end{equation*}
$\mathcal{S}$ is isomorphic to $\moduli[0,4]\times\moduli[0,4]$ with $R_u$ and $R_v$ coordinates on
the first  and second factors.

\paragraph{Type $[4]\times^2[4]$ strata.}

Now consider the $\mathcal{I}$-weighted stable curve shown in Figure~\ref{fig:44graph} with distinct weights
$r_1, r_2, r_3\in \mathcal{I}$, and let $\mathcal{S}$ be the corresponding $\mathcal{I}$-weighted boundary stratum.  We
label by $p_1, p_{-1}, p_2, p_{-2}$ the points on one irreducible component with weights $r_1, -r_1,
r_2, -r_2$ and label by $q_1, q_{-1}, q_2, q_{-2}$ the points on the other irreducible component
with weights $r_3, -r_3, -r_2, r_2$.  The stratum $\mathcal{S}$ is isomorphic to
$\moduli[0,4]\times\moduli[0,4]$ with coordinates
\begin{equation}
  \label{CR424}
  R_1=[q_1,q_{-1}, q_{-2}, q_2] \quad \text{and} \quad R_3=[p_1,p_{-1}, p_{-2}, p_2].
\end{equation}

The stratum $\mathcal{S}$ arises as a degeneration of the irreducible weighted boundary stratum with
weights $r_1, r_2, r_3$ by pinching a curve around the points of weights $r_1, -r_1, r_2$.  As
this curve is pinched, the cross-ratio $R_{[1,3]}$ tends to $1$.

\begin{figure}[htbp]
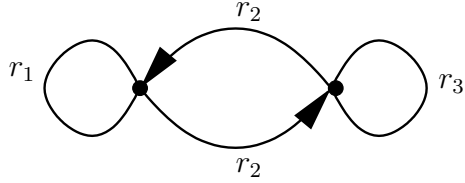

  \centering
  \ifpdf
  \input{44graph.pdftex_t}
  \else
  \input{44graph.pstex_t}
  \fi
  \caption{Type $[4]\times^2[4]$ $\mathcal{I}$-weighted stable curve}
  \label{fig:44graph}
\end{figure}

\paragraph{Calculation of $\mathcal{S}(T)$.}

We will write $R_i$ for $R_{[jk]}$ where $\{i,j,k\} = \{1,2,3\}$
and we let $(s_1,s_2,s_3)$ be the dual basis to $(r_1,r_2,r_3)$.

Whether $\mathcal{S}(T) = \mathcal{S}$ or not will depend on the following notion.  Given an
$\mathcal{I}$-weighted boundary stratum $\mathcal{S}$, we let $\Span(\mathcal{S})\subset\ratls^3$ denote
the $\ratls$-span of $\{Q(r): r\in\Weight(\mathcal{S})\}$, and let $\codim(\Span(\mathcal{S}))$ denote the
codimension of $\Span(\mathcal{S})$ in $\ratls^3$.

\begin{theorem}
  \label{thm:explCREQ}
  The locus $\mathcal{S}(T)$ is defined by the following equation.
  \begin{compactitem}
  \item{Case $[6]$:} For a boundary stratum of type $[6]$, we use the cross-ratio coordinates $R_1,
    R_2, R_3$ defined in Proposition~\ref{prop:CRpsisjsk}.  Then the subvariety $\mathcal{S}(T)$
    of the admissible boundary stratum $\mathcal{S}_{(r_1,r_2,r_3)}$ is given by the cross-ratio
    equation
    \begin{equation}
      \label{eq:CREQ} \prod_{i=1}^3 R_i^{a_i}= \zeta,
    \end{equation}
    where the $a_i$ are the unique (up to sign) relatively prime integers such that $a_i = t b_i$
    for some $t\in F$, and
    \begin{equation}
      \label{eq:1}
      b_i = N_\ratls^F(r_i)\left(\frac{s_i}{r_i}\right)^2,
    \end{equation}
    and where $\zeta$ is the root of unity $\zeta = e^{2\pi i u}$ with 
    \begin{equation}
      \label{eq:21}
      u = \langle T,\sigma \rangle,
    \end{equation}
    where
    \begin{equation}
      \label{eq:22}
      \sigma =  \sum_{i=1}^3 a_i s_{i+1} \otimes s_{i+2}.
    \end{equation}
    Here we interpret the extension class $T$ as an element of $\Sym_\ratls(F)$.
  \item{Case $[4]\times^2 [4]$:} The subvariety $\mathcal{S}(T)$ of the admissible boundary
    stratum with weights $\{r_1,r_2,r_3,r_4=-r_2\}$ is given, using the cross-ratio coordinates
    defined above, by
    \begin{equation} \label{eq:CR4ti24} R_1^{a_1} R_3^{a_3}= \zeta,
    \end{equation}
    where the $a_i$ and $\zeta$ are calculated from $\{r_1,r_2,r_3\}$ as in the preceding
    case $[6]$.
  \item{Case $[4]\times^4 [4]$:} There are two possibilities.  If $\codim(\Span(\mathcal{S}))=0$,
    then $\mathcal{S}(T)$ is the whole stratum.  If $\codim(\Span(\mathcal{S}))=1$, then
    $\mathcal{S}$ is a degeneration of an admissible irreducible weighted boundary stratum
    $\mathcal{S}_{(r_1,r_2,r_3)}$ with the property that exponents $a_i$ defined above satisfy
    $\sum_{i=1}^3 a_i =0$. Moreover, $\mathcal{S}(T)$ is cut out by the equation
    \begin{equation} \label{eq:CRstr444} (R_uR_v)^{a_1} \cdot
      \left(\frac{R_u}{1-R_u}\frac{R_v}{1-R_v}\right)^{a_3}= \zeta,
    \end{equation}
    where $\zeta$ is as in the case $[6]$.
  \end{compactitem}
  This is a complete list of the cases of boundary strata without separating curves, where for some
  admissible boundary stratum $\mathcal{S}$, we can have $\mathcal{S}(T) \subsetneq \mathcal{S}$.
\end{theorem}
\par
We will refer to the equations stated in the above theorem as the \emph{cross-ratio
equations}.
\par
The following lemmas determine the possibilities for $\codim(\Span(\mathcal{S}))$.
\par
\begin{lemma}
  \label{lemma:linindepQ} Suppose that the $\ratls$-span of $r_1,r_2, r_3 \in F \setminus\ratls$ is
  two-dimensional. Then $Q(r_1),Q(r_2)$, and $Q(r_3)$ are $\ratls$-linearly independent.
\end{lemma}
\par

\begin{proof}
  Embedding $F$ in $\reals^3$ by its three real embeddings, the map $Q$ becomes
  \begin{equation*}
    Q(x, y, z) = (yz, xz, xy),
  \end{equation*}
  which we regard as a degree two map $Q\colon\proj^2(\reals)\to\proj^2(\reals)$.  Suppose the
  $Q(r_i)$ are $\ratls$-linearly dependent.  They then lie on a line $L\subset\proj^2(\reals)$
  cut out by an equation $a_1x+a_2y+a_3z=0$ with each $a_i\in\ratls$.  Each coefficient $a_i$ of this
  equation must be nonzero, for if (say) $a_3$ were zero, then no irrational $s\in F$ could lie on $L$, since
  the equation $a_1 s^{(1)} + a_2 s^{(2)}=0$ implies $s\in\ratls$.

  The inverse image $f^{-1}(L)$ is  a nonsingular conic, so it intersects any line in at most two
  points.  Thus if the $r_i$ were $\ratls$-linearly dependent, they could not map to $L$.
\end{proof}
\par

\begin{lemma} \label{lemma:conedims}
If the stratum $\mathcal{S}$ is irreducible or if it is of type $[4] \times^2 [4]$, 
then $\codim(\Span(\mathcal{S})) =1$.
If it is of type  $[4] \times^4 [4]$, then either $\codim(\Span(\mathcal{S})) =0$
or $\codim(\Span(\mathcal{S})) =1$.
In all of the remaining cases, $\codim(\Span(\mathcal{S})) =0$.
\end{lemma}

\begin{proof}
  Since the set of weights contains a $\ratls$-basis of $F$, $\codim(\Span(\mathcal{S}))$ is at most
  $1$.  The preceding Lemma~\ref{lemma:linindepQ} implies that $\codim(\Span(\mathcal{S}))=0$
  whenever the curves $\mathcal{S}$ contains a component isomorphic to a thrice-punctured $\proj^1$.
  The only remaining cases are the irreducible stratum and strata of type $[4] \times^2 [4]$. In
  either case there are only three distinct weights.  We only need to remark that three vectors
  cannot span $\reals^3$ and contain $0$ in its convex hull at the same time.
\end{proof}
\par
We will show in Example~2 of Appendix~\ref{sec:bdcounting}
that this is a complete list of constraints.
\par

\begin{lemma}
  \label{lemma:cones}
  Suppose that $\{P_i\}_{i=1}^k$ are $k$ points in $\reals^n$, $k \geq n+2$, whose $\reals^+$-span
  is all of $\reals^n$ and such that no $n$ of the $P_i$ are
  contained in a subspace of dimension $n-1$.  Then there are $n+1$ points among the $P_i$, whose
  $\reals^+$-span also is all of $\reals^n$.
\end{lemma}
\par
\begin{proof}
  Given $k \geq n+2$ points $P_i$ in $\reals^n$ whose convex hull contains
  zero, we must show that there are $k-1$ among them whose convex hull still contains zero.
  The hypothesis on the span of subsets of $n+1$ elements will then imply that these vectors span
  $\reals^n$, and the claim follows from induction on $k$.
  \par
  Consider the linear map $f$ that assigns to $x \in \reals^k$ the sum $f(x) = \sum_{i=1}^{k} x_i
  P_i$. The hypothesis implies that $K = {\rm Ker}(f)$ contains $w = (w_1,\ldots,w_k)$ with
  $\sum_{i=1}^k w_i =1$ and $w_i >0$.  Since $\dim(K) >2$ there is also $0 \neq y \in K$ with $\sum
  y_i =0$.  The affine space $w + \lambda y$ has to intersect the coordinate hyperplanes at some
  point different from zero. This point yields a convex combination of zero with at most $k-1$
  summands.
\end{proof}
\par
\paragraph{Proof of Theorem~\ref{thm:explCREQ}.}
  We start with case $[6]$.  Recall that $\mathcal{S}(T)\subset\mathcal{S}$ is the subvariety cut
  out by the equations
  \begin{equation}
    \label{eq:25}
    \Psi(a) = e^{-2\pi i \langle T, a\rangle},
  \end{equation}
  as $a$ ranges in $N(\mathcal{S})\cap\Ann(\Lambda^1)\cap \bS_\zed(\mathcal{I}^\vee)$.  By
  Lemma~\ref{lemma:conedims}, this is a rank-one $\zed$-module, so by
  Proposition~\ref{prop:CRpsisjsk}, it is generated by $\sum_{i=1}^3 a_i s_{i+1}\otimes s_{i+2}$ for
  some relatively prime integers $a_i$, and the equation \eqref{eq:25} is simply \eqref{eq:CREQ}
  with $\zeta$ as in \eqref{eq:21}.  To find the $a_i$, we will find some rationals $b_i$ with $\sum
  b_i s_{i+1}\otimes s_{i+2}\in\Ann(\Lambda_1)$, and the $a_i$ will be a primitive integral
  multiple.
  \par
  If $b_i \in \ratls$, then $\sum b_i s_{i+1} \otimes s_{i+2} \in \Ann(\Lambda^1) $ if and only if
  $$ \trace\left(\sum_{i=1}^3 b_i s_{i+1}s_{i+2} x \right) = \left\langle \sum_{i=1}^3 b_i s_{i+1} \otimes s_{i+2}, \sum_{j=1}^3 r_j \otimes s_jx\right\rangle =0$$
  for all $x \in F$, thus if and only if $\sum b_i s_{i+1} s_{i+2} = 0$.
  \par
  If we let $\tilde{b}_i = N(r_i) \frac{s_i}{r_i}$ and take $c_i$ satisfying $\sum c_i/r_i = 0$ then
  we have
 $$ \sum_{i=1}^3 \tilde{b}_i \frac{c_i}{N(r_i)} s_{i+1} s_{i+2} = 0.$$
 From Lemma~\ref{lem:cipropNsr} below, we deduce that $(\tilde{b}_i \frac{c_i}{N(r_i)})_{i=1}^3$ is
 proportional to the $b_i$ as in the statement.  Thus the exponents in the cross-ratio equation are
 proportional to the $b_i$ as claimed.
 \par
 We next treat the case of a stratum $\mathcal{S}$ of type $[4] \times^4 [4]$.  As explained above
 along with the cross-ratio coordinates, this case is a degeneration of a boundary stratum of type $[6]$.
 Since $\Span(\mathcal{S})$ here is the same as for $\mathcal{S}_{(r_1,r_2,r_3)}$ we obtain the same
 equation, only the cross-ratio $R_2$ is equal to one identically.
 \par
 It remains to treat the case of a boundary stratum $\mathcal{S}$ of type $[4] \times^4 [4]$ in the
 case $\dim(\Span(\mathcal{S}))=2$. Lemma~\ref{lemma:cones} implies that $\mathcal{S}$ is a
 degeneration of some admissible stratum of type $[6]$, say $\mathcal{S}_{(r_1,r_2,r_3)}$
 given a suitable numbering of the weights.
 \par
 Next we show that $\sum a_i = 0$. Admissibility implies  that \eqref{eq:sumci}
 below holds for some $c_i \in \ratls$.  The hypothesis on the dimension of the span implies the
 equation~\eqref{eq:sumci} and
 $$ \frac{1}{r_1+r_2+r_3} = \frac{e_1}{r_1} + \frac{e_2}{r_2}$$
 for some $e_1,e_2 \in \ratls$. We may moreover rescale such that $r_1=1$ and solve the system for
 cubic equations killing $r_2$ and $r_3$ respectively.  These equations must be the minimal
 polynomials of $r_2$ and $r_3$. We obtain
 $$N^F_\ratls(r_2) =  -\frac{c_2e_2}{c_1e_1}\quad \text{and} \quad 
 N^F_\ratls(r_3) = \frac{c_3^2e_2}{c_2c_1e_1 - c_1^2e_2}.$$ Using the Corollary~\ref{cor:2ndver} to
 the calculations in case $[6]$ below, we only need to check that $\sum c_i^2/N^F_\ratls(r_i) =0$,
 which is obvious.
 \par
 We may normalize the degeneration from the boundary stratum $\mathcal{S}_{(r_1,r_2,r_3)}$ to
 $\mathcal{S}$ as follows.  Let $p_1 = 0$, $p_2 = 1$, $p_3 = \infty$ and let the $p_{-i}$ all
 converge to the same point $\mu$, that is, $p_{-i}= \mu + \lambda_i t$ with $t \to 0$. Then
 $$ R_u= \frac{\mu-1}{\mu}, \quad R_v = 
 \frac{\lambda_1-\lambda_3}{\lambda_2 -\lambda_3}$$ and in the limit as $t \to 0$
 $$ R_1/R_2 = \frac{\mu-1}{\mu}\frac{\lambda_1-\lambda_3}{\lambda_2 -\lambda_3}, \quad R_3/R_2 =
 (1-\mu)\frac{\lambda_1-\lambda_3}{\lambda_1 -\lambda_2}.$$ Thus the cross-ratio equation
 $$ (R_1/R_2)^{a_1} \cdot (R_3/R_2)^{a_3} = \zeta$$
 for $\mathcal{S}_{(r_1,r_2,r_3)}$ becomes
 $$ (R_uR_v)^{a_1} \cdot \left(\frac{R_u}{1-R_u}\frac{R_v}{1-R_v}\right)^{a_3}= \zeta,$$
 as we claimed.
 \par
 The last statement in an immediate consequence of Lemma~\ref{lemma:conedims}.
 \qed
 \par
We give here the lemma needed above and as corollary a second version of calculating the
exponents of the cross-ratio equation. Using the no-half-space condition, there are rational
coefficients $c_i$ such that
\begin{equation}
  \label{eq:sumci}
  \frac{c_1}{r_1} + \frac{c_2}{r_2} + \frac{c_3}{r_3} = 0.
\end{equation}
\par
\begin{lemma}
  \label{lem:cipropNsr}
  If the $r_i$ and $c_i$ are as in \eqref{eq:sumci}, then the
  triple $(c_1,c_2,c_3)$ is proportional to $(N(r_i) {s_i}/{r_i})_{i=1}^3$.
\end{lemma}
\par
\begin{proof} Note that the triple $(N(r_i) {s_i}/{r_i})_{i=1}^3$ is (up to a factor $r_1/s_1$)
integral by rationality. It thus suffices to check that
  $$ \sum_{i=1}^3 \left(N(r_i) \frac{s_i}{r_i}\right) \cdot \frac{1}{r_i} = 0.$$
  We have
  \begin{align}
    \notag \sum_{i=1}^3 \left(N(r_i) \frac{s_i}{r_i^2}\right) \cdot \frac{r_1}{s_1} &=
    \sum_{i=1}^3 r_i^{(2)} r_i ^{(3)}\frac{s_i^{(1)}}{r_i^{(1)}}\frac{r_1^{(1)}}{s_1^{(1)}}\\
    \notag
    &= \sum_{i=1}^3 r_i^{(2)} r_i ^{(3)}\frac{s_i^{(2)}}{r_i^{(2)}}\frac{r_1^{(2)}}{s_1^{(2)}}
    \quad\text{(by rationality)}\\
    \label{eq:16}
    &= \frac{r_1^{(2)}}{s_1^{(2)}}\sum_{i=1}^3s_i^{(2)}r_i^{(3)}
  \end{align} Consider the $3$ by $3$ matrices $R = (r_i^{(j)})$ and $S = (s_i^{(j)})$.  Since the
  bases $(r_i)$ and $(s_i)$ are dual, we have $RS^t=I$.  Thus $S^tR=I$ as well, and \eqref{eq:16} is $0$.
\end{proof}

\begin{cor}
  \label{cor:2ndver}
  The exponents $a_i$ appearing in the cross-ratio equation \eqref{eq:CREQ} are the unique (up to
  sign) relatively prime integers with $a_i = t b_i'$ for some $t\in F$ and
  $$ b_i' = c_i^2/N^F_\ratls(r_i).$$
\end{cor}

\paragraph{Period coordinates.}

In preparation for the proof of Theorem~\ref{thm:boundary_suff}, we now define local coordinates
around certain Lagrangian boundary strata $\mathcal{S}\subset\barmoduli[3](L)$ in terms of exponentials
 of entries of period matrices.

Let $\mathcal{S}\subset\barmoduli[3](L)$ be a Lagrangian boundary stratum obtained by pinching
curves $\gamma_1, \ldots, \gamma_m$ on $\Sigma_3$.  We say that such a boundary stratum is
\emph{nice} if the complement of any two of the $\gamma_i$ is connected.  There are five topological
types of nice boundary strata in $\barmoduli[3](L)$, representing stable curves of type $(1,1)$,
$(2, 1)$, $(2, 2)$, $(3,1)$ and $(4,2)$.

Let $\alpha_i\in L\subset H_1(\Sigma_3; \zed)$ denote the homology class of $\gamma_i$ after
choosing an orientation.
\begin{lemma}
  \label{lem:nice_lin_indep}
  If $\mathcal{S}\subset\barmoduli[3](L)$ is nice boundary stratum, then there are elements
  $\sigma_1, \ldots, \sigma_n\in \Hom(L, \zed)$ such that
  \begin{equation}
    \label{eq:17}
    \langle\sigma_i,\alpha_j\otimes\alpha_j\rangle=\delta_{ij}.
  \end{equation}
\end{lemma}

\begin{proof}
  We represent a curve in $\mathcal{S}$ by a directed graph $G$ with the edges weighted by elements
  of $L$.  A closed circuit $c$ in $G$ determines a functional $\beta_c\in\Hom(L, \zed)$ defined as
  follows.  If $e$ is an edge with weight $\gamma$, then $\beta_c(\gamma) = n$, where $n$ is the
  number of times $c$ crosses $e$ in the forward direction minus the number of times $c$ crosses $e$
  in the reverse direction.

  Each of the graphs in Figure~\ref{fig:graphs} representing nice boundary strata has the property
  that for each edge $e$ there are two circuits $c$ and $d$ which pass through $e$ once and have no
  other edge in common.  For each edge $f$, write $\rho(f) = w\otimes w$, where $w$ is the weight of
  $f$.  Then the functional $\beta_c\otimes\beta_d$ maps $\rho(e)$ to $1$ and $\rho(f)$ for any
  other edge  $f$ to $0$.
\end{proof}

Choose $\sigma_1, \ldots, \sigma_m\in\bS(\Hom(L, \zed))$ as in the lemma, and choose a basis
$\tau_1, \ldots, \tau_n$ of the annihilator $N(\mathcal{S})\subset\bS(\Hom(L, \zed))$ of
$\{\alpha_i\otimes\alpha_i\}_{i=1}^m$.

Let $U\subset\barmoduli[3](L)$ be the open subset consisting of $\moduli[3](L)$, $\mathcal{S}$, and
any intermediate boundary stratum obtained by pinching some subset of the curves $\{\gamma_i\}$.
We consider the map $\Xi\colon U\to \cx^m\times(\cx^*)^n$ defined by
\begin{equation*}
  \Xi = (\Psi(\sigma_1), \ldots, \Psi(\sigma_m), \Psi(\tau_1), \ldots, \Psi(\tau_n)),
\end{equation*}
sending $\mathcal{S}$ to $(0, \ldots, 0)\times(\cx^*)^n$.  

Any automorphism $T$ of $L$ induces an automorphism $\phi_T$ of $\barmoduli(L)$ defined by replacing
the marking $\rho$ of the marked surface $(X, \rho)$ with $\rho\circ T$.  Let $\iota\colon L\to L$
be the negation homomorphism $\phi(\alpha) = -\alpha$.  We define $\barmoduli'(L)$ to be the
quotient of $\barmoduli(L)$ by the involution $\phi_\iota$.

Each of the meromorphic functions $\Psi(\alpha)$ on $\barmoduli(L)$ is constant on orbits of
$\phi_\iota$ and so defines a meromorphic function $\Psi'(\alpha)$ on $\barmoduli'(L)$.  If
$\mathcal{S}$ is fixed by $\phi_\iota$, then so is $U$, and the map
$\Xi$ then factors through to a map $\Xi'\colon U'\to \cx^m\times(\cx^*)^n$, where
$U'=U/\phi_\iota$.

\begin{lemma}
  \label{lem:boundary_coordinates}
  Consider a nice boundary stratum $\mathcal{S}\subset\barmoduli[3](L)$.  If $\mathcal{S}$ is not
  fixed by $\pi_\iota$, then  for any basis $(\tau_1, \ldots, \tau_n)$ of
  $N(\mathcal{S})$, the functions $\Psi(\tau_1), \ldots, \Psi(\tau_n)$ form a system of local
  coordinates on $\mathcal{S}$.  If $\mathcal{S}$ is 
  fixed by $\phi_\iota$, then  for any basis $(\tau_1, \ldots, \tau_n)$ of
  $N(\mathcal{S})$, the functions $\Psi'(\tau_1), \ldots, \Psi'(\tau_n)$ form a system of local
  coordinates on $\mathcal{S}/\phi_\iota$.  
\end{lemma}

\begin{proof}
  It is enough to produce a single basis of $N(\mathcal{S})$ which yields a system of local
  coordinates, since the coordinate systems defined by any two bases are related by an automorphism of the
  algebraic torus $(\cx^*)^n$.

  Any type $[6]$ stratum $\mathcal{S}$ is fixed by $\phi_\iota$.  
  Corollary~\ref{cor:CRarecoord6} implies that the functions $\Psi'(s_i\otimes s_j)$ for $i\neq j$
  identify $\mathcal{S}/\phi_\iota$ with an open subset of $(\cx^*)^3$ (the involution $\phi_\omega$
  was called $i$ in that Corollary), and so they give a system of local coordinates on
  $\mathcal{S}/\phi_\iota$. 

  Any type $[4]\times^4[4]$ stratum is also fixed by $\phi_\iota$.  We use the notation for theses
  strata from p.~\pageref{page:4times4}.  Under the identification of $\mathcal{S}$ with
  $\moduli[0,4]\times\moduli[0,4]$ the map $\phi_\iota$ is just the involution exchanging the two
  factors.

  Let $\{s_1, \ldots, s_3\}$ be a basis of $F$ dual to $\{r_1, \ldots, r_3\}$.  Let
  $\tau_1 = (s_2 - s_1)\otimes s_3$ and $\tau_2 = (s_3 - s_1) \otimes s_2$.  From
  the definition of $\Psi$,
  \begin{equation*}
    \Psi'(\tau_1) = R_u R_v \quad\text{and} \quad
    \Psi'(\tau_2) = (1-R_u)(1-R_v),
  \end{equation*}
  a system of local coordinates on $\moduli[0,4]\times\moduli[0,4]/\phi_\iota$.

  The remaining cases are strata not fixed by $\phi_\iota$.  We leave these simpler cases to the reader.
\end{proof}

\begin{prop}
  \label{prop:period_coordinates}
  Consider a nice $L$-weighted boundary stratum $\mathcal{S}$ in $\barmoduli[3](L)$.  If
  $\mathcal{S}$ is not fixed by $\phi_\iota$, then the map $\Xi$
  is locally biholomorphic on a neighborhood of $\mathcal{S}$. Otherwise $\Xi'$ is
  locally biholomorphic on a neighborhood of $\mathcal{S}/\phi_\iota$.  In either case, the map
  $\Xi$ is open. 
\end{prop}

\begin{proof}
  Suppose $\mathcal{S}$ is not fixed by the involution.
  Centered at an arbitrary point of $\mathcal{S}$, we choose plumbing coordinates $t_1, \ldots,
  t_m$, $s_1, \ldots, s_n$, as in \S\ref{sec:RS}, so that each divisor $D_i$ where $\gamma_i$ has
  been pinched is cut out by $t_i=0$.  We must show that the Jacobian of $\Xi$ at $(\boldsymbol{0},
  \boldsymbol{0})$ is nonzero.   The functions $\Psi(\sigma_i)$
  vanish to order one on $D_i$ and zero on $D_j$ for $j\neq i$.  We then have
  $\pderiv{\Psi(\sigma_i)}{t_j}(\boldsymbol{0}, \boldsymbol{0}) = 0$ if $i\neq j$,
  $\pderiv{\Psi(\sigma_i)}{s_j}(\boldsymbol{0}, \boldsymbol{0}) = 0$ for all $i$ and $j$, and
  $\pderiv{\Psi(\sigma_i)}{t_i}(\boldsymbol{0}, \boldsymbol{0}) \neq 0$ for all $i$.  Thus, to show
  that the Jacobian of $\Xi$ at $(\boldsymbol{0}, \boldsymbol{0})$ is nonzero, it suffices to show
  that the matrix $(\pderiv{\Psi(\tau_i)}{s_j}(\boldsymbol{0}, \boldsymbol{0}))$ is invertible.  In
  other words, we must show that the functions $\Psi(s_j)$ locally define a system of local
  coordinates on $\mathcal{S}$.  This is the content of Lemma~\ref{lem:boundary_coordinates}.

  The case where $\mathcal{S}$ is fixed is nearly identical.  Note that since the quotient mapping
  $\barmoduli[3](L)\to\barmoduli[3]'(L)$ is unbranched along the boundary divisors, the order of
  vanishing of any $\Psi'(a)$ along $D_i$ is also given by the formula of
  Theorem~\ref{thm:meromorphic}.

  The last statement follows, since any quotient map~-- in particular, the canonical map
  $\moduli[3](L)\to\moduli[3]'(L)$~--  is open.
\end{proof}

\paragraph{Closures of algebraic tori.}

The period coordinates above reduce the problem of computing the boundary of the eigenform locus to
computing the closures of algebraic tori $T\subset(\cx^*)^n\subset\cx^n$, which we now consider.

Consider the algebraic torus $T = (\cx^*)^k\times(\cx^*)^\ell\subset\cx^k\times(\cx^*)^\ell$.  We
identify the character group $\chi(T)$ with $\zed^k\oplus\zed^\ell$ by assigning to $(\ba, \bb) =
(a_1, \ldots, a_k, b_1, \ldots, b_\ell)$ the character $\lambda_{(\ba, \bb)}\colon T\to\cx^*$
defined by
\begin{equation*}
  \lambda_{(\ba, \bb)}(\bz, \bw) =  z_1^{a_1}\cdots z_k^{a_k}w_1^{b_1}\cdots w_\ell^{b_\ell}
\end{equation*}
Given a subgroup $L$ of $\chi(T)$ with $\chi(T)/L$ torsion-free and a homomorphism $\phi\colon
L\to\cx^*$, we define $T_{A, \phi}$ to be the subvariety of $T$ cut out by the monomial equations
\begin{equation}
  \label{eq:18}
  \lambda_{(\ba, \bb)}(\bz, \bw) = \phi(\ba, \bb)
\end{equation}
for each $(\ba, \bb)\in L$, a translate of a subtorus of $T$.

Let $\Delta = \{\boldsymbol{0}\}\times(\cx^*)^\ell$.  We define
\begin{gather*}
  C = \{(\ba, \bb)\in \chi(T) : a_i\geq 0 \text{ for }1 \leq i \leq k\},\quad\text{and} \\
  N = \{\boldsymbol{0}\}\oplus\zed^\ell\subset\chi(T).
\end{gather*}
Let $\Delta_{L, \phi}$ be the subvariety of $\Delta$ cut out by the monomial equations \eqref{eq:18}
for $(\ba, \bb)\in L\cap N$.

\begin{theorem}
  \label{thm:torus_closure}
  The closure $\overline{T}_{L, \phi}\cap \Delta$ is nonempty if and only if $L\cap C\subset N$, in which
  case we have $\overline{T}_{L, \phi}\cap \Delta = \Delta_{L, \phi}$.
\end{theorem}

\begin{proof}
  Suppose $(\ba, \bb)$ is a nonzero element of $(L\cap C)\setminus N$.  The equation \eqref{eq:18}
  is then satisfied on $T_{L, \phi}$, but $\lambda_{(\ba, \bb)}(\bz, \bw)\equiv 0$ on $\Delta$, so
  $\Delta$ and $\overline{T}_{L, \phi}$ must be disjoint.

  Conversely, suppose $L\cap C\subset N$.  Then the orthogonal projection $p(L)$ of $L$
  onto the $\zed^k$ factor of $\chi(T)$ satisfies $p(L)\cap C = 0$.  Theorem~15.7 of \cite{roman}
  states that given a subspace $V$ of $\reals^n$ with $V\cap \{\bx\in\reals^n : x_i \geq 0 \text{
    for all $i$}\} = 0$, there is a vector $\by\in V^\perp$ with each coordinate positive.  Thus we
  may find an integral $\bc\in p(L)^\perp\subset\zed^k$ with positive coordinates.

  Note that the curve parameterized by
  \begin{equation*}
    f(w) = (d_1 w^{c_1}, \ldots, d_kw^{c_k}, e_1, \ldots, e_\ell)
  \end{equation*}
  lies in $T_{L, \phi}$ if and only if for each $(\ba, \bb)\in L$, the equation
  \begin{equation}
    \label{eq:19}
    d_1^{a_1} \dots d_k^{a_k} e_1^{b_1} \dots e_\ell^{b_\ell}= \phi(\ba,\bb)
  \end{equation}
  is satisfied, in which case $(0, \ldots, 0, e_1, \ldots, e_\ell)\in \overline{T}_{L,
    \phi}\cap\Delta$.

  Choose some $(\boldsymbol{0}, \be)\in\Delta_{L, \phi}$, and let $(\ba_i,\bb_i) = (a_{i1},
  \ldots. a_{ik}, b_{i1}, \ldots, b_{i\ell})$ for $1\leq i \leq \dim(L)$ be a basis of $L$ with
  $a_{ij}=0$ for $i \leq \dim(L\cap N)$.  We must find $g_1, \ldots, g_k$ satisfying the equations
  \begin{equation}
    \label{eq:20}
    a_{i1} g_1 + \dots + a_{ik} g_k + b_{i1} \log e_1 + \dots + b_{i\ell}\log e_\ell = \log\phi(\ba_i, \bb_i).
  \end{equation}
  The first $\dim(L\cap N)$ equations don't involve the $g_i$ and are satisfied automatically
  because $(\boldsymbol{0}, \be)\in\Delta_{A, \phi}$ as long as the values of $\log$ were chosen correctly.
  The vectors $\ba_{\dim(L\cap N)+1}, \ldots, \ba_{\dim(L)}$ are linearly independent, so the matrix
  $(a_{ij})$ (with $\dim(L\cap N)< i \leq \dim(L)$ and $1\leq j \leq k$) has maximal rank.  Thus we
  can solve \eqref{eq:20} for the $g_i$.  Setting $d_i = e^{g_i}$, \eqref{eq:19} is satisfied.
\end{proof}

\paragraph{Proof of Theorem~\ref{thm:boundary_suff}.}

It suffices to show that for any cusp packet $(\mathcal{I}, T)$ and admissible
$\mathcal{I}$-weighted boundary stratum $\mathcal{S}\subset\barmoduli[3](\mathcal{I})$ the variety
$\mathcal{S}(T)$ lies in the closure of $\RM(\mathcal{I}, T)$.

For nice boundary strata, the map $\Xi$ of Proposition~\ref{prop:period_coordinates} reduces the
computation of the closure of $\RM$ to the computation of the closure of an algebraic torus in
$\cx^n$ (since under an open mapping, the inverse image of the closure of a set is equal to the
closure of the inverse image), which is done in Theorem~\ref{thm:torus_closure}.  It is easily
checked that the condition of this theorem is equivalent to the admissibility condition.  This handles
admissible boundary strata of type $(1,1), (2,1), (2,2), (3, 1)$, and $(4,2)$ in
Figure~\ref{fig:graphs}.

Admissible boundary strata which are in the boundary of a nice admissible boundary stratum
$\mathcal{S}$ with $\codim(\mathcal{S})=0$ are then automatically in the closure of
$\RM(\mathcal{I}, T)$.  It follows from
Lemma~\ref{lemma:cones} that any admissible boundary stratum $\mathcal{S}$ with
$\codim(\mathcal{S})=0$ is in the boundary of such a nice admissible stratum, since some collection
of nodes can be unpinched to obtain a stratum of type $[4]\times^4[4]$ or $[5]\times^3[3]$ where the
cone condition still holds.  This handles admissible boundary strata of type $(3,2), (3, 3), (4,1)$, and
$(4,3)$.

It remains to consider admissible boundary strata of type $(2,3)$, $(2,4)$, $(3,4)$, $(3,5)$, $(4,4)$, and
$(4,5)$.  Any such boundary stratum in the closure of a unique irreducible
Lagrangian boundary stratum $\mathcal{S}$.  The weights of  $\mathcal{S}$ define the equation
\begin{equation}
  \label{eq:23}
  \Psi(\sigma) = u,
\end{equation}
with $u$ and $\sigma$ as in \eqref{eq:21} and \eqref{eq:22}.  Let $V\subset\barmoduli[3](\mathcal{I})$
be the subvariety cut out by this equation.  For any stratum
$\mathcal{S}'\subset\overline{\mathcal{S}}$, we have $\mathcal{S}'(T)= \mathcal{S}'\cap V$ by the
definition of $\mathcal{S}'(T)$, so we must show for any such $\mathcal{S}'$ that $\mathcal{S}'\cap
V \subset{\barRM(\mathcal{I}, T)}$.  Since we have already handled irreducible boundary strata,
we know that $V\cap\mathcal{S} = {\barRM(\mathcal{I}, T)} \cap \mathcal{S}$.  It follows that
$\overline{V\cap \mathcal{S}}\subset{\barRM(\mathcal{I}, T)}$.  If $\overline{S}\cap V$ were
irreducible, it would follow that $\overline{\mathcal{S}}\cap V = \overline{\mathcal{S}\cap V}$, and
we would be done.

We see the irreducibility of $\overline{\mathcal{S}}\cap V$ as follows.  Since $V$ is
codimension-one and $\mathcal{S}\cap V$ is irreducible, as is easily seen from the form of the
cross-ratio equation \eqref{eq:CREQ}, $\overline{\mathcal{S}}\cap V$ could only fail to be
irreducible if a two-dimensional stratum in the boundary of $\mathcal{S}$ were contained in $V$.
Such a stratum must be of type $(2,3)$ (that is, $[4]\times^2[4]$) or $(2,4)$ in
Figure~\ref{fig:graphs}.  The restriction of the equation~\eqref{eq:23} to a type $[4]\times^2[4]$
stratum is the cross-ratio equation \eqref{eq:CR4ti24} which is not satisfied on an entire stratum.
Similarly, a type $(2,4)$ stratum is isomorphic to $\moduli[0,5]$, and the equation \eqref{eq:23}
reduces to the equation $R = u$, where $R$ is a cross-ratio of four marked points and $u$ is a root
of unity.  This equation is not satisfied on the entire stratum.
\qed

\section{Existence of an admissible basis} \label{sec:exist_adm}

In this section we construct, for any totally real cubic number field $F$ with ring of
integers $\mathcal{O}_F$, an $\mathcal{O}_F$-ideal with an admissible basis.  
This will be used in the next section to show $\GLtwoRplus$-noninvariance
of eigenform loci.

\begin{lemma}
  \label{lem:monogenetic_basis}
  For any cubic number field $F$, there is some fractional $\mathcal{O}_F$-ideal $\mathcal{I}$ with
  basis $\{1, \alpha, \alpha^2\}$.
\end{lemma}

\begin{proof}
  Given $\alpha\in F\setminus\ratls$, let $\mathcal{I}_\alpha\subset F$ be the lattice $\langle 1,
  \alpha, \alpha^2\rangle$.  If $a X^3 + b X^2 + c X + d\in \zed[X]$ is the minimal polynomial of
  $\alpha$, one checks that
  \begin{equation*}
    R = \langle 1, a\alpha, a\alpha^2+b\alpha\rangle \quad 
    \text{satisfies} \quad R \cdot \mathcal{I_\alpha} \subset 
    \mathcal{I_\alpha}.
  \end{equation*}

  We must arrange that $R = \mathcal{O}_F$.  Let $\{1, \mu, \nu\}$ be a basis of $\mathcal{O}_F$.
  Associated to this basis is the index form, an integral binary cubic form which is defined by
  \begin{equation*}
    F(x,y)^2 = \disc(x\nu-y\mu)/\disc(F)
  \end{equation*}
  for $x,y\in \ratls$ (see \cite[Proposition~8.2.1]{Co00}), where $\disc(\alpha)$ is the
  discriminant of the lattice $\mathcal{I}_\alpha$.  If we choose $\alpha$ to be a root of $F(x,1)$,
  then $R = \ord_F$ by \cite[Proposition~8.2.3]{Co00}.
\end{proof}

\begin{prop}
  \label{prop:exist_admtriple}
  Given a totally real cubic field $F$, there is an
  $\ord_F$-ideal $\mathcal{I}$ with an admissible basis.
\end{prop}

\begin{proof}
  Let $\mathcal{I}$ be a fractional ideal with basis $\{1, \alpha, \alpha^2\}$
  which is provided by Lemma~\ref{lem:monogenetic_basis}.  The basis given by $r_1 = \alpha$, $r_2 = (1-\alpha)$, and 
  $r_3 = \alpha(\alpha-1)$ satisfies the equation
  \begin{equation}
    \label{eq:rec_dep}
    \frac{1}{\NFQ(r_1)}\frac{\NFQ(r_1)}{r_1} + \frac{1}{\NFQ(r_2)}\frac{\NFQ(r_2)}{r_2} +
    \frac{1}{\NFQ(r_3)}\frac{\NFQ(r_3)}{r_3} = 0
  \end{equation}
  so
  $$\dim \Span \{\NFQ(r_1)/r_1, \NFQ(r_2)/r_2, \NFQ(r_3)/r_3\}=2.$$

  The no-half-space condition is then equivalent to the coefficients of \eqref{eq:rec_dep} having
  the same sign, that is $\NFQ(\alpha)<0$, and $\NFQ(1-\alpha)<0$.  We are free to replace $\alpha$
  with $\alpha'=\alpha-k$ for any $k\in\zed$, since the basis $\{1, \alpha', \alpha'^2\}$
  spans the same lattice.  Thus the problem is reduced to finding $k\in \zed$ such that
  $\NFQ(\alpha+k)$ and $\NFQ(\alpha+k+1)$ have opposite signs.

  Define $P(k) = \NFQ(\alpha+k)$.  Then $P(k) = -F(k)$, where $F$ is the monic minimal polynomial of
  $\alpha$.  We claim that there are consecutive integers at which $P$ has opposite signs.  In fact,
  this holds for any polynomial $P$ of odd degree with no integral roots, for if $P$ had the same
  sign at any two consecutive integers, then it must have the same sign at all integers.  This is
  impossible, as the sign of $P(x)$ as $x\to\infty$ is the opposite of $P(x)$ as $x\to -\infty$.
\end{proof}

\par
\begin{example}
  {\rm Consider the field $F = \ratls[x]/\langle x^3-x^2-10x+8 \rangle$ of discriminant $D=961$. Its
    ring of integers $\mathcal{O}_F = \langle 1,x,(x^2+x)/2 \rangle$ is not monogenetic, i.e.\ does
    not have a basis of the form $\{1, \theta, \theta^2\}$ for any $\theta$ in $F$. The class number
    of $\mathcal{O}_F$ is one, so the above algorithm provides a basis of this form spanning some
    fractional ideal similar to $\mathcal{O}_F$.
    \par
    One calculates the index form to be $F(X,1) = 2X^3 - X^2 - 5X + 2$, thus if $\theta$ is a root
    of this polynomial,
    then $\mathcal{O}_F = \langle 1,2\theta,2\theta^2-\theta \rangle$ and $\mathcal{I} = \langle 1,
    \theta, \theta^2\rangle$. Here $N(\alpha) = -1$ and $\NFQ(1-\alpha)=-1$, so the last step of the
    proof is unnecessary. 
  }
\end{example}
\par
\begin{cor} \label{cor:6inbd}
  For any field $F$ the closure of the eigenform locus ${\cal
    E}_{\ord_F}$ intersects a boundary stratum of type $[6]$, that is, a stratum  of trinodal curves.
\end{cor}
\par
We do not know if the class of the ideal class of $\mathcal{I}$ given by
Lemma~\ref{lem:monogenetic_basis} always is the class of $\Ord_F$.  Nor do we know if there is
always an admissible basis of $\Ord_F$.  Computer experiments
using the algorithm described in Appendix~\ref{sec:bdcounting} suggest an affirmative answer.  This
algorithm also produces examples of ideal classes with no such bases.

\section{\Teichmuller curves and the $\GLtwoRplus$ action} \label{sec:apTc}

In preparation for the next sections, we recall the well-known action of $\GLtwoRplus$ on
$\Omega\moduli$ and the basic properties of \Teichmuller curves in $\moduli$.

\paragraph{Translation surfaces.}

A Riemann surface $X$ equipped with a nonzero holomorphic one-form $\omega$ is otherwise known as a
\emph{translation surface}.  The form $\omega$ defines a metric $|\omega|$ on $X\setminus
Z(\omega)$, where $Z(\omega)$ is the set of zeros of $\omega$, assigning to a vector $v$ the length
$|\omega(v)|$.  The metric $|\omega|$ has cone singularities at the zeros of $\omega$.

The form $\omega$ defines an atlas of charts $\{\phi_\alpha\colon U_\alpha\to \cx\}$
covering $X\setminus Z(\omega)$, where $\phi_\alpha(z) = \int_p^z\omega$ for some choice of
basepoint $p\in U_\alpha$.  The transition functions of this atlas are translations of $\cx$, and
the form $\omega$ is recovered by $\omega|_{U_\alpha} = \phi_\alpha^{-1}(dz)$.

Any translation-invariant geometric structure on $\cx$ can then be pulled back to $X$ via this
atlas.  In particular, for any slope $\theta\in\reals\cup\{\infty\}$ there is a foliation
$\mathcal{F}_\theta$ of $X$ by geodesics of slope $\theta$.

\paragraph{$\GLtwoRplus$ action.}

We can now regard $\Omega\moduli$ as the moduli space of genus $g$ translation surfaces.
$\GLtwoRplus$ acts on $\Omega\moduli$ as follows.  We identify $\cx$ with $\reals^2$ in the usual
way so that a matrix $A\in\GLtwoRplus$ determines a $\reals$-linear automorphism of $\cx$.
Replacing the atlas of charts $\{\phi_\alpha\colon U_\alpha\to \cx\}$ defined above with
$\{A\circ\phi_\alpha\colon U_\alpha\to \cx\}$ yields a new atlas with transition functions also
translations of $\cx$.  Pulling back the complex structure of $\cx$ and the one-form $dz$ via this
atlas defines a new translation surface $A\cdot(X, \omega)$.

\paragraph{Strata.}

Given a partition $n_1, \ldots, n_r$ of $2g-2$, there is the stratum
$$\Omega\moduli(n_1, \ldots,n_r)\subset\Omega\moduli$$
of forms with exactly $r$ zeros of orders given by the $n_i$.  This
stratification is preserved by the $\GLtwoR$-action.

\paragraph{Veech surfaces and  \Teichmuller curves.}

We define the affine automorphism group of a translation surface $(X, \omega)$  to be the group $\Aff^+(X,
\omega)$ of orientation preserving, locally affine homeomorphisms of $(X, \omega)$.
There is a homeomorphism
\begin{equation*}
  D\colon\Aff^+(X, \omega)\to \SLtwoR,
\end{equation*}
sending a map $A$ to its derivative $DA$ in a local translation chart.  We define $\SL(X, \omega) =
D(\Aff^+(X, \omega)) \subset\SLtwoR$.  The group $\SL(X, \omega)$ is known as the \emph{Veech group}
of $(X, \omega)$.

The surface $(X, \omega)$ is said to be \emph{Veech} if
$\SL(X, \omega)$ is a lattice in $\SLtwoR$.  The group $\SL(X, \omega)$ coincides with the
stabilizer of $(X, \omega)$ under the $\GLtwoRplus$-action.  Thus $(X, \omega)$ is Veech if and only if
$\GLtwoRplus\cdot(X, \omega)\subset\Omega\moduli$ descends to an immersed finite volume Riemann
surface (orbifold) in $\moduli$.  An immersed finite volume Riemann surface arising in this way is
called a \emph{\Teichmuller curve} and is necessarily isometrically immersed with respect to the
\Teichmuller metric.

A \Teichmuller curve can also be regarded as an embedded smooth curve in $\proj\Omega\moduli$.

\paragraph{Periodicity.}

A \emph{saddle connection} on a translation surface $(X, \omega)$ is an embedded geodesic segment
connecting two zeros of $\omega$.

The foliation $\mathcal{F}_\theta$ of slope $\theta$ is said to be \emph{periodic} if every leaf of
$\mathcal{F}_\theta$ is either closed (i.e.\ a circle) or a saddle connection. In this case, we call
$\theta$ a periodic direction.  A periodic direction $\theta$ yields a decomposition of $(X,
\omega)$ into finitely many maximal cylinders foliated by closed geodesics of
slope $\theta$.  The complement of these cylinders is a finite collection of saddle connections.

Veech proved the
following strong periodicity property of Veech surfaces.

\begin{theorem}[\cite{veech89}]
  \label{thm:Veech_periodicity}
  Suppose $(X, \omega)$ is a Veech surface with either a closed geodesic or a saddle connection of
  slope $\theta$.  Then the foliation $\mathcal{F}_\theta$ is periodic.
\end{theorem}

Given a Veech surface $(X, \omega)$ generating a \Teichmuller curve $C\subset\proj\Omega\moduli$,
there is a natural bijection between the cusps of $C$ and the periodic directions on $(X, \omega)$,
up to the action of $\SL(X, \omega)$.  The cusp associated to a periodic direction $\theta$ is the
limit of the geodesic $A_t R\cdot(X, \omega)$, where $R\subset\SOtwoR$ is a rotation which makes
$\theta$ horizontal, and
\begin{equation*}
  A_t =
  \begin{pmatrix}
    e^{-t} & 0\\
    0 & e^t
  \end{pmatrix}.
\end{equation*}
\label{page:surgery}
The stable form in $\proj\Omega\barmoduli$ which is the limit of this cusp is obtained by cutting
each cylinder of slope $\theta$ along a closed geodesic and gluing a half-infinite cylinder to each
resulting boundary component (see \cite{masur75}).  These infinite cylinders are the poles of the resulting
stable form, and the two poles resulting from a single infinite cylinder are glued to form a node.

A periodic direction $\theta$ of a Veech surface $(X, \omega)$ generating a \Teichmuller curve $C$
is \emph{irreducible} if the complement of the cylinders of $\mathcal{F}_\theta$ is a connected
union of saddle connections.  Equivalently, a periodic direction is irreducible if the stable curve
at the limit of the corresponding cusp of $C$ is irreducible.  An irreducible periodic direction
always has $g$ cylinders, where $g$ is the genus of $X$.

\begin{lemma} \label{le:31topo1}
  Every Veech surface $(X, \omega)$ having at most two zeros has an irreducible periodic direction.
\end{lemma}
\par
\begin{proof}
  If $(X, \omega)$ has only a single zero, then every periodic direction is irreducible.

  If $(X, \omega)$ has two zeros, take a saddle connection $I$ joining them.  Such a saddle connection can be
  obtained by straightening any path joining the two zeros to a geodesic path.
  The direction
  determined by $I$ is periodic by Theorem~\ref{thm:Veech_periodicity}, and this direction is
  irreducible as the graph of saddle connections is connected.
\end{proof}
 
\paragraph{Algebraic primitivity.}
 
The trace field of a Veech surface $(X, \omega)$ is the field $\ratls(\trace A: A\in \SL(X, \omega))$.
The trace field of $(X, \omega)$ is a number field which is totally real (see \cite{moeller06} or
\cite{hubertlanneau}) whose degree is at most the genus of $X$ (see \cite{mcmullenbild}).  A Veech
surface $(X, \omega)$ is said to be \emph{algebraically primitive} if the degree of its trace field
is equal to the genus of $X$.

Our finiteness theorem for algebraically primitive \Teichmuller curves will require the following
facts.
\begin{theorem}[\cite{moeller06, moeller}]
  \label{thm:alg_prim_rm_torsion}
  Suppose $(X, \omega)$ is an algebraically primitive Veech surface.  We then have
  \begin{itemize}
  \item $\GLtwoRplus\cdot(X, \omega)$ lies in the locus of eigenforms for real multiplication by the
    trace field of $(X, \omega)$.
  \item For any two distinct zeros $p$ and $q$ of $\omega$ the divisor $p-q$, regarded as a point
    in $\Jac(X)$, is torsion.
  \end{itemize}
\end{theorem}

The following lemma shows that the heights of cylinders in an irreducible periodic direction of an
algebraically primitive Veech surface can be
recovered from knowledge of their widths.

\begin{lemma} \label{lem:heightknown} Suppose $(X, \omega)\in\Omega\moduli[g]$ is an eigenform for
  real multiplication by a totally real field $F$ of degree $g$, and suppose the horizontal
  direction of $(X, \omega)$ is periodic and irreducible.  Then the vector $(r_i)_{i=1}^g$ of widths
  of the $g$ horizontal cylinders is a real multiple of a basis of $F$ over $\ratls$, and the
  corresponding vector $(s_i)_{i=1}^g$ of heights of these cylinders is a real multiple of the
  dual basis of $F$ over $\ratls$ with respect to the trace pairing.
 \end{lemma}

 \begin{proof}
   Let $M\subset H_1(X; \ratls)$ be the $g$-dimensional subspace generated by the core curves of
   cylinders, and let $N = H_1(X; \ratls)/M$.  Real multiplication gives both $M$ and $N$ the
   structure of one-dimensional $F$-vector spaces, so we may choose isomorphisms of $F$-vector spaces
   $\phi\colon M \to F$ and $\psi\colon N\to F$.  Since $\omega$ is an eigenform, there are
   constants $c,d\in\reals$ and an embedding $\iota\colon F\to\reals$ such that
   \begin{equation}
     \label{eq:28}
     \int_\alpha\omega = c \iota(\phi(\alpha)) \qtq{and} \Im\int_\beta\omega=d\iota(\psi(\beta))
   \end{equation}
   for all $\alpha\in M$ and $\beta\in N$.

   The intersection pairing between $M$ and $N$ yields a perfect pairing $\langle \, ,
   \,\rangle\colon F\times F\to\ratls$ which is compatible with the action of $F$ in the sense that
   $\langle \lambda x, y\rangle = \langle x, \lambda y\rangle$ for all $\lambda\in F$.  A second
   such pairing is given by $(x, y) = \trace(xy)$.  Since the space of all such perfect pairings is
   a one-dimensional $F$-vector space, there is a $\lambda\in F$ such that
   \begin{equation}
     \label{eq:29}
     \langle x, \lambda y\rangle = \trace(xy)
   \end{equation}
   for all $x,y\in F$.

   Let $\alpha_i\in M$ be the class of a core curve of the $i$th horizontal cylinder $C_i$, let $r_i =
   \phi(\alpha_i)$, and let $s_i$ be the dual basis of $F$ to the $r_i$.  Choose $\beta_i\in H_1(X;
   \ratls)$ such that $\beta_i \equiv\psi^{-1}(\lambda s_i)\pmod M$.  By \eqref{eq:29}, the $\beta_i$
   are dual to the $\alpha_i$ with respect to the intersection pairing.  It follows that $\beta_i$
   crosses $C_i$ once and no other cylinder, so the height of $C_i$ is $\Im\int_{\beta_i}\omega$.
   By \eqref{eq:28}, we have
   \begin{equation*}
     \int_{\alpha_i} \omega = c \iota(r_i) \qtq{and} \Im\int_{\beta_i}\omega =
     d\iota(\lambda)\iota(s_i). \qedhere
   \end{equation*}
 \end{proof}

\section{$\GLtwoRplus$ non-invariance} \label{sec:non-inv}

In this section we show that the $\GLtwoRplus$ action on $\Omega\moduli[g]$ admits
a continuous extension to the Deligne-Mumford compactification. We deduce from this and the previous
sections that the eigenform locus for real multiplication by the ring of integers in any totally
real cubic field is {\em not} invariant under the action of $\GLtwoRplus$.  McMullen
proved non-invariance in \cite{mcmullenbild}  for the maximal order in $\ratls(\cos(2\pi/7))$
using the existence of a curve with a special automorphism group.
\par
\paragraph{$\GLtwoRplus$-action  on $\Omega\barmoduli$.}

The definition of the $\GLtwoRplus$ action on Abelian differentials works
just as well for stable Abelian differentials $(X, \omega)$, regarding $\omega$ as a holomorphic
one-form on the punctured Riemann surface $X$.  The opposite-residue condition is preserved by
linearity of the $\GLtwoRplus$-action on $\reals^2$: $A\cdot(-v) = -A\cdot v$.  Thus we obtain an
action of $\GLtwoRplus$ on $\Omega\barmoduli$ and $\Omega\augteich(\Sigma_g)$.

\begin{prop} \label{prop:contextension}
  The action of $\GLtwoRplus$ on $\Omega \barmoduli$ is continuous.
\end{prop}
\par
\begin{proof}
  We show that the action of $\GLtwoRplus$ on $\Omega
  \augteich(\Sigma_g)$ is continuous.  As the $\GLtwoRplus$-action on $\Omega\augteich$ commutes with the
  action by the mapping class group, this action then descends to a continuous action on $\Omega\barmoduli$.
  
  We claim that under the action of $\GLtwoRplus$ on $\Omega\augteich(\Sigma_g)$ the hyperbolic
  lengths of simple closed curves vary continuously.  Since the topology of $\augteich(\Sigma_g)$ is
  the smallest topology such that hyperbolic lengths of simple closed curves are continuous
  functions $\ell_\gamma\colon\augteich(\Sigma_g)\to \reals_+ \cup \{\infty\}$, it follows that
  under this action, the underlying Riemann surfaces are varying continuously.

  That the length of a simple closed curve $\gamma$ varies continuously follows easily from
  considering the annular covering of $X$ corresponding to $\langle \gamma\rangle\subset\pi_1(X)$.
  The modulus of this annulus varies continuously under quasiconformal deformation, and the length
  of $\gamma$ is determined by this modulus (see for example \cite[Proposition~7.2]{DoHu93}).
  
  Consider a form $([f\colon \Sigma_g\to X], \omega)\in\Omega\augteich(\Sigma_g)$.  Say the collapse
  $f$ pinches a set of curves $S$ on $\Sigma_g$.  We may choose a set of curves $\alpha_1,\ldots,
  \alpha_g$ on $\Sigma_g$ that generate a Lagrangian subspace of $H_1(\Sigma_g, \zed)$ and such that
  each of the $\alpha_i$ is either one of the curves in $S$ or intersects each curve in $S$
  trivially.  We obtain a trivialization of the bundle $\Omega\augteich(\Sigma_g)$ over a
  neighborhood of $X$ sending a form $\eta$ to $(\eta(\alpha_1), \ldots, \eta(\alpha_g))\in\cx^g$.

  Say$A\cdot(Y, \eta) = (Z, \zeta)$.  From the definition of the $\GLtwoRplus$ action, we have
  \begin{equation*}
    \zeta(\alpha_i) = A\cdot\eta(\alpha_i),
  \end{equation*}
  with $A\in\GLtwoRplus$ acting on $\cx\isom\reals^2$ in the usual way.  Thus
  $\eta(\alpha_i)$ varies continuously under the $\GLtwoRplus$-action, and so the action on
  $\Omega\augteich(\Sigma_g)$ is continuous.
\end{proof}

\paragraph{Four-punctured spheres.}

Given $r_1, r_2\in\cx$, we let $\mathcal{R}_{(r_1, r_2)}\isom\moduli[0,4]$ be the moduli space of pairs $(X,
\omega)$, where $X$ is the four-punctured sphere $\proj^1\setminus\{p_1, p_{-1}, p_2, p_{-2}\}$. and
$\omega$ is the unique meromorphic one-form with simple poles at the $p_i$ with residue $r_{\pm i}$
at $p_{\pm i}$.  We identify $\mathcal{R}_{(r_1, r_2)}$ with $\cx\setminus\{0,1\}$ via the
cross-ratio $R=[p_1,p_{-1},p_{-2}, p_2]$ and write $(X_R, \omega_R)$ for the form associated to the
cross-ratio $R$.

If $r_1, r_2\in \reals$, then the subgroup $P\subset\GLtwoRplus$ of matrices fixing the vector $(1,0)$ acts on
$\mathcal{R}_{(r_1,r_2)}$, as this is the subgroup of $\GLtwoRplus$ preserving the residues $r_i$.

\begin{prop}
  \label{prop:stuff}
  Suppose $r_1, r_2\in\reals$, and $r_1\neq \pm r_2$.  We then have
  \begin{itemize}
  \item
    The horizontal foliation of each $(X_R, \omega_R)\in\mathcal{R}_{(r_1, r_2)}$ is periodic.  Each
    $(X_R, \omega_R)$ has either two or three cylinders (counting the two half-infinite cylinders of
    width $r_i$ as a single cylinder).
  \item The form $\omega_R$ has a double zero for the single value of $R$,
    \begin{equation}
      \label{eq:26}
      R = \left(\frac{r_1-r_2}{r_1+r_2}\right)^2.
    \end{equation}
  \item We define $\Spine_{(r_1,r_2)}\subset\mathcal{R}_{(r_1,r_2)}$ to be the locus of two-cylinder
    forms.  $\Spine_{(r_1, r_2)}$ is the locus of singular leaves of a quadratic differential on
    $\mathcal{R}_{(r_1, r_2)}$.  $\Spine_{(r_1, r_2)}$ is homeomorphic to a figure-9, with the three
    pronged singularity at the unique form $(X_R, \omega_R)$ with a double zero.  The one-pronged
    singularity is at $R=1$, the point in the boundary of $\mathcal{R}_{(r_1, r_2)}$ obtained by
    pinching the curve separating $p_{\pm 1}$ from $p_{\pm 2}$.
  \item $\Spine_{r_1, r_2}$ is the locus of points fixed by the action of $P$ on $\mathcal{R}_{(r_1, r_2)}$.
  \end{itemize}
\end{prop}

\begin{proof}
  See \cite[Proposition~7.3]{Ba07L} for the first statement, \cite[Proposition~6.10]{bainbridge07}
  for the second statement, and \cite[Proposition~7.4]{Ba07L} for the third statement.

  For the final statement, suppose $(X, \omega)\in \mathcal{R}_{(r_1,r_2)}$ is a three-cylinder
  surface.  Then there is a single finite horizontal cylinder $C\subset X$ with a simple zero of
  $\omega$ on the top and bottom boundaries of $C$.  The period $\int_\gamma\omega$ along a curve
  joining these two zeros has nonzero imaginary part, so it is not fixed by any matrix in $P$.  Thus
  $P$ does not fix $\omega$.

  If $(X, \omega)\in\Spine_{(r_1,r_2)}$, then $(X, \omega)$ is obtained by gluing four half-infinite
  cylinders to  graph (the spine of $(X, \omega)$. There is an affine automorphism of $(X, \omega)$
  with derivative $P$ which is the identity on the spine.  Thus $(X, \omega)$ is stabilized by the
  action of $P$. 
\end{proof}

\paragraph{$\GLtwoRplus$ non-invariance.}

We write $\Omega\E^\iota\subset\Omega\moduli$ for the locus of $\iota$-eigenforms (as opposed to its
projectivization $\E^\iota$).

\begin{theorem} \label{thm:sl_noninvariance} Let $\mathcal{O}$ be a totally real cubic order and
  $X\subset\Omega\moduli[3]$ an irreducible component of $\Omega\mathcal{E}^\iota_\mathcal{O}$.  If
  $\overline{X}\subset\Omega\barmoduli[3]$ has nontrivial intersection with a boundary stratum of
  type $[6]$, then $X$ is not invariant under the action of $\GLtwoRplus$.
\end{theorem}

\begin{proof}
  Suppose $\overline{X}$ meets the locus $\mathcal{R}_{(r_1, r_2, r_3)}$ of irreducible stable forms
  with poles of residues $(\pm r_1, \pm r_2, \pm r_3)$, with $(r_1, r_2, r_3)$ an admissible basis
  of $\iota(F)$.  In the boundary of $\mathcal{R}_{(r_1, r_2, r_3)}$ is a stratum $\mathcal{R}'$ of
  type $[4]\times^2[4]$ parameterizing forms with two nodes of residue $\pm r_2$, one of residue
  $\pm r_1$, and one of residue $\pm r_3$.  We identify $\mathcal{R}'$ with $\mathcal{R}_{(r_1,
    r_2)}\times\mathcal{R}_{(r_3, r_2)}\isom \moduli[0,4]\times\moduli[0,4]$, with cross-ratio
  coordinates $R_1$ on $\mathcal{R}_{(r_1, r_2)}$ and $R_3$ on $\mathcal{R}_{(r_3, r_2)}$ as in the
  previous paragraph.

  By Theorems~\ref{thm:explCREQ} and \ref{thm:boundary_suff}, $\overline{X}\cap(\mathcal{R}_{(r_1,
    r_2)}\times\mathcal{R}_{(r_3, r_2)})$ contains an irreducible component $V$ cut out by the
  equation
  \begin{equation}
    \label{eq:27}
    R_1^{a_1}R_3^{a_3} = \zeta
  \end{equation}
  for some root of unity $\zeta$.  We suppose $X$ is $\GLtwoRplus$-invariant, in which case $V$ is
  invariant under $P\subset\GLtwoRplus$ by Proposition~\ref{prop:contextension}.

  We define $i, \psi_j\colon\cx\to\cx$ by $\psi_j(z) = z^{a_i}$, and $i(z) = \zeta/z$.
  Since the spine in $\mathcal{R}_{(r_i, r_2)}$ is the locus fixed by the action of
  $P\subset\GLtwoRplus$ by Proposition~\ref{prop:stuff}, if $V$ is preserved by this action, we must have
  \begin{equation*}
    \upsilon_3^{-1}i\psi_1(\Spine_{(r_1, r_2)})\subset\Spine_{(r_3, r_2)}.
  \end{equation*}
  Moreover, since the $\psi_j$ and $i$ are local homeomorphisms, for $p$ a one or three-pronged
  singularity of $\Spine_{(r_1, r_2)}$, we must have that $\psi_3^{-1}i\psi_1(p)$
  consists entirely of one or three-pronged (respectively) singularities of $\Spine_{(r_3, r_2)}$.
  Since each spine has only one singularity of each type, we must have $a_3 = \pm 1$.  By switching
  the roles of $r_1$ and $r_3$, we must also have $a_1=\pm 1$.  As the one-pronged singularity of
  each spine is located at $R_j=1$, we must have $\zeta=1$, or else $\psi_3^{-1}i\psi_1(1)\neq 1$.

  It remains to consider the case where $a_j = \pm 1$ and $\zeta=1$.  Given the location of the
  three-pronged singularities \eqref{eq:26} and the cross-ratio equation \eqref{eq:27}, we obtain
  \begin{equation*}
    \left(\frac{r_1-r_2}{r_1+r_2}\right)\left(\frac{r_3-r_2}{r_3+r_2}\right)^{\pm 1} = 1,
  \end{equation*}
  which implies
  \begin{equation*}
    \frac{r_1}{r_2} = \pm \frac{r_3}{r_2}.
  \end{equation*}
  This contradicts the requirement that $(r_1, r_2, r_3)$ is a basis of $F$.
\end{proof}

\begin{cor} \label{cor:sl_noninvariance} If $\ord_F$ is the maximal order in a totally real cubic
  number field $F$, then the eigenform locus $\Omega{\cal E}_{\ord_F}^\iota$ is not invariant under the action
  of $\GLtwoRplus$.
\end{cor}

\begin{proof}
  If $\mathcal{O}_F$ is a maximal totally real cubic order, Proposition~\ref{prop:exist_admtriple}
  provides an admissible basis of some ideal in $\mathcal{O}$.  By Theorem \ref{thm:boundary_suff},
  the eigenform locus $\mathcal{E}_{\mathcal{O}_F}$ then intersects the corresponding irreducible
  boundary stratum, so  $\mathcal{E}_{\mathcal{O}_F}$ is not invariant
  by Theorem~\ref{thm:sl_noninvariance}.  
\end{proof}

It should be true also for nonmaximal orders $\mathcal{O}$ that no irreducible component of
$\Omega\mathcal{E}^\iota_\mathcal{O}$ is $\GLtwoRplus$-invariant. To achieve this using our approach
one needs to have information about which symplectic extensions of $\mathcal{O}$-modules arise from
cusps of a given irreducible component $X$ of $\mathcal{E}_\mathcal{O}$. This seems like a quite
delicate number theoretic question.

\section{Intersecting the eigenform locus with strata}
\label{sec:intersecting}

Given the results of the previous section, one might now ask whether 
the intersection of the eigenform locus with lower-dimensional
strata or the hyperelliptic locus is $\GLtwoRplus$-invariant.
Refined versions of the proof of Theorem~\ref{thm:sl_noninvariance} are likely to give
negative answers to this question as well, provided that the intersection has
large enough dimension so that the degeneration techniques can still be applied. 
\par
The most basic dimension question is, whether the eigenform locus lies in the 
principal stratum $\omoduli[3](1,1,1,1)$. Motivation
for this question is the following coarse heuristics. Almost all primitive 
\Teichmuller curves in genus two are obtained by intersecting the eigenform locus
with the minimal stratum $\Omega \moduli[2](2)$. In genus three, the minimal stratum
$\omoduli[3](4)$ has codimension three in the principal stratum. Hence if the 
the eigenform locus  $\mathcal{E}_\Ord$ lies generically in the principal stratum, then
the expected dimension for its intersection (in $\proj \omoduli[3]$) with
$\proj \omoduli[3](4)$ is zero - too small for a \Teichmuller curve.
On the other hand, components of $\mathcal{E}_\Ord$ that lie generically
not in the principal stratum are a potential source of \Teichmuller curves.
We show that such components do not exist.
\par
\begin{theorem} \label{thm:genprincipal}
For any given order $\Ord$ in a totally real cubic number field each 
component of the eigenform locus $\Omega\mathcal{E}_\Ord$ lies generically
in the principal stratum. 
\end{theorem}
\par
The theorem will follow from an intersection property of the real multiplication
locus with small strata.
\par
\begin{lemma} \label{lemma:exitsRingDeg}
Given a weighted admissible boundary stratum $\mathcal{S}$ of type $[4] \times^2 [4]$ there 
is a weighted admissible boundary stratum  $\mathcal{S}'$ of type 
$[3] \times^2 [3] \times^1 [3] \times^2 [3] \times^1$ which
is a degeneration of $\mathcal{S}$.
\end{lemma}
\par
\begin{proof}
Let $\pm r_1, \pm r_2$ be the weights in one component of curves parameterized by 
$\mathcal{S}$ and let $\pm r_2, \pm r_3$ be the weights in the other component. 
Admissibility implies that the $\ratls^+$-span of $Q(r_1), Q(r_2), Q(r_3)$
is a half-plane $H$ in $\reals^3$. In each of the two component we can pinch
further curves.  They necessarily carry the  weights $\pm (r_1 \pm r_2)$ 
resp.\  $\pm (r_2 \pm r_3)$, the signs depending on the choice
of the curve. By Lemma~\ref{lemma:linindepQ} we know that $Q(r_1 \pm r_2)$ does not lie in $H$.
In the Galois closure of F we calculate
$$ Q(r_1 \pm r_2) = Q(r_1) + Q(r_2) \pm (r_1^\sigma r_2^{\sigma^2} + 
r_2^\sigma r_1^{\sigma^2}).$$
Consequently the two choices of the sign lead to $Q$-images on different
sides of $H$. To produce  $\mathcal{S}'$ is thus suffices to pinch
some curve that acquires the weight $r_2+ r_3$  and also to pinch a
curve on the other component acquiring the weight $r_1 \pm r_2$ 
with the sign chosen such that $Q(r_2+ r_3)$ and $Q( r_1 \pm r_2)$
lie on opposite sides of $H$.
\end{proof}
\par
\begin{lemma} \label{lem:cusphasconepants}
For any given order $\Ord$ in a totally real cubic number field each 
cusp of the eigenform locus $\mathcal{E}_\Ord$ has non-empty
intersection with a boundary stratum parameterizing stable curves
without separating curves and all whose components are thrice punctured
projective lines (i.e.\ a pants decomposition without  separating curves).
\end{lemma}
\par
\begin{proof}
Since the boundary of the locus of $\RM$ is obtained by intersecting
with a divisor of $\barmoduli[3]$, every boundary stratum is contained in the
closure of a two-dimensional boundary stratum of $\RM$. Suppose this two-dimensional stratum
is an admissible weighted boundary stratum $\mathcal{S}$ with $\dim(\Span(\mathcal{S})) = 3$.
Case distinction and dimension count shows that $\mathcal{S}$ does not
contain any separating curves. Any degeneration of $\mathcal{S}$ is again
admissible. Thus in this case it suffices to pinch  enough non-separating curves to obtain
a pants decomposition.
\par
The only case of an admissible weighted boundary stratum $\mathcal{S}$ that gives a
two-dimensional component of $\partial \RM$ and with
the property  $\dim(\Span(\mathcal{S})) = 2$ is the stratum of type $[6]$. We can degenerate
this to a stratum of type $[4] \times^2 [4]$ without changing admissibility.
Now Lemma~\ref{lemma:exitsRingDeg} concludes the proof.
\end{proof}
\par
\paragraph{Proof of Theorem~\ref{thm:genprincipal}.}
  By Lemma~\ref{lem:cusphasconepants}, there exists a stable form on the boundary of each
  component of $\mathcal{E}_\mathcal{O}$ with each of the four irreducible components a thrice
  punctured sphere.  This form must then have four simple zeros, one in each irreducible component.
  Since the eigenform over a degenerate curve has simple zeros, so does the eigenform over a 
  general curve.
  \qed
\par

\section{Finiteness for the stratum $\Omega\moduli[3](3,1)$}
\label{sec:finiteness}

The aim of this section is to prove the following finiteness result for \Teichmuller curves using
the cross-ratio equation and the torsion condition of Theorem~\ref{thm:alg_prim_rm_torsion}. This
stratum contains one of the two known algebraically primitive \Teichmuller curves in genus three,
the billiard table $T(2,3,4)$ whose unique irreducible cusp in $\Omega\barmoduli[3]$ is described in
Example~\ref{ex:t234} below.
\par
\begin{theorem} \label{thm:fin31}
  There are only finitely many algebraically primitive \Teichmuller
  curves in the stratum $\Omega \moduli[3](3,1)$.
\end{theorem}

This theorem will follow from the following finiteness theorem for cusps.

\begin{theorem}
  \label{thm:cusps_finite}
  There are only finitely many points in $\proj\Omega\barmoduli[3](3,1)$ which are irreducible cusps of
  algebraically primitive \Teichmuller curves in $\proj\Omega\moduli[3](3,1)$.
\end{theorem}

\paragraph{Heights.}

The proof of Theorem~\ref{thm:cusps_finite} will require some facts about heights of subvarieties of
$\proj^n(\barratls)$ which we summarize here.  Unless stated otherwise, proofs can be found in
\cite{HiSi00}.

Consider a number field $K$ and a point $P = (x_0 : \ldots:x_n)\in \proj^n(K)$.  The \emph{absolute
  logarithmic Weil height} of $P$ is
\begin{equation*}
  h(P) = \frac{1}{[K:\ratls]}\log\prod_{v\in{M_K}} \max\{\|x_0\|_v, \ldots, \|x_n\|_v\},
\end{equation*}
where $M_K$ is the set of places of $K$, and $\|\cdot\|_v$ is the normalized absolute value at $v$.
The height $h(P)$ is unchanged under passing to an extension of $K$, so $h$ is a well-defined
function $h\colon \proj^n(\barratls)\to [0, \infty)$.

There is a more general notion of the height of a subvariety $V$ of $\proj^n(\barratls)$.  The precise
definition is not important for us; see \cite[p.~446]{HiSi00}.  We write
$h(V)\in[0,\infty)$ for the height of $V$.  

We will require the following properties of heights:

\begin{itemize}
\item (Northcott's Theorem) A collection of points in $\proj^n(\barratls)$ with uniformly bounded
  height and degree is finite.
\item The height of a hypersurface $V\subset\proj^n(\barratls)$ cut out by a polynomial $f$ is equal
  to the height of the vector of coefficients of $f$.
\item (Arithmetic \Bezout Theorem \cite{philippon95}) If $X$ and $Y$ are irreducible projective
  subvarieties of $\proj^n(\barratls)$ with $Z_1, \ldots, Z_n$ the irreducible components of $X\cap
  Y$, then for some constant $C$,
  \begin{equation*}
    \sum_{i=1}^n h(Z_i) \leq \deg(X) h(Y) + \deg(Y) h(X) + C\deg(X)\deg(Y).
  \end{equation*}
\item The height of a zero-dimensional subvariety of $\proj^n(\barratls)$ is the sum of the heights
  of its individual points.
\item \cite[Theorem~B.2.5]{HiSi00} Given a degree $d$ rational map $\phi\colon\proj^n\to\proj^m$
  defined over $\barratls$ with indeterminacy locus $Z$, we have for any
  $P\in\proj^n(\barratls)\setminus Z$
  \begin{equation}
    \label{eq:31}
    h(\phi(P)) \leq d h(P) + O(1).
  \end{equation}
\end{itemize}

Finally, there is the important theorem of Bombieri-Masser-Zannier \cite{BMZ99} on intersections of
curves with algebraic subgroups of the torus $\Gm^n$.  We define $\mathcal{H}_k\subset\Gm^n$ to be
the union of all algebraic subgroups of dimension at most $k$.

\begin{theorem}
  \label{thm:BMZ}
  Let $C\subset\Gm^n$ be a curve defined over $\barratls$ which is not contained in a translate of a
  subtorus.  Then $C\cap\mathcal{H}_{n-1}$ is a set of bounded height, and $C\cap\mathcal{H}_{n-2}$
  is finite.
\end{theorem}

The $\mathcal{H}_0$ case was proved in \cite{laurent}.  An effective version of this theorem was
proved in \cite{Ha08}.

\paragraph{Finiteness of cusps.}

We now begin working towards a weaker version of Theorem~\ref{thm:cusps_finite}, namely that there
are up to scale finitely possible triples of widths of cylinders of irreducible periodic directions
of algebraically primitive Veech surfaces in $\Omega\moduli[3](3,1)$.

We first introduce some notation which will be used throughout the next two paragraphs.  Consider
the moduli space $\moduli[0,8]$ of eight distinct labeled points in $\proj^1$.  We label these
points $p,q,x_1,x_2,x_3,y_1,y_2,y_3$.  Given a point $P\in\moduli[0,8]$, there is a unique (up to
scale) meromorphic one-form $\omega_P$ with a threefold zero at $p$, a simple zero at $q$, and
a simple pole at each $x_i$ or $y_i$.  We will usually make the normalization that $p=0$ and
$q=\infty$, and write
\begin{equation}
  \label{eq:34}
  \omega_P = \frac{z^3dz}{\prod_{i=1}^3(z-x_i)(z-y_i)}.
\end{equation}
Under this normalization, $\moduli[0,8]$ is naturally identified with an open subset of $\proj^5$
via $P\mapsto(x_1:\ldots:y_3)$.  We use this identification to define the Weil height $h$ on $\moduli[0,8]$.
We define $S(3,1)\subset\moduli[0,8]$ to be the locus of $P$
such that $\omega_P$ satisfies the opposite-residue condition $\Res_{x_i}\omega_P =
-\Res_{y_i}\omega_P$ for each $i$.   The variety $S(3,1)$ is locally parameterized by the projective
4-tuple consisting of the three residues and one relative period, so $S(3,1)$ is three-dimensional.

We define the cross-ratio morphisms $Q_i\colon S(3,1)\to\Gm$ and $R_i\colon S(3,1)\to\Gm$ by
\begin{equation*}
  Q_i = [p,q,y_i,x_i] \qtq{and} R_i = [x_{i+1},y_{i+1},y_{i+2}, x_{i+2}],
\end{equation*}
with indices taken mod $3$.  In the standard normalization of \eqref{eq:34}, $Q_i = y_i/x_i$.  We
define $Q,\CR\colon S(3,1)\to\Gm^3$ by $Q = (Q_1,Q_2,Q_3)$ and $\CR = (R_1,R_2,R_3)$.  We define
$\Res\colon S(3,1)\to\proj^2$ by $\Res(P) = (\Res_{x_i}\omega_P)_{i=1}^3$.  Finally, given $\zeta =
(\zeta_1, \zeta_2, \zeta_3)\in\Gm^3$, we define $S_\zeta(3,1)\subset S(3,1)$ to be the locus where
$Q_i = \zeta_i$ for each $i$.

\begin{lemma}
  \label{lem:cusp_normal_form}
  Any irreducible stable form $(X, \omega)\in\pobarm[3](3,1)$ which is a limit of a cusp of an
  algebraically primitive \Teichmuller curve $C\subset\pom[3](3, 1)$ is equal to $\omega_P$ for some
  $P\in S_{(\zeta_1,\zeta_2,\zeta_3)}(3,1)\cap \CR^{-1}(T)$ with the $\zeta_i$ nonidentity roots of
  unity and $T\subset\Gm^3$ a proper algebraic subgroup.  Moreover, $\Res(P)$ is a basis of some
  totally real cubic number field.
\end{lemma}

\begin{proof}
  By \cite{masur75}, a limit of an irreducible cusp of $C$ is an irreducible stable form
  with two zeros of order $3$ and $1$, and $6$ poles whose residues (up to sign and constant
  multiple) are the widths of the $3$ horizontal cylinders of $(X, \omega)$.    Since
  a form generating $C$ is an eigenform for real multiplication by Theorem~\ref{thm:alg_prim_rm_torsion} and
  the residues $r_i$ are widths of cylinders, they are a basis of the trace field by
  Lemma~\ref{lem:heightknown}.

  That the $\zeta_i$ are roots of unity follows from the torsion condition of
  Theorem~\ref{thm:alg_prim_rm_torsion}.  By Abel's theorem, there is an $n$ such that for each $(Y,
  \eta)\in C$ we may find a degree $n$ meromorphic function $Y\to\proj^1$ with a single pole of
  order $n$ at one zero of $\eta$ and a zero of order $n$ at the other zero of $\eta$.  Taking a
  limit of such functions (this is justified in \cite[p.{}~9]{Mo08}), we obtain a meromorphic
  function $f\colon X\to \proj^1$ with a single zero at $p$ and a single pole at $q$.  In the
  normalization of \eqref{eq:34}, such a function must be of the form $f(z) = z^p$.  Since $x_i$ and
  $y_i$ are identified, we must have $x_i^p = y_i^p$, as desired.

  That $\CR(P)$ lies on an algebraic subgroup follows directly from Theorems~\ref{thm:boundary_nec}
  and \ref{thm:explCREQ}.
\end{proof}

\begin{lemma}
  \label{lem:dimension_S}
  If the $\zeta_i$ are not all cube roots of unity, then $S_{(\zeta_1, \zeta_2, \zeta_3)}(3,1)$ is
  zero-dimensional.  Otherwise $S_{(\zeta_1, \zeta_2, \zeta_3)}(3,1)$ has a single one dimensional
  component, a line in $\moduli[0,8]$.  Specifically, if $\zeta_i = e^{2 \pi i/3}$ for all $i$, this
  component is the line $L$ cut out by the equation,
  \begin{equation*}
    x_1 + x_2 + x_3 = 0,
  \end{equation*}
  under the normalization $p=0$ and $q=\infty$.
\end{lemma}

\begin{proof}
  $S_{(\zeta_1, \zeta_2, \zeta_3)}(3,1)$ is cut out by the equations $y_i = \zeta_i x_i$ and
  \begin{equation}
    \label{eq:33}
    D_{i} = \zeta_i^3 \prod_{j\neq i}(x_i - x_j)(x_i - \zeta_j x_j) - \prod_{j\neq i} (\zeta_i x_i -
    x_j)(\zeta_i x_i - \zeta_j x_j).
  \end{equation}
  
  Suppose that $S_{(\zeta_1, \zeta_2, \zeta_3)}(3,1)$ has a positive dimensional component, and suppose
  first that (say) $\zeta_1$ is not a cube root of unity.  Then there is a homogeneous polynomial
  $P$ of some degree $d < 4$ which divides $D_k$ for all $k$.  Expanding $D_k$, we obtain
  \begin{equation*}
    D_k = x_k^4(\zeta_k^3-\zeta_k^4) + \cdots +\zeta_{k+1}x_{k+1}^2 \zeta_{k+2} 
    x^2_{k+2}(\zeta_k^3-1),
  \end{equation*}
  with indices taken mod $3$.  Because each $D_k$ contains $x_k^4$ with non-zero coefficient, each monomial
  $x_k^d$ appears in $P$ with non-zero coefficient. We have
  $$P(0,x_2,x_3) = \alpha_2x_2^d + \alpha_3x_3^d + \ldots \quad
  \mid \quad D_1(0,x_2,x_3) = \zeta_2 x_2^2 \zeta_3 x_3^2 (\zeta_1^3-1).$$
  This is not possible since the
  $\alpha_i$ are nonzero and $\zeta_1^3\neq 1$.

  Now suppose that $\zeta_i = e^{2\pi i/3}$ for all $i$.  A simple computation shows that $P =
  x_1+x_2+x_3$ divides each $D_k$, so $L$ is a component of $S_{(\zeta_1, \zeta_2,
  \zeta_3)}(3,1)$.  An argument as above shows that the quotients $D_k/P$ have no common factor, so
  $L$ is the only one-dimensional component.

  Finally, suppose the $\zeta_i$ are arbitrary cube roots of unity.  Replacing some of the cube
  roots of unity $e^{2\pi i/3}$ with their complex conjugates amounts to
  swapping the corresponding $x_i$ and $y_i$.  Thus the new  $S_{(\zeta_1, \zeta_2,
  \zeta_3)}(3,1)$ is simply a rotation of the old one.
\end{proof}

\begin{lemma}
  \label{lem:CRE_not_canonical}
  No one-dimensional component of any $S_{(\zeta_1,\zeta_2,\zeta_3)}(3,1)$ lies in $CR^{-1}(T)$ for $T$ any
  algebraic subgroup of $\Gm^3$.  
\end{lemma}

\begin{proof}
  By Lemma~\ref{lem:dimension_S}, we need only to show that the equation
  \begin{equation}
    \label{eq:32}
    R_1^{a_1}R_2^{a_2}R_3^{a_3}= \zeta
  \end{equation}
  is not satisfied identically on the line $L$ cut out by $x_1+x_2+x_3=0$.  We may assume without
  loss of generality that $a_1\neq 0$.  Normalizing so that $x_1=1$. and setting $x_3 = -1-x_2$, the
  left hand side of \eqref{eq:32} becomes a rational function $R$ in the single variable $x_2$ which
  must be identically equal to $\zeta$.  The factor $(2x_2+1)$ lies in the denominator of $R_1$ and
  appears nowhere else in $R$.  Since $\cx[x]$ is a unique factorization domain, it follows that $R$
  is not constant.
\end{proof}

\begin{prop} \label{prop:fin31res}
  There is a finite number of projectivized triples of real cubic numbers $(r_1:r_2:r_3)$
  such that for any irreducible periodic direction on any $(X,\omega) \in
  \Omega\moduli[3](3,1)$ generating an algebraically primitive \Teichmuller curve,  the projectivized
  widths of the horizontal cylinders is one of the triples $(r_1:r_2:r_3)$.

  In particular, there are only a finite number of trace fields $F$ of algebraically
  primitive \Teichmuller curves in $\Omega\moduli[3](3,1)$.
\end{prop}

\begin{proof}
  By Northcott's Theorem, we need only to give a uniform bound for the heights of the triples
  $(r_1:r_2:r_3)$ of widths of cylinders, or equivalently of residues of limiting irreducible stable
  forms satisfying the conditions of Lemma~\ref{lem:cusp_normal_form}.

  Let $T_i(\zeta_1,\zeta_2,\zeta_3)\subset\proj^5$ be the subvariety cut out by the polynomial $D_i$
  of \eqref{eq:33}.  Since $\|\zeta\|_v=1$ for any root of unity $\zeta$ and place $v$, it follows
  directly from the definition of the Weil height that there is a uniform bound on the heights of
  the $T_i(\zeta_1, \zeta_2, \zeta_3)$, independent of the root of unity.  Since
  $S_{(\zeta_1,\zeta_2,\zeta_3)}(3,1)$ is the intersection of the $T_i(\zeta_1,\zeta_2,\zeta_3)$ and the
  hypersurfaces defined by $x_i - \zeta_iy_i$ (which have height $0$), it follows from the
  Arithmetic \Bezout theorem that the varieties $S_{(\zeta_1, \zeta_2, \zeta_3)}(3,1)$ have uniformly
  bounded height.  Thus the zero dimensional components of the $S_{(\zeta_1, \zeta_2, \zeta_3)}(3,1)$ have
  uniformly bounded height as well.  By \eqref{eq:31}, the heights of these points increases by a
  bounded factor under the rational map $\Res$.  Thus the residue triples arising from the zero
  dimensional components of the $S_{(\zeta_1, \zeta_2, \zeta_3)}(3,1)$ have uniformly bounded heights.

  By Lemma \ref{lem:dimension_S}, it only remains to bound the heights of the residue triples
  arising from the line $L\subset \moduli[0,8]$ cut out by the equations $x_1+x_2+x_3=0$ and $x_i-\theta
  y_i$ for each $i$, where $\theta = e^{2\pi i/3}$.  Suppose a point $P\in L$ is a cusp of an
  algebraically primitive \Teichmuller curve.  By Lemma~\ref{lem:cusp_normal_form}, $\Res(P)$
  must be defined over a cubic number field, and $\CR(P)$ must lie in $\mathcal{H}_2$.  Let
  $L'\subset L$ be the set of points satisfying these two conditions.  If $\Res(P)$ lies in $\proj^2(F)$
  for some cubic number field $F$, then $P$ is defined over $F(\theta)$.  Thus $L'$ and $\CR(L')$ consist of
  points of degree at most $9$.  By
  Lemma~\ref{lem:CRE_not_canonical}, $\CR(L)$ is not contained in a translate of a subtorus of
  $\Gm^3$.  Thus Theorem~\ref{thm:BMZ} applies, and we conclude that $\CR(L)\cap\mathcal{H}_2$ is a
  set of points of bounded height.  Therefore $\CR(L')$ is finite by Northcott's theorem.
  The map $\CR$ is finite on $L$ by Lemma~\ref{lem:CRE_not_canonical}, so $L'$ and thus $\Res(L')$
  are finite as well.  Thus there are at most finitely many residue triples arising from $L$ as desired.
\end{proof}

\begin{remark} {\rm All of the estimates in the preceding propositions, 
in particular Theorem~\ref{thm:BMZ} and the height estimates are
effective. It is thus possible in principle to give a complete list of triples
$(r_1,r_2,r_3)$ that may appear in Proposition~\ref{prop:fin31res}.
Unfortunately the available bounds are so bad that this is currently not feasible. 
}\end{remark}
\par
\begin{example}{\rm \label{ex:t234} There is one known example of an algebraically primitive
    \Teichmuller curve in $\Omega\moduli[3](3,1)$, discovered in \cite{ks}.  It is the surface $(X,
    \omega)$ obtained by unfolding the $(2,3,4)$ triangle, shown in Figure~7 of \cite{ks}.  The
    trace field of $(X, \omega)$ is the field $K = \ratls[v]/P(v)$ of discriminant $81$, where $P(v)
    = v^3 - 3v +1$ has a solution $v=2\cos(2\pi/9)$.  The vertical direction is of type $[5]
    \times^3 [3]$, and the circumferences of the vertical cylinders are
    $$ \begin{array}{lcl} w_1 = 2\cos(3\pi/9) &=& 1 \\
      w_2 = -2(\cos(3\pi/9)+\cos(8\pi/9))&=& v^2 + v - 1\\
      w_3= 2(\cos(2\pi/9)+ \cos(3\pi/9)+\cos(8\pi/9) &=& -v^2-3\\
      w_4 = 2\cos(4\pi/9) &=& v^2-2. \end{array} $$ One can check that the $w_i$ form an admissible
    subset of $K$.  \par
    The horizontal direction is irreducible periodic, with cylinder widths, $$ \begin{array}{lcl}
      r_1 =  -(2w_1+w_2+w_3+w_4) &=& -v^2 -v \\
      r_2 =  w_1+w_2+w_3&=& v+1\\
      r_3 = -(3w_1+3w_2+2w_3+w_4)&=&-2v^2-3v+2. \\
    \end{array} $$ In fact, this is the unique irreducible cusp of the \Teichmuller curve spanned by
    $(X, \omega)$.  This cusp lies on the line $L$ of Lemma~\ref{lem:dimension_S}, as we will now
    show.  The irreducible cusp $(X_0, \omega_0)$ is of the form
    \begin{equation} \label{eq:6}
      \omega_0 = C\frac{z^3dz}{\prod(z-x_i)(z-\zeta_i x_i)} =
      \sum\left(\frac{r_i}{z-x_i}-\frac{r_i}{z-\zeta_i x_i}\right)
    \end{equation}
    for some constant
    $C$ and roots of unity $\zeta_i = e^{2\pi i p_i/q_i}$.  To calculate the $\zeta_i$, we consider
    a relative period.  There is a path joining the two zeros of $(X, \omega)$ of period $\sum
    r_i/3$, so the integral of $\omega_0$ along a path $\gamma$ joining $0$ to $\infty$ must be
    $\sum (a_i + 1/3)r_i$, for some integers $a_i$.  From \eqref{eq:6}, we calculate
    \begin{equation*}
      \frac{1}{3}\sum r_i = \int_\gamma \omega_0 = \sum r_i\log\zeta_i = \sum
      r_i\frac{p_i}{q_i},
    \end{equation*}
    so we must have $\zeta_i = e^{2\pi i/3}$ for each $i$ by
    the linear independence of the $r_i$.  One then calculates that up to scale there is a unique
    triple $(x_1,x_2,x_3)$ so that $\omega_0$ has the residues $r_i$,
    \begin{equation*} x_1 = 1,
      \quad x_2 = 2 - v^2, \qtq{and} x_3 = v^2-3.
    \end{equation*} Since the sum of the $x_i$ is
    $0$, this cusp lies on the line $L$.  }
\end{example}
\par
Theorem~\ref{thm:cusps_finite} now
follows directly from Proposition~\ref{prop:fin31res} and the following proposition.
\begin{prop}
  \label{prop:31fixxy}
  Given a basis $(r_1,r_2,r_3)$ over $\ratls$ of a totally real cubic number
  field, there are only finitely cusps of algebraically primitive \Teichmuller curves in
  $\Omega\moduli[3](3,1)$ having residues $(r_1,r_2,r_3)$.
\end{prop}
\begin{proof}
  Consider the
  variety $C = \Res^{-1}(r_1:r_2:r_3)\subset S(3,1)$ of forms having residues $\pm r_i$ and two
  zeros of order $3$ and $1$.  A dimension count shows that $C$ is at least one-dimensional.  In
  fact, $C$ is exactly one-dimensional, as $C$ is locally parameterized by the single relative
  period of the forms $\omega_P$.  Let $C_0$ be a component of $C$.  We suppose that $C_0$ contains
  infinitely many cusps of algebraically primitive \Teichmuller curves and derive a contradiction.
  Consider the image $Q(C_0)\subset(\cx^*)^3$.  We claim that $Q(C_0)$ is a curve.  If not, and
  $Q(C_0)= (\zeta_1,\zeta_2,\zeta_3)$, then $C_0$ is a component of $S_{(\zeta_1,\zeta_2,\zeta_3)}$.
  Then $C_0$ must be the line $L$ of Lemma~\ref{lem:dimension_S}.  It is easily checked that $\Res$
  is not constant along $L$, so this is impossible.  Now since $C_0$ contains infinitely many cusps
  of \Teichmuller curves, $Q(C_0)$ must contain infinitely many torsion points of $(\cx^*)^3$ by
  Lemma~\ref{lem:cusp_normal_form}.  From this it follows that $Q(C_0)$ is a translate of a subtorus
  of $(\cx^*)^3$ by a torsion point.  This is a consequence of the main result of \cite{laurent}.
  It can also be seen by first applying Theorem~\ref{thm:BMZ} to show that $Q(C_0)$ lies on a
  subtorus $T\subset(\cx^*)^3$, then applying Theorem~\ref{thm:BMZ} again to $T$.  We now claim that
  $Q(C_0)$ is in fact a subtorus of $(\cx^*)^3$, rather than a translate.  To see this, it suffices
  to show that the identity $(1,1,1)$ is contained in the closure of $Q(C_0)$.  Given a form $(X,
  \omega)$ representing a point $P\in C_0$, we may choose a saddle connection joining the two zeros
  $p$ and $q$.  Following \cite{emz}, we may collapse this saddle connection (and possibly
  simultaneously a homologous saddle connection) to obtain a path in $C_0$ such that the zeros $p$
  and $q$ collide.  Under this deformation, each cross-ratio $Q_i$ tends to $1$, so $(1,1,1)$ is in
  the closure, as desired.  It remains to show that $Q(C)$ is not a subtorus of $(\cx^*)^3$. If this
  were true, we could find roots of unity $\zeta_i$ and a projective triple $(x_1(a):x_2(a):x_3(a))$
  depending on a parameter $a$, such that for all $a \in \cx$ the differential $$ \omega_\infty =
  \left( \sum_{i=1}^3 \frac{r_i}{z-x_i(a)} - \frac{r_i}{z-\zeta_i^a x_i(a)} \right) dz =
  \frac{p(z)dz}{\prod_i(z-x_i(a))(z-\zeta_i^ax_i(a))}$$ has a triple zero at $z=0$ and a simple zero
  at $z=\infty$. The vanishing of the $z^4$-term of $p(z)$ implies.  $$ \sum r_i x_i (1-\zeta_i^a) =
  ,0$$ and the linear term (divided by $x_1x_2x_3$) also yields a linear equation.  Using the
  normalization $x_1=1$ we may solve the two linear equations for $x_2$ and $x_3$.  We then take the
  limit of $x_2$ and $x_3$ as $a\to 0$, applying l'H\^opital's rule twice. If we let $\zeta_i =
  e^{2\pi i q_i}$ for some $q_i \in \ratls$, we obtain
  \begin{equation}
    \label{eq:3}
    x_2(0) =
    \frac{q_3r_3-q_1r_1}{q_2r_2-q_3r_3} \qtq{and} x_3(0) = \frac{q_2r_2 - q_1r_1}{q_3r_3 - q_2r_2}.
  \end{equation}
  Taking the derivative of the $z^2$-term of $p(z)$ with respect to $a$ at $a=0$ and
  making the substitution \eqref{eq:3}, we obtain $$
  (q_3r_3-r_1q_1)(q_2r_2-q_1r_1)(q_1r_1+q_2r_2+q_3r_3)=0.$$ The $\ratls$-linear independence of the
  $r_i$ yields the desired contradiction.
\end{proof}

\paragraph{Finiteness of \Teichmuller  curves.}

Theorem~\ref{thm:fin31} follows from Theorem~\ref{thm:cusps_finite} and the following
proposition.

\begin{prop}
  \label{prop:finResfin}
  Suppose that there are at most finitely many
  irreducible cusps in $\proj\Omega\barmoduli[g]$ of algebraically primitive \Teichmuller curves in
  the stratum $\proj\Omega\moduli(m,n)$ (resp.\ in a component of the stratum
  $\proj\Omega\moduli(2g-2)$).  Then there are at most finitely many algebraically primitive
  \Teichmuller curves in $\proj\Omega\moduli(m,n)$ (resp.\ in this component of the stratum
  $\proj\Omega\moduli(2g-2)$).
\end{prop}

\begin{proof}
  Suppose $(X, \omega)\in\Omega\moduli(m,n)$
  generates an algebraically primitive \Teichmuller curve.  Let $\theta$ be an irreducible periodic
  direction on $(X, \omega)$, and let $I$ and $J$ each be either a saddle connection or periodic
  direction of slope $\theta$.  Since lengths of saddle connections or circumferences of cylinders
  of a given slope are unchanged under passing to the corresponding limiting stable form, from
  finiteness of irreducible cusps we obtain a constant $C$, depending only on the stratum, so that
  \begin{equation}
    \label{eq:30}
    \frac{1}{C} < \frac{\length(I)}{\length(J)} < C
  \end{equation}
  for
  $I$ and $J$ any saddle connections or closed geodesics of the same slope.  There is an irreducible
  periodic direction on $(X, \omega)$ by Lemma~\ref{le:31topo1}.  Choose one, and apply a rotation
  of $\omega$ so that it is horizontal.  Let $C_1, \ldots, C_g$ be the horizontal cylinders of $(X,
  \omega)$.  There must be some cylinder $C_i$ having one of the two zeros in its bottom boundary
  component and the other zero in the top.  Take a saddle connection $\gamma$ contained in $C_i$ and
  connecting these zeros.  Applying the action of a matrix $\left(
    \begin{smallmatrix}
      1 & t \\
      0 & 1
    \end{smallmatrix}
  \right)\in\SLtwoR$, we may take $\gamma$ to be vertical, whence the
  vertical direction is irreducible periodic with $g$ cylinders $D_i, \ldots, D_g$.  By
  Lemma~\ref{lem:heightknown}, after normalizing by the action of a diagonal element of
  $\GLtwoRplus$, we have $w(C_i) = r_i$ and $h(C_i)= s_i$ (where we write $w(C)$ and $h(C)$ for the
  height and width of the cylinder $C$) for some basis $(r_i)$ of $F$ (with a chosen real embedding)
  and dual basis $(s_i)$.  By finiteness of cusps, there are only finitely many possibilities for
  the $r_i$, and thus the $s_i$, so we may take them to be fixed.  Since the saddle connection
  $\gamma$ crosses only one cylinder, its length is bounded by a constant depending only on the
  stratum.  This implies that the $w(D_i)$ are bounded as well by \eqref{eq:30}.  Therefore the
  intersection matrix $(B_{ij}) = (C_i\cdot D_j)$ has bounded entries, and we may take it to be
  fixed.  The widths and heights of the $D_j$ are determined by $B$, as well as the widths and
  heights of the $C_i$, so we may take them to be fixed as well.  Now each intersection of $C_i$ and
  $D_j$ is isometric to a rectangle $R_{ij}$ of width $h(D_j)$ and height $h(C_i)$.  Thus the
  surface $(X, \omega)$ may be built by gluing the finite collection of rectangles consisting of
  $B_{ij}$ copies of $R_{ij}$ for each index $(i,j)$.  As there are only finitely many gluing
  patterns for a finite collection of rectangles, there are only finitely many possibilities for
  $(X, \omega)$.  \par In the case $\proj\Omega\moduli(2g-2)$ the same argument works. It is even
  simplified by the fact that every direction is irreducible.
\end{proof}

\section{Finiteness conjecture for  $\Omega\moduli[3](4)^{\rm hyp}$}

\label{sec:finconj} In this section, we
give numerical and theoretical evidence for the following conjecture, which together with
Proposition~\ref{prop:finResfin} implies Conjecture~\ref{conj:intro} for the case of the stratum
$\omoduli[3](4)^{\rm hyp}$.
\par
\begin{conj}
  \label{conj:4resfinite}
  There are only a finite number of possibilities for the projectivized
  triples $(r_1:r_2:r_3)$ of widths of cylinders of algebraically primitive \Teichmuller curves in
  $\omoduli[3](4)^{{\rm hyp}}$.
\end{conj}
\par
Everything in this section should  hold as well  for the other component $\omoduli[3](4)^{\rm odd}$ of
$\omoduli[3](4)$, but we stick to the hyperelliptic component for simplicity.  The hyperelliptic  component contains the other of the two known examples of
algebraically primitive \Teichmuller curves in genus three, Veech's $7$-gon.  We describe the stable
form which is the limit of the unique cusp of this curve in Example~\ref{ex:veech7} below. Finally we will give the
algorithm for searching any eigenform locus for \Teichmuller curves in $\Omega\moduli[3](4)$ which
is used to prove Theorem~\ref{thm:numevidence}.

\par
\paragraph{Finiteness for fixed admissibility
  coefficients.}
Recall from \eqref{eq:sumci} that if $\mathcal{S}$ is a weighted admissible
boundary stratum of type $[6]$, then the weights $r_i$ satisfy $\sum_{i=1}^3 c_i/ r_i =0$ for some
$c_i \in \zed$.  We call the triple $(c_1,c_2,c_3)$ of coprime integers the {\em admissibility
  coefficients} of the $r_i$.
\par
\begin{prop} \label{prop:fin4fixedci} For any fixed triple
  $(c_1,c_2,c_3)$ there is only a finite number of algebraically primitive \Teichmuller curves in
  $\omoduli[3](4)^\hyp$ that possess a direction whose cylinders have lengths with admissibility
  coefficients $(c_1,c_2,c_3)$.
\end{prop}
\par
This has as obvious consequence:
\par
\begin{cor}
  In $\omoduli[3](4)^\hyp$ there is only a finite number of algebraically primitive \Teichmuller
  curves meeting the infinite collection of weighted boundary strata provided by the
  algorithm in Proposition~\ref{prop:exist_admtriple}.
\end{cor}
\par
\paragraph{The limiting
  differential in the hyperelliptic case.}
We want to make the cross-ratio coordinates more
explicit and therefore normalize the hyperelliptic involution on the stable curve $X_\infty$
corresponding to a \Teichmuller curve in $\omoduli[3](4)^\hyp$. Necessarily, $X_\infty$ is
irreducible, and consequently the desingularization of $X_\infty$ is a $\proj^1$ with coordinate
$z$, where we may normalize the hyperelliptic involution to be $z \mapsto -z$ and $z=0$ is the
$4$-fold zero.  The preimages of the nodes are $\pm x_i$ for $i=1,2,3$, and we will at some points
in the sequel use the full threefold transitivity of M\"obius transformations to normalize moreover
$x_1=1$. The differential $\omega_\infty$ pulls back on the normalization to
\begin{equation}
  \label{eq:35}
  \omega_\infty = \sum_{i=1}^3 \left(\frac{r_i}{z-x_i} - \frac{r_i}{z+x_i} \right)dz =
  \frac{Cz^4}{\prod_{i=1}^3(z^2-x_i^2)} dz
\end{equation} for some constant $C$ that can obviously
be expressed in the $r_i$ and $x_i$.  Coefficient comparison yields the two equations
\begin{gather}
  \label{eq:1stra4hyp}
  \sum_{i=1}^3 r_i x_{i+1}x_{i+2} =0 \\
  \label{eq:2stra4hyp} \sum_{i=1}^3 r_i x_i (x_{i+1}^2 + x_{i+2}^2) =0,
\end{gather} where indices
are to be read mod $3$.   The cross-ratio map $\CR$ as defined by Equation~\ref{CR6}
is given by $$ \CR = (R_1,R_2,R_3), \quad \text{where} \quad R_i = \left(\frac{x_{i+1}+x_{i+2}}
  {x_{i+1}-x_{i+2}}\right)^2.$$ It will be convenient to use that $\CR$ factors as a composition of
the squaring map and the rational map $\CR_0 \colon \proj^2 \to (\cx^*)^3$ defined by $\CR_0 =
(R_1',R_2',R_3')$, where
$$ R_i'(x_1:x_2:x_3) = \frac{x_{i+1}+x_{i+2}}{x_{i+1}-x_{i+2}}.$$
\par
\begin{example} \label{ex:veech7}
{\rm Veech's $7$-gon curve lies in this stratum, and we conjecture it is the
only one. 
Let $F =\ratls[v]/\langle v^3 + v^2 - 2v - 1 \rangle$ be the
cubic field of discriminant $D=49$. 
There is a unique cusp whose cylinder widths are projectively
equivalent to
$$ r_1 = 1, \quad r_2 = v^2 + v - 2, \quad r_3 = v^2 - 2,$$
with $v=2\cos(2\pi/7)$.
Since 
$$\sum_i \frac{1}{r_i} = 0 \qtq{and} \NFQ(r_i) = 1$$
for all $i$, the cross-ratio exponents are all $1$.
Only one of the three solutions to equations~\eqref{eq:1stra4hyp} and \eqref{eq:2stra4hyp}
satisfies the cross-ratio equation $\prod R_i=1$, namely
$$ x_1 = 1, \quad x_2 = -v^2 - v + 1, \quad x_3 = v^2 + v - 2.$$
\par
Note in comparison with Proposition~\ref{prop:unlikelycanc} below 
that here the $c_i$, the $\NFQ(r_i)$ and also the moduli
of the cylinders are all one.  That is, all the auxiliary parameters are arithmetically
as simple as possible.
}
\end{example}
\par
Inside the domain of $\CR_0$ the rationality condition $\sum_{i=1}^3 c_i/ r_i =0$ 
together with the opposite-residue condition defines a curve $Y = Y_{(c_1,c_2,c_3)}$.
We want to apply Theorem~\ref{thm:BMZ} to this curve and now check
the necessary hypothesis.

\begin{lemma} \label{lemma:notinsubtorus} Let $X \subset (\cx^*)^n$ be an irreducible curve whose closure in
  $\cx^n$ contains points $P_1,\ldots,P_n$ where $P_i=(p_{i1},\ldots,p_{in})$ and where for all $i$
  we have $p_{ii}=0$ while $p_{ij} \neq 0$ for $i \neq j$.  Then $X$ is not contained in the
  translate of an $n-1$-dimensional algebraic subtorus in $\cx^*$.
\end{lemma}
\par
\begin{proof}
  Let $z_i$ be coordinates of $\cx^n$ and suppose on the contrary that $X$ is contained in such a
  torus given by the equation $\prod z_i^{b_i} = t$ for some $b_i \in \zed$ not all zero and $t \in
  \cx^*$. This equation holds on $X$, thus on its closure. Plugging in $P_i$ implies $b_i = 0$.
  Using all the $P_i$ we obtain the contradiction that all of the $b_i$ are zero.
\end{proof}
\par
\begin{cor}
  The curve $\CR_0(Y)$ does not lie in a translate of an algebraic subtorus in $(\cx^*)^3$.
\end{cor}
\par
\begin{proof}
  Normalizing $x_1=1$ and applying the degeneration $x_2 \to 0$ to $\CR_0(Y)$ we obtain the limit
  point $(1,0,1) \in \cx^3$.  Permuting coordinates, we obtain a limit point where any single
  coordinate vanishes, so we may apply Lemma~\ref{lemma:notinsubtorus} after verifying irreducibility.
  \par
  A computer algebra system with an algorithm for computing Weierstrass normal form (e.g.\ MAPLE)
  exhibits a birational map from $ Y_{(c_1,c_2,c_3)}$ to the curve
$$\widetilde{Y}: \quad y^2=
c_1^2x^6-3c_1^2x^5+3c_1^2x^4+(c_2^2-c_3^2-c_1^2)x^3+3c_3^2x^2-3c_3^2x+c_3^2. $$
A straightforward calculation shows that the right hand side is not a perfect square
for any $(c_1,c_2,c_3)$. Consequently, $\widetilde{Y}$ is irreducible and thus also
$ Y_{(c_1,c_2,c_3)}$.
\end{proof}
\par
\paragraph{Proof of Proposition~\ref{prop:fin4fixedci}.}
The preceding lemma allows us to apply Theorem~\ref{thm:BMZ}. As a consequence, 
the height of any point $(R_1,R_2,R_3) \in \CR_0(Y)$ that lies on an algebraic
subtorus is bounded. This applies in particular to the torus given by the
cross-ratio equation. More precisely, since the degree of $Y$ is independent
of the $c_i$  we deduce from \cite[Theorem~1]{Ha08} that there is a constant $C_1$ such that
\begin{equation} \label{eq:hR_i1}
h((R_1,R_2,R_3)) \leq C_1(1+h(c_1:c_2:c_3)).
\end{equation}
Moreover, the $R_i$ lie in a field of degree at most three over $F$ as can be
checked solving \eqref{eq:1stra4hyp} and \eqref{eq:2stra4hyp}. Consequently, 
by Northcott's theorem, there is only a finite number of possible $R_i$ lying
on $\CR(Y)$ and satisfying the cross-ratio equation.
\qed

\paragraph{Unlikely cancellations.}

We now show that if the finiteness conjecture fails, then there has to be a sequence
of \Teichmuller curves with the admissibility coefficients  $c_i$ becoming
more and more complicated simultaneously for all the directions on the
generating flat surface, but meanwhile there are miraculously enormous cancellations
making the cross-ratio exponents much smaller that the $c_i$. 
\par 
\begin{prop} \label{prop:unlikelycanc}
Suppose Conjecture~\ref{conj:4resfinite} fails for $\omoduli[3](4)^\hyp$. Then 
there exists a sequence of \Teichmuller
curves $\{C_n\}_{n \in \nats}$ generated by flat surfaces $(X_n,\omega_n)$ such that 
for every periodic direction $\theta$ on the $X_n$
\begin{itemize}
\item[i)]  the residues $r_{i,n,\theta}$ have admissibility coefficients 
$(c_{1,n,\theta},c_{2,n,\theta},c_{3,n,\theta})$ with the height lower bound
$$ h(c_{1,n,\theta},c_{2,n,\theta},c_{3,n,\theta})\geq n,$$
\item[ii)] and on the other hand the cross-ratio exponents have upper bound
$$ |a_i| \leq  C_2 (1+h(c_{1,n,\theta},c_{2,n,\theta},c_{3,n,\theta}))^2 $$
\end{itemize}
for some constant $C_2$ independent of $n$ and $\theta$.
\end{prop}
\par
Note that in ii) the height on the right is logarithmic in the the $c_i$, 
whereas on the left of the inequality we have the usual absolute value.
\par
As preparation we examine the image $Z \subset (\cx^*)^3$ of $\omoduli[3](4)^\hyp$ under $\CR$. 
\par
\begin{lemma} \label{le:Zoa}
There is no translate of an algebraic subtorus of $(\cx^*)^3$ contained in 
$Z$.
\end{lemma}
\par
\begin{proof}
  It suffices to prove the claim for the image $Z_0$ of $Z$ under $\CR_0$.
  The variety $Z_0$ is cut out by the equation
  \begin{equation}
    \label{eq:36}
    R_1'R_2' + R_1'R_3' + R_2'R_3' +1 =0.
  \end{equation}
  This variety does not contain the image of $y\mapsto (\alpha_1 y^{n_1}, \alpha_2 y^{n_2}, \alpha_3
  y^{n_3})$ for any nonzero $\alpha_i$ and integers $n_i$, as substituting the $\alpha_i y^{n_i}$
  into the left hand side of \eqref{eq:36} always yields a nonzero Laurent series in $y$.
\end{proof}
\par
\paragraph{Proof of Proposition~\ref{prop:unlikelycanc}.}
The existence of a sequence satisfying i) follows from
Proposition~\ref{prop:fin4fixedci}. That this sequence moreover satisfies ii)
follows from a close examination of the proof of \cite[Theorem~1]{Ha08}.
We fix $\theta$ and $n$ and drop these indices. We write $\underline{c} = (c_1: c_2: c_3)$.
We follow the notation in loc.~cit. The idea of Habegger is to use the geometry
of numbers to construct a subtorus $H_u$ of $(\cx^*)^3$ determined by a triple
$u=(u_1,u_2,u_3)$ of integers depending on a parameter $T$ such that for
a point $p=(R_1,R_2,R_3)$ in the intersection of $W = \CR_0(Y_{\underline{c}})$ and a torus of
codimension, one the following holds:
$$h(pH_u) \leq C_3(T^{-1/2}(h(p)+1) + T ) \quad \text{and} \quad \deg(pH_u) \leq C_4 T$$
for some constants $C_i$ (Lemma~5 of  loc.~cit.).
An application of the arithmetic \Bezout theorem yields
$$h(p) \leq C_5h(pH_u) + C_6 \deg(H_u)h(W) + C_7\deg(H_u)$$
where moreover we have a bound  $\deg(H_u) \leq C_8 T$.
Choosing $T$ large enough, controlled by $\deg(W)$ and the constants $C_i$ (i.e.\
independently of $h(W)$)
makes the contribution of $T^{-1/2}h(p)$ to the right hand side become inessential
and proves the height bound
$$ h(p) \leq C_9 (1+ h(w) \leq C_{10} (1+h(\underline{c}))$$
\par
We need more precisely Lemma~1 and Lemma~3 of loc.~cit. which construct the
$u$. Together they show that there exists $u$ with $|u| \leq T$ and
$h(p^u) \leq C_{11} T^{-1/2} h(p)$. Together with the previous estimate
this yields
$$ h(p^u) \leq C_{12} T^{-1/2} (1+h(\underline{c})),$$
where $C_{11}$ and $C_{12}$ depend only on the dimensions of the varieties in question,
not on $h(\underline{c})$. Since $p$ lies in a field of bounded degree over $F$,
choosing $T >  C_{13} (1+h(\underline{c})))^2$, with $C_{13}$ independent of $h(\underline{c})$, suffices
by Northcott's theorem to conclude that $h(p^u) =0$.
\par
We now have two cases. Either $u$ and the cross-ratio exponents
$(a_1,a_2,a_3)$ are proportional. In this case, ii) holds by  $|u| \leq T$
and the primitivity of the triple $(a_1,a_2,a_3)$. Or they are not proportional,
i.e.\ $p$ lies on a torus of codimension two. Then we can
apply \cite[Theorem~1]{Ha08} to $Z$ since the hypotheses are met by
Lemma~\ref{le:Zoa}. The conclusion of this theorem together with Northcott's
Theorem is that the second case can happen only a finite number of times.
\qed

\paragraph{A computer search for \Teichmuller curves.}

We now describe the algorithm underlying Theorem~\ref{thm:numevidence} given in the
introduction.
\par
We first claim that for given discriminant $D$ it is possible to list all the admissible triples
$(r_1,r_2,r_3)$ for all lattices $\mathcal{I}$ with coefficient ring $\mathcal{O}_\mathcal{I}$ of
discriminant $D$.  To do so, one has to first list all orders of discriminant $D$.  Cubic number
fields of discriminant up to $D$ have been tabulated by Belabas \cite{belabas}.  Given a number
field $F$ of discriminant at most $D$, enumerating all orders in $F$ of discriminant $D$ is a finite
search through all sub-$\zed$-modules $\mathcal{O}$ of the maximal order $\mathcal{O}_F$ of bounded index. To list
all $\mathcal{O}$-ideals is a finite search through all $\zed$-modules containing $\mathcal{O}$ up
to an index bound depending on $D$. Such a bound appears in the usual proofs of the finiteness of
class numbers, e.g.\ \cite[Theorem~2.6.3]{BoShaf}.  (We do not claim that this is an efficient
algorithm).  Given a lattice $\mathcal{I}$ in a cubic field, an algorithm to find all admissible
bases of $\mathcal{I}$ is described in Appendix~\ref{sec:bdcounting}.  In practice we have
restricted the search to maximal orders, since maximal orders have been tabulated and representative
elements of the ideal classes are easily computed by Pari.
\par
Fixing a cubic order $\mathcal{O}$, if there is a \Teichmuller curve in
$\mathcal{E}_\mathcal{O}\cap\pom[3](4)^{\rm hyp}$,
then it has a cusp whose limiting stable form $\omega_\infty$ is of the form \eqref{eq:35},  with
the triple $(r_i)$ in the finite list constructed above.  Normalizing $x_1=1$, equations
\eqref{eq:1stra4hyp} and \eqref{eq:2stra4hyp} reduce to a single cubic polynomial in $x_2$.  Solving
this cubic polynomial for $x_2$ (for each triple of $r_i$) and verifying that none of the solutions
satisfies the cross-ratio equation allows us to verify that there are no \Teichmuller curves in
$\mathcal{E}_\mathcal{O}\cap\pom[3](4)^{\rm hyp}$.  Applying this algorithm to the $1778$ fields of
discriminant less than $40000$ yields Theorem~\ref{thm:numevidence}.
\par

\begin{appendix} 

\section{Boundary strata in genus three: Algorithms, examples, counting} \label{sec:bdcounting}

In this appendix, we describe an algorithm for enumerating all boundary strata of a given eigenform
locus $\E^\iota$, and we some examples and counts of admissible boundary strata obtained from this algorithm.

\paragraph{Enumerating admissible $\mathcal{I}$-weighted strata from one example.}

Given a lattice $\mathcal{I}$ in a totally real cubic field, define a graph
$\mathcal{G}(\mathcal{I})$ as follows.  The vertices of $\mathcal{G}(\mathcal{I})$ are the
two-dimensional admissible $\mathcal{I}$-weighted boundary strata, up to similarity.  Two vertices
are connected by an edge if the corresponding strata have a common one-dimensional degeneration.

\begin{prop}
  \label{prop:graph_connected}
  $\mathcal{G}(\mathcal{I})$ is connected.
\end{prop}

\begin{proof}
  By Theorem~\ref{thm:boundary_suff}, the vertices of $\mathcal{G}(\mathcal{I})$ correspond to the
  two-dimensional boundary components of some cusp of some eigenform locus $\E^\iota$.
  Thus it suffices to show that the boundary in $\pobarmoduli[3]$ of each cusp of
  $\E^\iota$ is connected.

  Consider the normalization $Y_\mathcal{O}^\iota$ of $\barE^\iota$.  By normality, the canonical
  morphism $\E^\iota\to X_\mathcal{O}$ extends to a morphism
  $p\colon Y_\mathcal{O}^\iota\to\bX_\mathcal{O}$ (see \cite[Theorem~8.10]{bainbridge07}).  Since
  $\bX_\mathcal{O}$ is normal, $p^{-1}(c)$ is connected by Zariski's Main Theorem.  The image of
  $p^{-1}(c)$ in $\barE^\iota$ is then connected, as desired.
\end{proof}

It is a simple matter to enumerate all admissible $\mathcal{I}$-weighted boundary strata adjacent to a
given one: It suffices to perform all the (finitely many) possible degenerations (as defined
in Section~\ref{sec:suffg3}) of
the presently found boundary strata  and check which of them are admissible $\mathcal{I}$-weighted.
Then one tries all the possible undegenerations and so on, until this 
process adds no more admissible $\mathcal{I}$-weighted boundary strata to the known list. 
So Proposition~\ref{prop:graph_connected} allows us to enumerate all two dimensional
$\mathcal{I}$-weighted boundary strata starting from a single one.  Lower dimensional boundary
strata can be easily enumerated from the two-dimensional ones.

\paragraph{Producing {\em one} admissible $\mathcal{I}$-weighted boundary stratum. }

We now describe an algorithm which locates a single admissible $\mathcal{I}$-weighted boundary
stratum.  In practice this algorithm is fast and always succeeds, though we do not prove this.  The
algorithm of Proposition~\ref{prop:exist_admtriple} also works for lattices of the form $\langle1,
x, x^2\rangle$, but not every lattice is similar to one of this form.
\par
For an $\mathcal{I}$-weighted boundary stratum $\mathcal{S}$ let $\Cone(\mathcal{S}) \subset
\reals^3$ be the $\reals^+$-cone spanned by $\{Q(w): w\in \Weight(\mathcal{S})\}$,  considered
as a subset of $\reals^3$ via the three field embeddings of $F$. There are various possible shapes
of this cone, which we call its {\em type}.  It could be all of $\reals^3$, for short type $(A)$, it
could be a half-space $(H)$, a proper cone of dimension three strictly contained in a half-space
$(C)$, a two-dimensional subspace $(S)$, or a $2$-dimensional cone (``fan'') in a subspace $(F)$.
\par 
The idea of the algorithm is to simply start with any irreducible stratum $\mathcal{S}$ and then to
apply a sequence of degenerations and undegeneration to $\mathcal{S}$, at each stage trying to
increase, or at least not decrease, the size of $\Cone(\mathcal{S})$.

\begin{algo} \label{algo:start}
Given a lattice $\mathcal{I}$,  compute 
an  admissible $\mathcal{I}$-weighted boundary stratum $\mathcal{S}$.
\begin{compactitem}[(i)]
\item[(i)] Initialize $\mathcal{S}$ to be the irreducible boundary stratum  
with weights given by any $\zed$-basis of  $\mathcal{I}$.
\item[(ii)] While $\Cone(\mathcal{S})$ is neither of type $(A)$, $(H)$, nor $(S)$:
\begin{compactitem}[$\bullet$]
\item (Superfluous curves) If $\mathcal{S}$ has a node $n$ which lies on the boundary of two
  distinct irreducible components with $Q(\wt(n))$ in the interior of
  $\Cone(\mathcal{S})$, then  let 
  $\mathcal{S}_1$ be obtained from $S$ by undegenerating $n$.
\item (Try to degenerate) Else 
\begin{compactitem}[$\bullet$$\bullet$]
\item Loop through all degenerations $\mathcal{S}_1$ of  $\mathcal{S}$
and check if  $\mathcal{S}_1$ contains a node $n$ with  $Q(\wt(n)) \not \in \Cone(\mathcal{S})$.
\item (Got stuck) If no such degeneration was found, the algorithm is stuck.  Start again at {\it (i)}
  with a random new choice of initial basis.
\end{compactitem}
\item  Let $\mathcal{S} = \mathcal{S}_1$.
\end{compactitem}
\item [(iii)] If the type of $\Cone(\mathcal{S})$ is $(H)$, first undegenerate $\mathcal{S}$
until $\mathcal{S}$ contains only $4$ elements still spanning a half-space
and then undegenerate
the new $\mathcal{S}$ by removing the node $n$ with the
property that $Q(\wt(n))$ does not lie in the bounding hyperplane of $C$.
(The new $\mathcal{S}$  thus obtained is of type $(S)$).
\item [iv)] Return $\mathcal{S}$.
\end{compactitem}
\end{algo}
\par
As far as we know, it is possible for the algorithm to either get stuck with every choice of initial
basis, or to loop infinitely, 
 producing larger and larger
cones without ever giving a half-space or the full space.  We have never seen this happen, though
very rarely it gets stuck and must be restarted with a new initial basis.

Some counts of boundary strata obtained from this algorithm are shown in Figure~\ref{cap:countbdry}.

It would be interesting to give an algorithm in the spirit of  Algorithm~\ref{algo:start} which is
guaranteed to always find an admissible boundary stratum.

\newcounter{bdexamples}
\setcounter{bdexamples}{1}
\paragraph{Example~\arabic{bdexamples}: Discriminant $49$.}

Figure~\ref{cap:D49} presents the outcome of the preceding algorithm for the unique
ideal class of the maximal order in the field
$F = \ratls[x]/\langle x^3 + x^2 - 2x - 1 \rangle$ of discriminant $49$.
There are two two-dimensional boundary strata.  Dotted lines join each two-dimensional stratum to
its one-dimensional degenerations.
\begin{figure}[htb]
  \centering
  \includegraphics{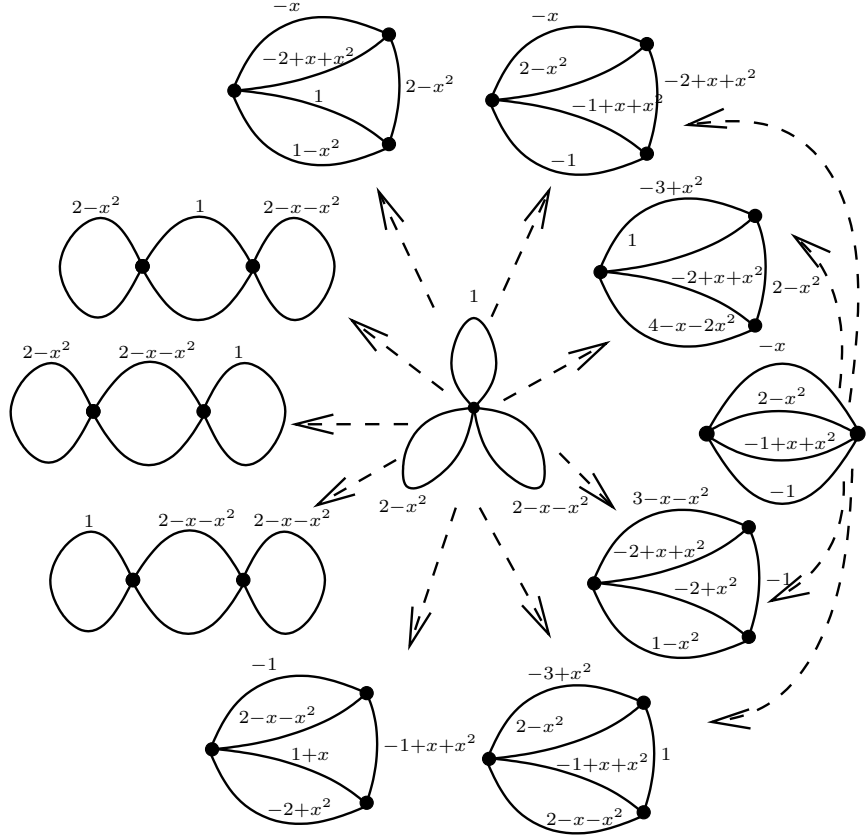}
  \caption{The boundary of the Hilbert modular threefold
    of discriminant $49$.}
\label{cap:D49}
\end{figure}
\par
\addtocounter{bdexamples}{1}

\paragraph{Example~\arabic{bdexamples}: All possible types of admissible strata
do occur.}

We give a list of examples showing that all possible types of boundary strata without separating
nodes  do occur.
\par
\begin{compactitem} 
\item If the stratum is of type $[6]$, then  $\dim(\Span) =2$ and
$D=49$ contains an example.
\item If the stratum is of type $[5] \times^3 [3]$ then $\dim(\Span) =3$.  Most cusps contain such an
  example, for example the unique cusp of the cubic field of discriminant $81$.
\item If the stratum is of type $[4] \times^4 [4]$ then 
$\dim(\Span) =2$ or $\dim(\Span) =3$. The second case frequently appears, 
e.g.\ for $D=49$. The first case rarely occurs, here is an example:
For the field $F=\ratls[x]/\langle x^3-x^2-10x+8\rangle$ with discriminant $961$, 
take the ideal $\mathcal{I} = \mathcal{O}_F$ and the weights $r_1 = 4 - x/2 - x^2/2$, 
$r_2 = 5 + x/2 - x^2/2$, $r_3=1$ and $r_4=-(r_1+r_2+r_3)$.
\item If the stratum is of type $[4] \times^2 [4]$ then $\dim(\Span) =2$.  These lie in the boundary
  of every irreducible stratum, for example in discriminant $49$.
\item All the remaining possible types of boundary strata without separating nodes have 
necessarily $\dim(\Span) =3$
and examples are easily obtained as degenerations of the preceding examples.
\end{compactitem}
\par
\addtocounter{bdexamples}{1}
\paragraph{Example~\arabic{bdexamples}: Ideal classes with no admissible bases.}

Consider one of the two fields of discriminant $3969$, namely $\ratls[x]/\langle x^3-21x-35\rangle$.
Its ideal class group is of order three.  According to a computer search, both of the ideal
classes $\mathcal{I}_1 = \langle 7, 7x, x^2-14\rangle$ and $\mathcal{I}_2 = \mathcal{I}_1^2 = \langle
7, x, x^2-3x-14\rangle$ do not admit any irreducible boundary
strata. But $\mathcal{I}_3 = \mathcal{O}_F = \langle 1, x, x^2-3x-14\rangle$ has a single
irreducible boundary stratum given by the weights $r_1 = 1$, $r_2= x+3$, $r_3 = x^2-2x-16$.

\begin{figure}[htbp]
\begin{tabular}{|c|c|c|c|c|} 
\hline
$D$    &  $h(D)$ &  regulator        &   $[6]$-components   & total $2$-dim components \\
\hline
    49    &  1    & 0.525454     &      1 & 2\\
    81    &  1    & 0.849287     &      1 & 6 \\
   148    &  1    & 1.662336     &      3 & 10\\
   169    &  1  &   1.365049     &      1 & 14\\
   229    &  1  &   2.355454    &       4 & 16 \\
   257    &  1  &   1.974593    &       2 & 19\\
   316    &  1  &   3.913458    &       7 & 26 \\
   321    &  1  &   2.569259    &       3 & 24 \\
......&&&&\\
   961    &  1  &  12.195781     &     19 & 104\\
   993    &  1  &   5.554643     &      5 & 69\\
......&&&&\\
  2597   &   3   &  4.795990    &       5+5+6 & 51+47+85 \\
......&&&&\\
  3969   &   3   &  4.201690    &       0+0+1 & 53+57+114 \\
  3969    &  3  &  12.594188   &       18+13+18 & 132 + 144+152\\
  8281    &  3  &  15.622299   &       12 +7+ 12& 259 + 224+ 266\\
  8281    &  3  &   7.949577   &        6+6+1 & 148 + 92+ 179 \\
......&&&&\\
 11884    &  1  &  72.746005   &       79 & 1008 \\
 ..... &&&&\\
20733 & 5       &  12.114993   & 12+21+8+8+12 & 250+222+138+143+281 \\
 .....&&&&\\
 22356 &     1  &  49.555997   &       31 & 967\\
 22356  &    1  &  32.935933   &       16 & 751\\
 22356   &   1  &  37.348523   &       23 & 787\\
 .....&&&&\\
28165 & 5       &   7.935079   & 4+2+2+4+6 & 174+125+121+152+337 \\
 46548    &  3  &  17.990764   &        6+6+10 & 289 + 306 + 719 \\
 46548  &    3  &  21.437334   &        9+9+16 & 324 + 337 + 741 \\

 .....&&&&\\
 84837 &     1  & 129.205864   &       73 & 2795 \\
 84872 &     3  &  60.681694   &       42+42+54 & 1121+1064+1373\\
 84889 &     1  &  77.482276   &       32 & 1913 \\
 84893 &     1  & 124.912555   &       85 & 2610 \\
 84905 &     1  &  73.229843   &       27 & 1723 \\
 84925 &     1  &  90.776953   &       37 & 2112 \\
 84945 &     1  &  82.760047   &       50 & 1879 \\
......&&&&\\
161249 &     1  &  65.942246   &     16 & 1882\\
161753 &     2  &  26.530548   &       10+10 & 641+ 1084\\
......&&&&\\
438492  &    1  &  504.944683   &      228 & 12265 \\
\hline
\end{tabular}
\caption{The number of  boundary components for given discriminant $D$}
 \label{cap:countbdry}
\end{figure}

\section{Components of the eigenform locus} \label{sec:HMvsRM}

In this section we show that, in contrast to the quadratic case, that the $\E^\iota\isom
X_\mathcal{O}$ is not necessarily connected for cubic orders $\mathcal{O}$.

Recall from \S\ref{sec:orders-number-fields} that the irreducible components of $X_\mathcal{O}$ correspond
bijectively to isomorphism classes of proper, rank-two, symplectic $\mathcal{O}$-modules.  One
example of such a module is $\mathcal{O}\oplus\mathcal{O}^\vee$.
We will show that there is such a module $M$ 
such that for {\em no} submodule $\mathcal{I}$ of $M$ the sequence
\begin{equation*}
  0 \to \mathcal{I} \to M \to \mathcal{I}^\vee\to 0,
\end{equation*}
is split, thus $M$ is not isomorphic to $\mathcal{O}\oplus\mathcal{O}^\vee$.
\par
We remark that such examples cannot exist for the ring of integer $\Ord_F$
since Dedekind domains are projective and that they can neither exist
for $[F:\ratls]=2$ e.g.\ by structure theorems for rings all whose ideals
are generated by two elements \cite{Bass62}.
\par
The calculations will be easier to do in the local situation, and if the above
sequence was split, it would be also split locally. Choose a 
totally real cubic number field $F$ and a prime $p$ different 
from $2$ and from $3$ such that the residue field $k$ is isomorphic 
to $\FF{p^3}$. Let $K$ be the completion of  $F$ at the prime $p$. Let $R_K$ be the ring of integers 
in $K$ and let $R$ be the preimage of the prime field under the
surjection $R_K \to k$. We will exhibit  an $R$-module
$M$ with the claimed properties. From there it is obvious
how to construct a module over $\Ord$, the preimage
of the prime field under $\Ord_F \to k$, that also has
the claimed properties.
\par
For simplicity we suppose moreover that $R_K$ is monogenetic, 
i.e.\ that $R_K = \zed_p[\theta]/f$ for some cubic polynomial $f$.
\par
\begin{lemma}
We have
$$ R_K = R_K^\vee \subset_{p^2} R^\vee \subset_{p} p^{-1} R_K,$$
where the subscripts denote the index. In fact, 
$$ R^\vee = \{r \in p^{-1} R_K| \trace(pr) \equiv 0 \mod (p)\}.$$
More precisely, there exists a $\zed_p$-basis $\{1,x,y\}$ of
$R_K$ which is orthogonal with respect to the trace pairing.
Then 
$$R = \left\langle 1,px,py \right\rangle_\zed, \quad 
R_K^\vee = \left\langle 1, \frac xp, \frac yp 
\right\rangle_\zed. $$
\end{lemma}
\par
\begin{proof}
The ring $R_K^\vee$ is generated by $\theta^i/f'(\theta)$ for
$i=0,1,2$. Since $f'(\theta)$ is a unit in $R_K$ be the hypothesis
on the residue field, we obtain  $R_K = R_K^\vee$.
\par
Suppose $s \in p^{-1}R_K$. We use that by definition any  $y \in R$
is congruent mod $(p)$ to  $z \in \zed$. Thus since
$$ \trace(rs) \equiv z\trace(r) \mod (p)$$
we conclude that $r \in R^\vee$ if and only if $\trace(pr)=0$ (using
$p \neq 3$).   
\end{proof}
\par
\begin{lemma} \label{le:finRmod}
All of the quotients $R_K/pR_K$, $R^\vee/pR^\vee$ and $R/pR$ are three-dimensional as
$\FF{p}$ vector spaces but different as
$R$-modules:
\begin{itemize}
\item $R_K/pR_K$ splits into a direct sum of $\langle 1 \rangle$ 
and $\langle x,y \rangle$, orthogonal with respect to the trace pairing.
\item $R/pR$ has the irreducible $R$-submodule $\langle px,py \rangle$ 
and the corresponding sequence is not split.
\item $R^\vee/pR^\vee$, as the dual of the preceding module, has
the quotient $R$-module $\langle x/p,y/p \rangle$, and the 
corresponding sequence is not split.
\end{itemize} 
\end{lemma}
\par
\begin{proof}
The structure of  $R_K/pR_K$ is obvious. Suppose 
$1+p(ax+by)$ generates an $R$-submodule of $R/pR$ of dimension
one over $\FF{p}$. Multiplying
by $px$ we see that this submodule contains also $px$,  We thus obtain a contradiction.
\end{proof}
\par
\begin{lemma} \label{le:extgroups}
We can calculate $\Ext$-groups as follows:
\begin{equation}
\begin{aligned}
&\Ext_R^1(R^\vee,R) = \Hom_R(R^\vee, R/pR)/\Hom_R(R^\vee,R) \cong \FF{p}\\
&\Ext_R^1(R^\vee,R_K) = \Hom_R(R^\vee, R_K/pR_K)/\Hom_R(R^\vee,R_K)  \cong \FF{p}\\
&\Ext_R^1(R^\vee,R^\vee) = 0 \\
\end{aligned}
\end{equation}
\end{lemma}
\par
\begin{proof}

The short exact sequence of multiplication by $p$ gives a long exact
sequence
$$\Hom_R(R^\vee,M) \to   \Hom_R(R^\vee, M/pM) \to \Ext^1(R^\vee,M) \to \Ext^1(R^\vee,M),$$
where the last map is induced by multiplication by $p$. Under the second map
the image of $\Hom_R(R^\vee, M/pM)$  is  $p$-torsion and thus  
$\Ext^1(R^\vee,M)$ is $p$-torsion as well.
\par
We first deal with the case $M=R$. Obviously $p^2 R_K$ is contained 
in $\Hom(R^\vee,R)$ and we claim they are equal. If such a homomorphism
was given by multiplication with an element $s \not\in p^2 R_K$, take 
$t = x/p \in R$ where $x$ is as above. Then $ts \not\in pR_K$ and
its reduction is not in the prime field, since the reductions of $\{1,x,y\}$ are linearly
independent over $\FF{p}$. This contradiction proves the claim. 
\par
First we claim that 
$$\Hom_R(R^\vee, R/pR) \cong \Hom_{\FF{p}}( k/\FF{p},\Ker(\trace)),$$
where we consider $\Ker(\trace) \subset k$. A homomorphism 
from $R^\vee$ to $R/pR$ factors through $R^\vee/pR^\vee$. 
By Lemma~\ref{le:finRmod} there are no isomorphisms between them, 
in fact the classification of quotient resp.\ submodules in this 
lemma shows more precisely that such a homomorphism
factors through an element in $\Hom_{\FF{p}}( k/\FF{p},\Ker(\trace))$.
Both on  the  quotient module $\langle x/p,y/p \rangle \cong k/\FF{p}$
and on the submodule $\langle px,py \rangle \cong \Ker(\trace)$, 
the ring $R$ acts through its quotient
$\FF{p}$ so that indeed every  ${\FF{p}}$-homomorphism is 
and $R$-homomorphism. Multiplication by $p^2R$ defines a subspace
isomorphic to $k$ inside  $\Hom_{\FF{p}}( k/\FF{p},\Ker(\trace))$.
This concludes the second isomorphism of the second claim.
\par
Second we look at the case $M=R_K$. Now  $\Hom(R^\vee,R_K) ()\cong pR$
and elements in $\Hom_R(R^\vee, R_K/pR_K)$ factor through
$\Hom_{\FF{p}}( k/\FF{p},\Ker(\trace))$ using the submodule structure
of the finite $R$-modules determined in Lemma~\ref{le:finRmod}.
\par
The last statement follows by the same reasoning. 
\end{proof}
\par
\begin{prop}
Let $0\to R\to M \to R^\vee\to 0$ be a symplectic extension corresponding to
a non-trivial element in $\Ext^1_R(R^\vee,R)$. Then $M$ is a proper
$R$-module. Moreover, $M$ has a unimodular symplectic structure
and the $R$-action is by self-adjoint endomorphisms. $M$ is
not a direct sum of two $R$-modules of rank one.
\end{prop}
\par
\begin{proof} The trace pairing $R$ and $R^\vee$ induces a
symplectic and unimodular pairing on $M$. The $R$-submodule
$R$ of $M$ is isotropic for this alternating pairing. Thus
if $M$ is an $\tilde{R}$-module for some ring $\tilde{R}$ 
containing $M$ and acting by self-adjoint endomorphisms, 
then $R$ is also an $\tilde{R}$-module. This implies
$\tilde{R}=R$, i.e.\ that $M$ is a proper $R$-module. 
\par
It remains to show that $M$ is not a direct sum. If it is, 
then $M \cong \fra \oplus \fra^\vee$. If we apply 
$\Hom(R^\vee,\cdot)$ to the extension defining $M$, we obtain 
an exact sequence
$$ \Hom_R(R^\vee,R^\vee) \to \Ext^1_R(R^\vee,R) \to \Ext^1_R(R^\vee,M) \to 
\Ext^1_R(R^\vee,R^\vee).$$
The first map is a non-zero map $R \to \FF{p}$ by the fact that
$M$ was constructed as a non-trivial extension. The hypothesis on $M$
implies that  
$$\Ext^1_R(R^\vee,M) = \Ext^1_R(R^\vee,\fra) \oplus  \Ext^1_R(R^\vee,\fra^\vee) 
\cong \FF{p}.$$
Since  $\Ext^1_R(R^\vee,R^\vee)=0$ it remains to show that
at least one of the two groups  $\Ext^1_R(R^\vee,\fra)$
and $\Ext^1_R(R^\vee,\fra^\vee)$ is non-zero. The $\Ext$-groups
don't change if we replace $\fra$ by $p\fra$. Under this equivalence 
the pair $(\fra,\fra^\vee)$ is either $(R,R^\vee)$, $(R^\vee,R)$
or $(R_K, R_K)$. Thus the claim follows from Lemma~\ref{le:extgroups}.
\end{proof}

\end{appendix}

\bibliography{my}

\end{document}